\documentclass[11pt,reqno]{article}

\usepackage{amsthm}
\usepackage[english]{babel}
\usepackage{amsmath}
\usepackage{amssymb}
\usepackage{graphicx}
\usepackage{color}
\usepackage{bm}
\usepackage{mathtools}
\usepackage{subfigure}
\usepackage{cite}
\usepackage{hyperref}

\newcommand{\nm}{\noalign{\smallskip}}
\newcommand{\ds}{\displaystyle}
\usepackage{pgf}
\usepackage{pgffor}
\usepgflibrary{plothandlers}
\usepackage{pgfplots}
\pgfplotsset{compat=newest}
 \pgfplotsset{width=15cm}
\pgfplotsset{plot coordinates/math parser=false}
\newlength\figureheight
\newlength\figurewidth

\newtheorem{theo}{Theorem}[section]

\newtheorem{prop}[theo]{Proposition}

\newtheorem{rem}[theo]{Remark}

\newcommand{\R}{\mathbb{R}}

\newcommand{\psibf}{{\psi}}

\newcommand{\phibf}{{\phi}}

\def\nm{\noalign{\medskip}}

\newcommand{\nubf}{{\nu}}

\numberwithin{equation}{section} \numberwithin{figure}{section}

\newcommand{\dr}{\partial}

\newcommand{\g}{\nabla}

\newcommand{\bE}{\mathbf{E}}
\newcommand{\bH}{\mathbf{H}}

\newcommand{\RR}{\mathbb{R}}
\def\nm{\noalign{\medskip}}

\newtheorem{theorem}{Theorem}[section]
\newtheorem{lem}{Lemma}[section]

\newcommand{\beq}{\begin{equation}}
\newcommand{\eeq}{\end{equation}}

\DeclareMathAlphabet{\itbf}{OML}{cmm}{b}{it}

\begin{document}
\title{Optimal estimate of electromagnetic field concentration between two nearly-touching inclusions in the quasi-static regime}
\author{Youjun Deng\thanks{\footnotesize School of Mathematics and Statistics, Central South University, Changsha, Hunan, China.
(youjundeng@csu.edu.cn, dengyijun\_001@163.com).}
\and Hongyu Liu\thanks{\footnotesize Department of Mathematics, City University of Hong Kong, Hong Kong SAR, China.
(hongyu.liuip@gmail.com).} \and Liyan Zhu\thanks{\footnotesize School of Mathematics and Statistics, Central South University, Changsha, Hunan, China.(math\_zly@csu.edu.cn, math\_zly@163.com).}}

\maketitle

\begin{abstract}

We investigate the electromagnetic field concentration between two nearly-touching inclusions that possess high-contrast electric permittivities in the quasi-static regime. By using layer potential techniques and asymptotic analysis in the low-frequency regime, we derive low-frequency expansions that provide integral representations for the solutions of the Maxwell equations.
For the leading-order term $\bE_0$ of the asymptotic expansion of the electric field, we prove that it has the blow up order of $\epsilon^{-1} |\ln \epsilon|^{-1}$ within the radial geometry, where $\epsilon$ signifies the asymptotic distance between the inclusions. By delicate analysis of the integral operators involved, we further prove the boundedness of the first-order term $\bE_1$. We also conduct extensive numerical experiments which not only corroborate the theoretical findings but also provide more discoveries on the field concentration in the general geometric setup. Our study provides the first treatment in the literature on field concentration between nearly-touching material inclusions for the full Maxwell system.  

\medskip

\noindent { {\bf Mathematics subject classification
(MSC2000):} 78A40, 35C20, 78A48, 35R30}


\end{abstract}
\maketitle

\section{Introduction}

In situations when two or more nearly-touching inclusions display high-contrast material properties in comparison to their background medium, there may occur strong field concentrations in the narrow areas in between the nearly-touching domains, which is a central issue in the theory of composite materials. In the engineering design, such substantial field concentrations could result in the failure of composite material structures or other severe practical outcomes  \cite{BA}. On the other hand, electric field eruptions can be employed to facilitate the realization of sub-wavelength imaging and delicate spectroscopy \cite{SYHA2018}.

Consequently, the precise quantification of physical field concentrations is of significant importance. Extensive literature has delved into it, primarily focusing on the static situations under various physical settings,
see \cite{bonnetier2000elliptic,DongxuEpllipticSIMA,LN03,li2000gradient} for related studies for general elliptic systems; \cite{ammari2013,BLL17,baolili,DongliARMA2019,LHZZSIAM2020,KLY15,KY19,yun2009optimal} for the Lam\'e system; \cite{AKKY21, DongXuCNS, LiXuSIAM} for the Stokes flow problem; \cite{ammari2007optimal,ammari2005gradient,bao2009gradient,bao2010gradient,ChenLiXu2021,kang2014characterization,Lek10,lim2009blow,yun2007estimates} for the conductivity problem.
Here, we mention a few results for conductivity problem, the gradient of the voltage potential , i.e., the solution of the conductivity equation, represents the electric field, which is related to our research in this article.
When the conductivities of inclusions are away from $0$ and $\infty$, it is numerically showed by Bab$\breve{\text{u}}$ska et al. in \cite{BA} that the gradient of solutions remains bounded independent of $\epsilon$ (the distance between two inclusions). Further, it was proved by Li and Nirenberg in \cite{li2000gradient}, Bonnetier and Vogelius in \cite{bonnetier2000elliptic} that the gradient field is uniformly bounded  if the material parameter is bounded both below and above \cite{bonnetier2000elliptic,li2000gradient}.
When  the conductivity equals to $\infty$ (inclusions are perfect conductors) or $0$ (insulators), it was shown in \cite{Buca1984,Kell1993,Mar1996} that the gradient field generally becomes unbounded as $\epsilon\to 0$.
For the perfect conductor problem, i.e., when the conductivity of the inclusion is $\infty$, it has been proven that the general blowup rate of the gradient field is $\epsilon^{-1/2}$ in two dimensions \cite{ammari2005gradient,bao2009gradient,yun2007estimates},  the order is $|\epsilon \ln \epsilon|^{-1}$ in three dimensions \cite{bao2009gradient,kang2014characterization}, and order $\epsilon^{-1}$ in dimensions greater than three  \cite{bao2009gradient,bao2010gradient}.
For more works on the perfect conductivity problem and closely related
works, see e.g.\cite{lim2009blow} and the references therein.
For the insulator conductivity problem, i.e., when the conductivity of the inclusion is $0$, in two dimensions, it is dual to the perfectly conducting case (by means of the harmonic conjugation), and the optimal blow-up rate is also $\epsilon^{-1/2}$, see\cite{ammari2007optimal,ammari2005gradient}, furthermore, Yun extended in \cite{yun2007estimates,yun2009optimal} the results allowing two inclusions to be any bounded strictly
convex smooth domains. For the higher dimensional case, Li and Yang in \cite{li2022gradient} proved for some generic constant $\beta\in\mathbb{R}_+$, the blowup rate decreases to $\epsilon^{-1/2+\beta}$. 
Yun in \cite{Yun2016} proved the optimal blow-up
rate on the $\epsilon$ -- segment connecting two spherical inclusions in three dimensions is $\epsilon ^{\frac{\sqrt{2}-2 }{2} }$, Dong, Li and Yang in \cite{dong2022optimal} proved that when the inclusions are balls, the optimal value of  $\beta$  is  $\left[-(n-1)+\sqrt{(n-1)^{2}+4(n-2)}\right] / 4 \in(0,1 / 2)$  in dimensions  $n \geq 3$.  When the conductivity of two inclusions with different signs, Dong and Yang in \cite{Dongyang} proved that the gradient and higher order derivatives are bounded
independent of $\epsilon$ in two dimensions, more special, when one inclusion is an insulated and the other one is a perfect, the derivatives of any order is bounded for any dimensions $n \geq 2$ and general smooth strictly convex inclusions.

As mentioned above, most of the existing literature focuses on the static case, that is, the frequency is zero. In the recent advances of composite materials, it is important to consider the interactions between quasistatic wavefields and high contrast inclusions closely spaced, see  \cite{AZ, LLZ,MBDL} and the references cited therein. Therefore, it is natural to consider field concentration under this new circumstance.
Recently, Deng et al. \cite{deng2022gradient,deng2023gradient,DKLZ23} established gradient estimates for two nearly touching high contrast inclusions related to the Helmholtz system under the quasi-static case. However, no relevant study targeting the Maxwell equations exists so far.


In this paper, we present the blow-up estimate for the change in the electric field in the Maxwell system according to the distance between a pair of nearly touching spherical conductors. Our analysis relies on the use of vector layer potential theory to represent the solution, carrying out asymptotic analysis through low-frequency expansion, and estimating the leading term and linear term of the expansion. To the best of our knowledge, our work is the first result in the estimation of the Maxwell system.


The rest of the paper is organized as follows.
In Section 2, we introduce the mathematical setup of the research. The main theoretical results of this paper are presented in Section 3. In Section 4, we review some useful results of the layer potential techniques of the Maxwell equations. Section 5 derives the asymptotic expansions of operators, potentials, and solutions. In Section 6, we prove the main theorem by deriving the optimal estimation of the leading terms of the electric field and obtaining the boundedness of the first-order terms. Section 7 is devoted to the numerical results.

\section{Mathematical setup}

In this Section, we present the mathematical setup of our study in Sections 2--6, where we mainly focus on the radial case. In Section 7 we shall consider the general geometric case. 

We assume that  $ B_{1} $ and  $B_{2}$ are two conducting spheres embedded in $\mathbb{R}^{3}$, as shown in Figure \ref{model},  which have the the radius $r \in\mathbb{R}_+$, the centers $\mathbf{c}_1$, $\mathbf{c}_2 \in\mathbb{R}^3$, respectively, and separated by a distance $\epsilon$,
\begin{equation*}
\mathbf{c}_1=(-r-\frac{\epsilon}{2}, 0,0)\quad \mbox{and} \quad \mathbf{c}_2=(r+\frac{\epsilon}{2}, 0,0).
\end{equation*}
Here, $\epsilon:=\mathrm{dist}(B_1, B_2)$. Throughout we set $\epsilon\ll 1$. Denote $D=B_1 \cup B_2$. Physically,  $B_j, j=1,2$  denote bounded inclusions that is embedded inside a homogeneous medium in  $\mathbb{R}^3 \backslash \bar{D}$. The electromagnetic (EM) medium is characterised by the electric permittivity $ \varepsilon$ , magnetic permeability  $\mu $ and electric conductivity  $\sigma$.
\begin{figure}[!htpb]
	\centering
	{\includegraphics[width=1\textwidth]{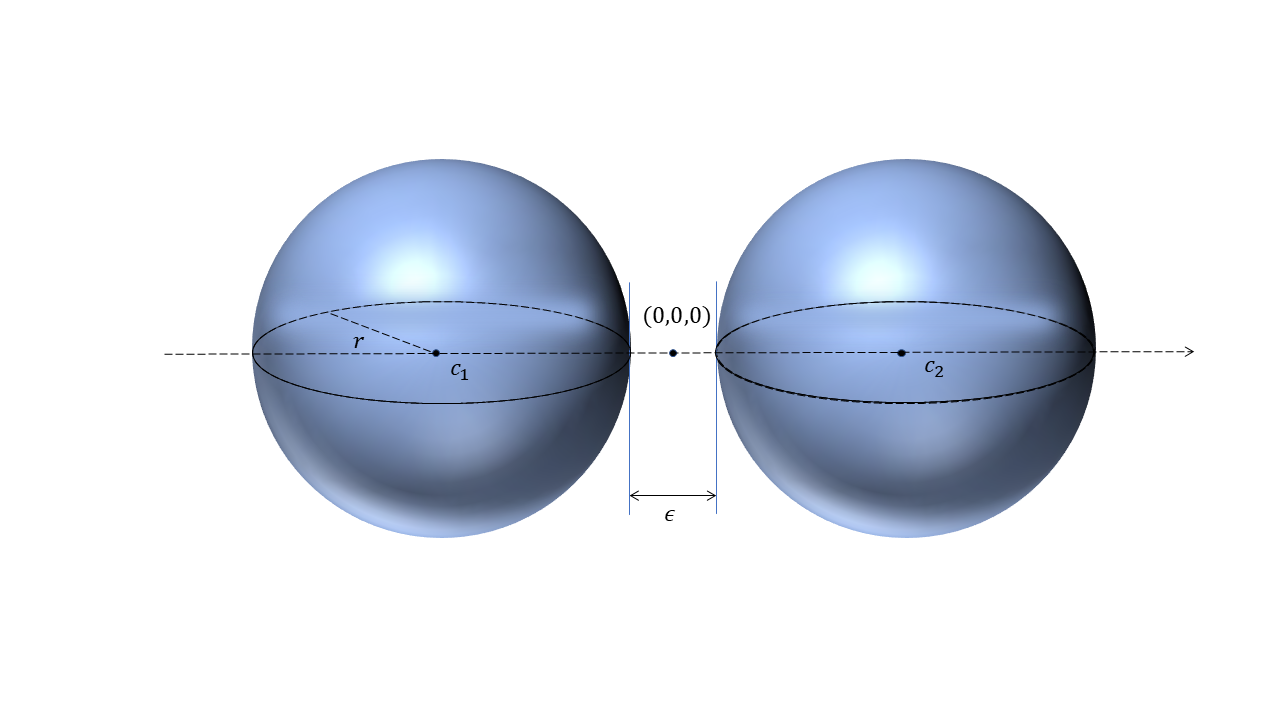}}
	\caption{\label{model} Model: Two close-to-touching spheres embedded in the background. }
\end{figure}
Throughout, we assume that $\varepsilon=\mu=1$ in $\mathbb{R}^3\setminus\overline{D}$, $\sigma=0 $ in $ \mathbb{R}^{3},$  $\varepsilon=\varepsilon_c$ and $\mu=1$ in $D$, where $\varepsilon_c\in\mathbb{R}_+$.  Define $ k:=\omega \sqrt{\varepsilon \mu}$  to be the wavenumber with respect to a frequency $ \omega \in \mathbb{R}_{+}$.

Assume the incident field $(\bE^{i}, \bH^{i})$ is given by the electromagnetic plane waves as follows
\begin{equation}\label{in}
\mathbf{E}^{{i}}(\mathbf{x}, \omega, \mathbf{d}, \mathbf{q}):=\frac{1}{\sqrt{\varepsilon}} \mathbf{p} e^{\mathrm{i} k \mathbf{x} \cdot \mathbf{d}}, \quad \mathbf{H}^{{i}}(\mathbf{x}, \omega, \mathbf{d}, \mathbf{q}):=\frac{1}{\sqrt{\mu}}(\mathbf{d} \times \mathbf{p}) e^{\mathrm{i} k \mathbf{x} \cdot \mathbf{d}},
\end{equation}
where $\mathrm{i}:=\sqrt{-1}$, $\mathbf{d}=\left ( d_1,d_2,d_3 \right )  \in \mathbb{S}^{2}:=\left\{\mathbf{x} \in \mathbb{R}^{3}:|\mathbf{x}|=1\right\} $ signifies the incident direction and  $\mathbf{p}=\left ( p_1,p_2,p_3 \right )  \in \mathbb{R}^{3} $  denotes a polarisation vector. It holds that  $\mathbf{p} \perp \mathbf{d}$. Let  $\mathbf{E}(\mathbf{x}, \omega, \mathbf{d}, \mathbf{q}) $ and $ \mathbf{H}(\mathbf{x}, \omega, \mathbf{d}, \mathbf{q}) $, respectively, denote the total electric and magnetic fields and they satisfy the time-harmonic Maxwell equations:
\begin{equation}
\label{eq:maxwell} \left\{
\begin{array}{ll}
\nabla \times \bE = \mathrm{i} \omega \mu \bH &  \mbox{in} \quad \mathbb{R}^3, \\
\nabla  \times \bH= - \mathrm{i} \omega \varepsilon \bE & \mbox{in} \quad \R^3, \\
{[}{\nu} \times \bE]= [{ \nu} \times \bH] = 0 & \mbox{on} \quad \dr {D},
\end{array}
\right.
\end{equation}
  subject to the Silver-M\"{u}ller radiation condition:
$$\lim_{|\mathbf{x}|\rightarrow\infty} |\mathbf{x}|(\sqrt{\mu} (\bH- \bH^{i}) \times\hat{\mathbf{x}}-\sqrt{\varepsilon} (\bE-\bE^{i}))=0,$$
where $\hat{\mathbf{x}} = \mathbf{x}/|\mathbf{x}|$. Here,  $\nu$ signifies the exterior unit normal vector to the boundary of the concerned domain  $D$,  $[{\nubf} \times{\bE}]$ and $[{\nubf} \times{\bH}]$ denote the jump of ${\nubf} \times{\bE}$
and ${\nubf} \times{\bH}$ along $\dr D$, namely,
\begin{equation*}
[\nubf\times{\bE}]=(\nubf\times \bE)^+ -(\nubf \times \bE)^-,\quad [\nubf\times{\bH}]=(\nubf\times \bH)^+ -(\nubf\times \bH)^-.
\end{equation*}

In this paper, we focus on examining the behavior of electric $\bE$ in the narrow region between the inclusions. Specifically, we consider the quasi-static regime. In the following discussions, we assume that $\omega\ll 1$ 
and $\varepsilon_c=O(\frac{1}{\omega})=\frac{\tilde{C} }{\omega}, \tilde{C} $ is a constant independent of $\varepsilon_c$ and $\omega$. To this end, we suppose $B_1$ and $B_2$ are inclosed by a bounded domain $\Omega$ and denote by $\| \bE\|_{L^{\infty}(\Omega\setminus\overline{D})}$ the $L^{\infty}$ norm of the electric field outside the two spheres.

\section{Main result}
Here is our main theoretical result in this paper.
\begin{theorem}\label{th:main01}
Suppose $\omega \ll 1$, and $\varepsilon_c=O(\frac{1}{\omega})=\frac{\tilde{C} }{\omega}$. Let $\bE\in H_{loc}\left ( curl;\mathbb{R}^3 \right ) $ be a solution of \eqref{eq:maxwell}, for sufficiently small $\epsilon>0$,  there exist positive constants $C_1$, $C_2$ depending only on $r$ and $\Omega$ such that
\begin{equation}\label{main}
\begin{aligned}
\|\bE\|_{L^\infty(\Omega\backslash \overline{B_1\cup B_2})}  \sim C_1 r|p_1| \frac{1}{\epsilon |\ln \epsilon|}+C_2\omega+O(\omega^2).
\end{aligned}
\end{equation}
\end{theorem}
\begin{rem}
Theorem \ref{main} characterizes the leading order blow up rate is $\frac{1}{\epsilon |\ln \epsilon|}$, which is in accordance with the electric filed gradient estimate for three dimensional Laplace system. We want to also mention that $C_2$ is a constant which does not depend on $\epsilon$, i.e., the first order term does not blow up.
\end{rem}

\section{Boundary integral operators}\label{sec2}
We start by recalling  some known properties of the boundary integral operator, which will be used in Section \ref{sec22} to derive the asymptotic expansion of the electric field.

\subsection{Definitions}

In what follows, we let  $\nabla_{\dr D}\cdot$ denote the surface divergence. Denote by $L_T^2(\dr D):=\{\phibf\in {L^2(\dr D)}^3, \nubf\cdot \phibf=0\}$. Let $H^s(\partial D)$ be the usual Sobolev space of order $s$ on $\partial D$. We often use the spaces
\begin{align*}
\mathrm{TH}({\rm div}, \dr D):&=\Bigr\{ {\phi} \in L_T^2(\partial D):
\nabla_{\partial D}\cdot {\phi} \in L^2(\partial D) \Bigr\},\\
\mathrm{TH}({\rm curl}, \dr D):&=\Bigr\{ {\phi} \in L_T^2(\partial D):
\nabla_{\partial D}\cdot ({\phi}\times {\nu}) \in L^2(\partial D) \Bigr\},
\end{align*}
equipped with the norms
\begin{align*}
&\|{\phi}\|_{\mathrm{TH}({\rm div}, \dr D)}=\|{\phi}\|_{L^2(\dr D)}+\|\nabla_{\dr D}\cdot {\phi}\|_{L^2(\dr D)}, \\
&\|{\phi}\|_{\mathrm{TH}({\rm curl}, \dr D)}=\|{\phi}\|_{L^2(\dr D)}+\|\nabla_{\dr D}\cdot({\phi}\times \nu)\|_{L^2(\dr D)}.
\end{align*}

For $k>0$, the fundamental outgoing solution $\Gamma^k$ to the Helmholtz
operator $(\Delta+k^2)$ in $\mathbb{R}^3$ is given by
\begin{equation}\label{Gk} \ds\Gamma^k
({\mathbf{x}}) = -\frac{e^{\mathrm{i}k|{\mathbf{x}}|}}{4 \pi |{\mathbf{x}}|},\quad \mathbf{x}\ne0.
\end{equation}
For a density $\Phi \in \mathrm{TH}({\rm div}, \dr D)$, the
vectorial single layer potential was introduced in (\ref{Gk}) by
\begin{equation*}
\ds\mathcal{A}_D^{k}[\Phi]({\mathbf{x}}) := \int_{\dr D} \Gamma^k({\mathbf{x}}-{\mathbf{y}})
\Phi({\mathbf{y}}) d \sigma({\mathbf{y}}), \quad {\mathbf{x}} \in \mathbb{R}^3.
\end{equation*}
For a scalar density $\varphi \in L^2(\dr D)$, the single layer
potential and double layer potential are defined by
\begin{equation*}
{S}_D^{k}[\varphi]({\mathbf{x}}) := \int_{\dr D} \Gamma^k({\mathbf{x}}-{\mathbf{y}}) \varphi({\mathbf{y}}) d \sigma({\mathbf{y}}), \quad {\mathbf{x}} \in \mathbb{R}^3.
\end{equation*}
\begin{equation*}
D_{D}^{k}[\varphi]({\mathbf{x}})=\int_{\partial D} \frac{\partial \Gamma^{k}({\mathbf{x}}-{\mathbf{y}})}{\partial \nu_{{\mathbf{y}}}} \varphi({\mathbf{y}}) \mathrm{d} \sigma({{\mathbf{y}}}), \quad {\mathbf{x}} \in \mathbb{R}^{3}.
\end{equation*}
Here  $\nu_{{\mathbf{x}}} \in \mathbb{S}^{2} $ signifies the exterior unit normal vector to the boundary of the concerned domain  $D$  at  ${\mathbf{x}}$.
The Neumann-Poincar\'e operators $ K_{D}^{k},\left(K_{D}^{k}\right)^{*}$  are bounded from  $L^{2}(\partial B) $ into $ L^{2}(\partial B) $ and given by
\begin{equation*}
\begin{aligned}
K_{D}^{k}[\varphi](\mathbf{x}) & =\int_{\partial D} \frac{\partial \Gamma^{k}({\mathbf{x}}-\mathbf{y})}{\partial \nu_{{\mathbf{y}}}} \psi({\mathbf{y}}) \mathrm{d} \sigma({{\mathbf{y}}}), {\mathbf{x}} \in \partial D, \\
\left(K_{D}^{k}\right)^{*}[\varphi]({\mathbf{x}}) & =\int_{\partial D} \frac{\partial \Gamma^{k}({\mathbf{x}}-{\mathbf{y}})}{\partial \nu_{{\mathbf{x}}}} \psi({\mathbf{y}}) \mathrm{d} \sigma({{\mathbf{y}}}), {\mathbf{x}} \in \partial D .
\end{aligned}
\end{equation*}

For a density $\Phi \in \mathrm{TH}({\rm div}, \dr D)$,  the  boundary operators $ \mathcal{M}_D^k$ and $\mathcal{L}_D^k$ are defined as:
\begin{equation*}
\begin{aligned} \mathcal{M}_D^k:
\mathrm{L}_T^2 (\partial D)  & \longrightarrow \mathrm{L}_T^2 (\partial D)  \\
\Phi & \longmapsto \mathcal{M}^k_D[\Phi]= \nubf({\mathbf{x}})  \times \g \times \int_{\dr D} \Gamma^k({\mathbf{x}},{\mathbf{y}})  \Phi({\mathbf{y}}) d\sigma({\mathbf{y}}),
\end{aligned}
\end{equation*}
\begin{equation*}
\begin{aligned} \mathcal{L}_D^k:
\mathrm{TH}({\rm div}, \dr D) & \longrightarrow \mathrm{TH}({\rm div}, \dr D) \\
\Phi & \longmapsto \mathcal{L}^k_D[\Phi]= \nubf({\mathbf{x}})  \times \left( k^2\mathcal{A}_D^k[\Phi]({\mathbf{x}}) + \g{S}_D^k[\g_{\dr D} \cdot \Phi]({\mathbf{x}})\right).
\end{aligned}
\end{equation*}

\subsection{Boundary integral  identities}

Let us first recall the following jump formula, we refer to \cite{ADM16,kangbook} and references therein for details.
\begin{prop}\label{propjumpS}
For a scalar density $\varphi \in L^2(\dr D)$, the single-layer potential is continuous on $\mathbb{R}^3$, and the normal derivative satisfies the following jump condition :
\begin{equation*}
\frac{\dr}{\dr \nu} \left( {S}_{D}^k [\varphi] \right)^{\pm} = (\pm \frac{1}{2} I + ({K}^k_{  D})^* )[\varphi]\, \quad \mbox{ on } \dr D,
\end{equation*}
where ${I}$ represents the identity operator on  $\varphi \in L^2(\dr D)$, the superscripts $\pm$ indicate the limits from outside and inside $D$ respectively.
\end{prop}
\begin{prop}\label{propjumpM}
	Let $ \phibf \in L^2_T(\dr D)$. Then $\mathcal{A}_D^k[\phibf]$ is continuous on $\mathbb{R}^3$ and its curl satisfies the following jump formula:
	\begin{equation*}
	\left( \nubf \times \g \times \mathcal{A}_D^k[\phibf]\right)^\pm = \mp \frac{\phibf}{2} + \mathcal{M}_D^k[\phibf] \quad \mbox{ on } \dr D,
	\end{equation*}
	where \begin{equation*}
	\forall \mathbf{x}\in \dr D, \quad \left( \nubf(\mathbf{x}) \times \g \times \mathcal{A}_D^k[\phibf]\right)^\pm (\mathbf{x})= \lim_{t\rightarrow 0^+} \nubf(\mathbf{x}) \times \g \times \mathcal{A}_D^k[\phibf] (\mathbf{x}\pm t \nubf(\mathbf{x})).
	\end{equation*}
\end{prop}

\section{Low-frequency asymptotic expansion} \label{sec22}
In this section, we shall derive the asymptotic expansion of the operator introduced in Section \ref{sec2} in terms of low frequency. Furthermore, we will use the layer potentials defined in Section \ref{sec2} to obtain representations and asymptotic expansions of the solutions.
\subsection{Asymptotics for the operators}
For $k>0$, as $k \rightarrow 0$, the fundamental outgoing solution $\Gamma^k$
has the following asymptotic expansion
\begin{equation*}
\begin{aligned}
\ds\Gamma^k(\mathbf{x}) &= -\frac{e^{\mathrm{i}k|\mathbf{x}|}}{4 \pi |\mathbf{x}|}=-\frac{1}{4 \pi |\mathbf{x}|}-\frac{\mathrm{i}}{4 \pi }k-\frac{|\mathbf{x}|}{8 \pi }k^2+O(k^3)\\
\nm
&\ds:=\Gamma_0(\mathbf{x})+\Gamma_1(\mathbf{x})k+\Gamma_2(\mathbf{x})k^2+O(k^3).
\end{aligned}
\end{equation*}
Based on this, we have the following expansions for later usage.
Let $ \phibf \in L^2_T(\dr D)$, as $k\to 0$, we have
	\begin{equation*}
	\begin{aligned}
\ds\mathcal{A}_D^{k}[\phibf](\mathbf{x})
	&=\ds \int_{\dr D} \Gamma^k(\mathbf{x}-\mathbf{y})
	\phibf(\mathbf{y}) d \sigma(\mathbf{y}) \\
	\nm
	&=\ds \int_{\dr D} \Gamma_0(\mathbf{x}-\mathbf{y})
	\phibf(\mathbf{y}) d \sigma(\mathbf{y})+\int_{\dr D} \Gamma_1(\mathbf{x}-\mathbf{y})
	\phibf(\mathbf{y}) d \sigma(\mathbf{y})k\\
	\nm
	&\ds\quad +\int_{\dr D} \Gamma_2(\mathbf{x}-\mathbf{y})
	\phibf(\mathbf{y}) d \sigma(\mathbf{y})k^2+O(k^3) \\
	\nm
	&:=\ds \acute{\mathcal{A}}  _D^{0} [\phibf](\mathbf{x})+\acute{\mathcal{A}}  _D^{1} [\phibf](\mathbf{x})k+\acute{\mathcal{A}} ^{2}[\phibf](\mathbf{x})k^2+O(k^3),
	\end{aligned}
	\end{equation*}
we also have
\begin{equation}\label{kerA}
\begin{aligned}
\nabla \times\ds\mathcal{A}_D^{k}[\phibf](\mathbf{x}) =&\ds\int_{\dr D}\nabla_{\mathbf{x}} \times \Gamma^0(\mathbf{x}-\mathbf{y})  \phibf(\mathbf{y}) d \sigma(\mathbf{y})\\
\nm
&
+\int_{\dr D}\nabla_{\mathbf{x}} \times\Gamma^2(\mathbf{x}-\mathbf{y})k^2  \phibf(\mathbf{y}) d \sigma(\mathbf{y})+O(k^3)\\
\nm
=&\ds\nabla \times\ds\acute{\mathcal{A}} ^{0}_D[\phibf](\mathbf{x})+\nabla \times\ds\acute{\mathcal{A}}_D ^{2}[\phibf](\mathbf{x})k^2+O(k^3),
\end{aligned}
\end{equation}
and
\begin{equation}\label{kerkerA}
\begin{aligned}
&\ds \nabla \times\nabla \times\ds\mathcal{A}_D^{k}[\phibf](\mathbf{x})\\
\nm
&=\ds k^2\ds\mathcal{A}_D^{k}[\phibf](\mathbf{x})+\nabla{S}_D^{ k}[\g_{\dr D}\cdot \phibf](\mathbf{x})\\
\nm
&= k^2\ds\acute{\mathcal{A}}_D ^{0}[\phibf](\mathbf{x})+\ds\int_{\dr D}\nabla_{\mathbf{x}} \left (   \Gamma^0(\mathbf{x}-\mathbf{y})\g_{\dr D}\cdot \phibf(\mathbf{y}) \right ) d \sigma(\mathbf{y})\\
\nm
&~~~+\ds\int_{\dr D}\nabla_{\mathbf{x}} \left ( \Gamma^2(\mathbf{x}-\mathbf{y}) \g_{\dr D}\cdot \phibf(\mathbf{y}) \right ) d \sigma(\mathbf{y})k^2+O(k^3)\\
\nm
&:=\ds k^2\ds\acute{\mathcal{A}}_D ^{0}[\phibf](\mathbf{x})+\nabla\acute{{S}} _D^{ 0}[\g_{\dr D}\cdot \phibf](\mathbf{x})+\nabla\acute{{S}} _D^{ 2}[\g_{\dr D}\cdot \phibf](\mathbf{x})k^2+O(k^3).
\end{aligned}
\end{equation}
Similarly, for the boundary operators $\mathcal{M}_D^k$ and  $\mathcal{L}_D^k$, we have the expansions that
\begin{equation}\label{a_defM}
\begin{aligned}
\mathcal{M}^k_D[\phibf]
&= \nubf(\mathbf{x})  \times\int_{\dr D}\nabla_{\mathbf{x}} \times \Gamma^0(\mathbf{x}-\mathbf{y}) \phibf(\mathbf{y}) d \sigma(\mathbf{y})\\
&\quad
+ \nubf(\mathbf{x})  \times\int_{\dr D}\nabla_{\mathbf{x}} \times\Gamma^2(\mathbf{x}-\mathbf{y})k^2  \phibf(\mathbf{y}) d \sigma(\mathbf{y})+O(k^3)\\
\nm
&\ds:=\acute{\mathcal{M}} ^0_D+\acute{\mathcal{M}} ^2_Dk^2+O(k^3).
\end{aligned}
\end{equation}
\begin{equation}\label{a_L}
\begin{aligned}
\mathcal{L}^k_D[\phibf]&=\nubf \times\acute{\mathcal{A}} ^0_D[\phibf]k^2+\nu\times \nabla \acute{{S}} _D^{ 0}[\g_{\dr D}\cdot \phibf] +\nu\times \nabla \acute{{S}} _D^{ 2}[\g_{\dr D}\cdot \phibf]k^2+O(k^3).
\end{aligned}
\end{equation}

\subsection{Asymptotics for the fields $(\bE^{i}, \bH^{i})$}
Form \eqref{in}, it follows that
\begin{equation*}
\begin{aligned}
\bH^{i}&=\mathbf{d} \times \mathbf{p}+\mathrm{i}(\mathbf{d} \times \mathbf{p}) (\mathbf{x} \cdot \mathbf{d}) k-(\mathbf{d} \times \mathbf{p}) (\mathbf{x} \cdot \mathbf{d}) ^2k^2-\mathrm{i}(\mathbf{d} \times \mathbf{p}) (\mathbf{x} \cdot \mathbf{d}) ^3k^3+O(k^4)\\
&:=\bH^{i}_0+\bH^{i}_1k+\bH^{i}_2k^2+\bH^{i}_3k^3+O(k^4).
\end{aligned}
\end{equation*}
Similarly, for the incident electric field, there holds the following asymptotic expansion:
\begin{equation}\label{as_in2}
\begin{aligned}
\bE^{i}&=\mathbf{p}+\mathrm{i}\mathbf{p} (\mathbf{x} \cdot \mathbf{d}) k-\mathbf{p} (\mathbf{x} \cdot \mathbf{d}) ^2k^2-\mathrm{i}\mathbf{p}(\mathbf{x} \cdot \mathbf{d}) ^3k^3+O(k^4)\\
&:=\bE^{i}_0+\bE^{i}_1k+\bE^{i}_2k^2+\bE^{i}_3k^3+O(k^4).
\end{aligned}
\end{equation}

\subsection{The representation of the solution}
Using the layer potentials defined in Section \ref{sec2}, the solution to  (\ref{eq:maxwell}) can be represented as
\begin{equation}
\label{represent}
\ds\bE (\mathbf{x})= \left \{
\begin{array}{ll}
\ds \bE^{i}(\mathbf{x}) +  \nabla \times \mathcal{A}_{D}^{k}
[\phibf](\mathbf{x}) +
\nabla\times\nabla\times\mathcal{A}_{D}^{k} [\psibf](\mathbf{x}) ,\quad &\mathbf{x} \in \mathbb{R}^3 \setminus \overline{{D}},\\
\nm \ds\nabla \times \mathcal{A}_{D}^{k_c} [\phibf](\mathbf{x}) +
\nabla\times\nabla\times\mathcal{A}_{D}^{k_c} [\psibf](\mathbf{x}) ,\quad &\mathbf{x}
\in {D},
\end{array}
\right .
\end{equation}
and
\begin{equation*}
\bH(\mathbf{x}) = -\frac{\mathrm{i}}{\omega }\bigr(\nabla \times \bE\bigr)(\mathbf{x}),\quad \mathbf{x} \in \mathbb{R}^3\setminus \dr {D}.
\end{equation*}
Here $ k=\omega$, $ k_c=\omega \sqrt{\varepsilon_c}$, and the pair $(\phibf, \psibf) \in TH({\rm div}, \dr D)
\times TH({\rm div}, \dr D)$ is the unique solution to
\begin{equation*}
\begin{bmatrix}
\ds\mathcal{I} +  \mathcal{M}_D^{k_c} -
\mathcal{M}_D^{k} &
\ds\mathcal{L}_D^{k_c} - \mathcal{L}_D^{k} \\
\ds\mathcal{L}_D^{k_c} - \mathcal{L}_D^{k} & \ds \left(
\frac{k_c^2+k^2}{2 } \right)\mathcal{I} +
{k_c^2}\mathcal{M}_D^{k_c} -
{k^2}\mathcal{M}_D^{k}
\end{bmatrix}
\begin{bmatrix}
\phibf \\ \psibf
\end{bmatrix}
= \left.\begin{bmatrix}
{\nu}\times \bE^{i}\\
\mathrm{i} \omega {\nu} \times \bH^{i}
\end{bmatrix}\right|_{\dr D},
\end{equation*}
which is equivalent to
\begin{equation}
\label{phi_psi2}
\begin{bmatrix}
\ds \mathcal{I} +  \mathcal{M}_D^{k_c} -
\mathcal{M}_D^{k} &
\ds\mathcal{L}_D^{k_c} - \mathcal{L}_D^{k} \\
\ds\frac{\mathcal{L}_D^{k_c} - \mathcal{L}_D^{k}}{\omega }  & \ds\frac{ \left(
	\frac{k_c^2+k^2}{2 } \right)\mathcal{I} +
	{k_c^2}\mathcal{M}_D^{k_c} -
	{k^2}\mathcal{M}_D^{k}}{\omega }
\end{bmatrix}
\begin{bmatrix}
\phibf \\ \psibf
\end{bmatrix}
= \left.\begin{bmatrix}
{\nu}\times \bE^{i}\\
\mathrm{i}  {\nu} \times \bH^{i}
\end{bmatrix}\right|_{\dr D}, 
\end{equation}
where $\mathcal{I}$ represents the identity operator on $(\phibf, \psibf) \in TH({\rm div}, \dr D)
\times TH({\rm div}, \dr D)$.
\subsubsection{Asymptotic expansions and representation for the potential $\left[\phibf , \psibf\right]^{'}$}
Let
\begin{equation*}
\begin{aligned}
\mathcal{N}^k=
\begin{bmatrix}
\ds \mathcal{I} +  \mathcal{M}_D^{k_c} -
\mathcal{M}_D^{k} &
\ds\mathcal{L}_D^{k_c} - \mathcal{L}_D^{k} \\
\ds\frac{\mathcal{L}_D^{k_c} - \mathcal{L}_D^{k}}{\omega } & \ds \frac{ \left(
	\frac{k_c^2+k^2}{2 } \right)\mathcal{I} +
	{k_c^2}\mathcal{M}_D^{k_c} -
	{k^2}\mathcal{M}_D^{k}}{\omega }
\end{bmatrix},\\
\end{aligned}
\end{equation*}
\eqref{phi_psi2} can be written as
\begin{equation}\label{lk_phi}
\mathcal{N}^k
\begin{bmatrix}
\phibf \\\psibf
\end{bmatrix}
= \left.\begin{bmatrix}
{\nu}\times \bE^{i}\\
\mathrm{i}  {\nu} \times \bH^{i}
\end{bmatrix}\right|_{\dr D}.
\end{equation}

By using the asymptotic expansion of operators \eqref{a_defM} and \eqref{a_L}, we have the following asymptotic expansion,
\begin{equation}\label{lk}
\begin{aligned}
&\mathcal{N}^k=\begin{bmatrix}
\ds \mathcal{I}   &
\ds 0 \\
\nm
\ds \nubf(\mathbf{x}) \times \acute{\mathcal{A}}_D^0+ \nubf(\mathbf{x}) \times\g\acute{{S}} _D^2[\g_{\dr D} \cdot ] & \frac{\tilde{C} }{2}\mathcal{I}   +\tilde{C} \acute{\mathcal{M}} _D^{0}
\end{bmatrix}
+\begin{bmatrix}
\ds 0   &
\ds 0 \\
\nm
\ds \tilde{C} ^{\frac{3}{2}} \nubf(\mathbf{x}) \times \acute{\mathcal{A}}_D^1&0
\end{bmatrix}k^{\frac{1}{2}}
\\
\nm
&+\begin{bmatrix}
\ds \tilde{C}  \acute{\mathcal{M}} _D^{2}   &
\ds \tilde{C} \nu \times \left ( \acute{\mathcal{A}}_D^0 +\g\acute{{S}} _D^2[\g_{\dr D} \cdot ] \right )  \\
\nm
 \nu \times\left ( \tilde{C} ^2 \acute{\mathcal{A}}_D ^2+\tilde{C} ^2\g\acute{{S}} _D^4[\g_{\dr D} \cdot ] - \acute{\mathcal{A}}_D^0- \g\acute{{S}} _D^2[\g_{\dr D} \cdot ] \right )    & \frac{\mathcal{I}}{2}   - \acute{\mathcal{M}} _D^{0}
\end{bmatrix}k+O(k^{\frac{3}{2}})\\
\nm
&\ds :=\mathcal{N}_0+\mathcal{N}_{\frac{1}{2} }k^{\frac{1}{2} }+\mathcal{N}_1k+O(k^{\frac{3}{2}}).
\end{aligned}
\end{equation}
Denote
\begin{equation*}
\begin{aligned}
\mathcal{N}_0&=\begin{bmatrix}
\ds \mathcal{I}   &
\ds 0 \\
\nm
\ds \nubf(\mathbf{x}) \times \acute{\mathcal{A}}_D^0+ \nubf(\mathbf{x}) \times\g\acute{{S}} _D^2[\g_{\dr D} \cdot ] & \tilde{C} \Big(\frac{1}{2}\mathcal{I}   +\acute{\mathcal{M}} _D^{0}\Big)
\end{bmatrix}\\
&:=\begin{bmatrix}
\ds \mathcal{I}  &
\ds 0 \\
\ds \mathcal{Q} & \ds \mathcal{W}
\end{bmatrix}.
\end{aligned}
\end{equation*}
We note that $-1/2$ is not an eigenvalue of $\mathcal{M}_D^{k}$ (cf. \cite{RG}), and thus the operator $\mathcal{W}$ is invertible on $TH(\mathrm{div}, \partial D)$. 


We assume the following expansion that
for the pair $(\phibf, \psibf) \in TH({\rm div}, \dr D)
\times TH({\rm div}, \dr D)$,
\begin{equation}\label{phi}
\begin{aligned}
\begin{bmatrix}
\phibf \\\psibf
\end{bmatrix}:=\begin{bmatrix}
\phibf^0\\
\psibf^0
\end{bmatrix}+
\begin{bmatrix}
\phibf^{\frac{1}{2}}\\
\psibf^{\frac{1}{2}}
\end{bmatrix}k^{\frac{1}{2}}
+\begin{bmatrix}
\phibf^1 \\
\psibf^1
\end{bmatrix}k+O(k^{\frac{3}{2}}).
\end{aligned}
\end{equation}
Meanwhile, the incident field can be expanded as
\begin{equation*}
\begin{bmatrix}
 \bE^{i}\\ \bH^{i}
\end{bmatrix}:=\begin{bmatrix}
 \bE^{i}_0\\
\bH^{i}_0
\end{bmatrix}+k \begin{bmatrix}
 \bE^{i}_1\\
 \bH^{i}_1
\end{bmatrix}
+O(k^2).
\end{equation*}
Thus, the right-hand part in \eqref{phi_psi2} can be written as
\begin{equation}\label{right}
\begin{bmatrix}
{\nu}\times \bE^{i}\\
\mathrm{i}{\nu} \times \bH^{i}
\end{bmatrix}:=\begin{bmatrix}
{\nu}\times \bE^{i}_0\\
\mathrm{i}{\nu} \times \bH^{i}_0
\end{bmatrix}+k \begin{bmatrix}
{\nu}\times \bE^{i}_1\\
\mathrm{i}{\nu} \times \bH^{i}_1
\end{bmatrix}
+O(k^2).
\end{equation}

Noting that the operator $\mathcal{N}^k$ is invertible for $k\ll 1$, and from \eqref{lk_phi}, we obtain

\begin{equation}\label{phi_ilk}
\begin{bmatrix}
\phibf \\ \psibf
\end{bmatrix}
=\left(\mathcal{N}^k\right)^{-1} \left.\begin{bmatrix}
{\nu}\times \bE^{i}\\
\mathrm{i}  {\nu} \times \bH^{i}
\end{bmatrix}\right|_{\dr D}.
\end{equation}
Considering the operator $\mathcal{N}^k$ and by \eqref{lk}, we have

\begin{equation*}
\begin{aligned}
\ds\mathcal{N}^k=\mathcal{N}_0+\mathcal{N}_{\frac{1}{2} }k^{\frac{1}{2} }+\mathcal{N}_1k+O(k^{\frac{3}{2}})=\mathcal{N}^0\left(\mathcal{I}+\mathcal{N}_0^{-1}\mathcal{N}_{\frac{1}{2} }k^{\frac{1}{2} }+\mathcal{N}_0^{-1}\mathcal{N}_1k+O\left(k^{\frac{3}{2}}\right)\right).
\end{aligned}
\end{equation*}
 Hence, we can further deduce that
\begin{equation*}
\begin{aligned}
\ds \left(\mathcal{N}^k\right)^{-1}&=\left(I+\mathcal{N}_0^{-1}\mathcal{N}_{\frac{1}{2} }k^{\frac{1}{2} }+\mathcal{N}_0^{-1}\mathcal{N}_1k+O(k^{\frac{3}{2}})\right)^{-1}\mathcal{N}_0^{-1}\\
&=\sum_{n=0}^{\infty } \left ( -\mathcal{N}_0^{-1}\mathcal{N}_{\frac{1}{2} }k^{\frac{1}{2} }-\mathcal{N}_0^{-1}\mathcal{N}_1k-O(k^{\frac{3}{2}})\right ) ^n\mathcal{N}_0^{-1}\\
&=\mathcal{N}_0^{-1}-\mathcal{N}_0^{-1}\mathcal{N}_{\frac{1}{2} }\mathcal{N}_0^{-1}k^{\frac{1}{2} }-\mathcal{N}_0^{-1}\mathcal{N}_1\mathcal{N}_0^{-1}k+O\left(k^{\frac{3}{2}}\right),
\end{aligned}
\end{equation*}
where $\mathcal{N}_0^{-1}$ is given by
\begin{equation*}
\begin{aligned}
\left(\mathcal{N}_0\right)^{-1}&=\begin{bmatrix}
\ds \mathcal{I}   &
\ds 0 \\
\nm
\ds \nubf(\mathbf{x}) \times \acute{\mathcal{A}}_D^0+ \nubf(\mathbf{x}) \times\g\acute{{S}} _D^2[\g_{\dr D} \cdot ] & \frac{\tilde{C} }{2}\mathcal{I}   +\tilde{C} \acute{\mathcal{M}} _D^{0}
\end{bmatrix}^{-1}\\
&:=\begin{bmatrix}
\ds \mathcal{I}  &
\ds 0 \\
\ds -\mathcal{W}^{-1}\mathcal{Q} & \ds \mathcal{W}^{-1}
\end{bmatrix}.
\end{aligned}
\end{equation*}
Therefore, by \eqref{phi_ilk}, we have
\begin{equation}\label{phi_ilk2}
\begin{aligned}
&\begin{bmatrix}
\phibf \\\psibf
\end{bmatrix}
=\left(\mathcal{N}_0^{-1}-\mathcal{N}_0^{-1}\mathcal{N}_{\frac{1}{2} }\mathcal{N}_0^{-1}k^{\frac{1}{2} }-\mathcal{N}_0^{-1}\mathcal{N}_1\mathcal{N}_0^{-1}k+O\left(k^{\frac{3}{2}}\right)\right) \left.\begin{bmatrix}
{\nu}\times \bE^{i}\\
\mathrm{i}  {\nu} \times \bH^{i}
\end{bmatrix}\right|_{\dr D}.
\end{aligned}
\end{equation}
Note that there is
\begin{equation*}
\nubf(\mathbf{x}) \times \acute{\mathcal{A}}_D^1\left [ \nu\times \bE_0^{i} \right ] =-\frac{1}{4\pi} \nubf(\mathbf{x}) \times\int_{\partial D }\nu\times \bE_0^{i} d\sigma_y=-\frac{1}{4\pi}\nubf(\mathbf{x}) \times\int_{ D }\nabla \times \bE_0^{i} d\sigma_y=0.
\end{equation*}
From \eqref{phi_ilk2} and \eqref{right}, we obtain the representations that
\begin{equation}\label{phi1}
\phibf^0={\nu}\times \bE^{i}_0,
\end{equation}
\begin{equation}\label{phi2}
\nu(\mathbf{x}) \times \acute{\mathcal{A}}_D^0\left [ \phi^0 \right ] + \frac{\tilde{C} }{2}\mathcal \psi^0  +\tilde{C} \acute{\mathcal{M}} _D^{0}\left [ \psi^0 \right ]=\mathrm{i}\nu\times \bH_0^{i},
\end{equation}
\begin{equation*}
\phibf^{\frac{1}{2}}=0,
\end{equation*}
\begin{equation*}
\psibf^{\frac{1}{2}}=-\mathcal{W}^{-1}\tilde{C} ^{\frac{3}{2}} (\nubf(\mathbf{x}) \times \acute{\mathcal{A}}_D^1)\mathcal{W}^{-1}\mathcal{Q}[\nu\times \bE_0^{i}]=0,
\end{equation*}
\begin{equation}\label{phi3}
\phibf^1={\nu}\times \bE^{i}_1- \tilde{C} \acute{\mathcal{M}} _D^{2}\left [ \phibf^0 \right ]-\tilde{C} \nu \times \acute{\mathcal{A}}_D ^0\left[\psi^0 \right]-\tilde{C} \nu \times \g\acute{{S}} _D^2[\g_{\dr D} \cdot \psi^0].
\end{equation}

\subsection{Asymptotics for the solution $\bE(\mathbf{x})$}
By substituting the expansions of the potential $\phi, \psi$ from equation \eqref{phi}, as well the expansions of the operators $\nabla \times \mathcal{A}_D^k$ from equation \eqref{kerA} and $\nabla \times\nabla \times \mathcal{A}_D^k$ from equation \eqref{kerkerA} into the representation of the solution $\bE(\mathbf{x})$ as illustrated in equation \eqref{represent}, the asymptotics expansion of $\bE(\mathbf{x})$ is as follows. For $ \mathbf{x} \in \mathbb{R}^3 \setminus \overline{{D}}$, it follows that
\begin{equation*}
\begin{aligned}
\bE (\mathbf{x})= &
\ds \bE^{i}(\mathbf{x})+\nabla \times \mathcal{A}_{D}^{k}
[\phibf](\mathbf{x})+
\nabla\times\nabla\times\mathcal{A}_{D}^{k} [\psibf](\mathbf{x})\\
\nm
=&\ds \bE^{i}_0(\mathbf{x}) +k\bE^{i}_1(\mathbf{x})+\nabla \times \mathcal{A}_{D}^{k}
[\phibf^0+\phibf^1 k](\mathbf{x}) +
\ds\nabla\times\nabla\times\mathcal{A}_{D}^{k} [ \psibf^0+ k \psibf^1](\mathbf{x})+O(k^2)\\
\nm
=&\ds \bE^{i}_0(\mathbf{x}) +k\bE^{i}_1(\mathbf{x})+\nabla \times \acute{\mathcal{A}}_{D}^{0}
[\phibf^0+\phibf^1 k](\mathbf{x})\\
\nm
&\ds+ k^2 \nabla \times \acute{\mathcal{A}}_{D}^{2}
[\phibf^0+\phibf^1 k](\mathbf{x})+k^2\acute{\acute{\mathcal{A}}}_D^0[\psibf^0+ k \psibf^1](\mathbf{x})\\
\nm
&\ds+\nabla\acute{{S}} _D^{0}[\g_{\dr D}\cdot [\psibf^0+ k \psibf^1]]+k^2\nabla\acute{{S}} _D^{ 2}[\g_{\dr D}\cdot [\psibf^0+ k \psibf^1]]+O(k^2)\\
\nm
:=&\ds\bE_0(\mathbf{x})+k\bE_1(\mathbf{x})+O(k^2).
\end{aligned}
\end{equation*}
In a similar way, for $\mathbf{x} \in {{D}}$, we have
\begin{equation*}
\begin{aligned}
\bE (\mathbf{x})=
& \nabla \times \mathcal{A}_{D}^{k_c}
[\phibf^0+\phibf^1 k](\mathbf{x})+
\ds\nabla\times\nabla\times\mathcal{A}_{D}^{k_c} [ \psibf^0+ k \psibf^1 ](\mathbf{x})+O\left(k^2\right)\\
\nm
=&\ds  \nabla \times \acute{\mathcal{A}} _{D}^{0}
[\phibf^0+\phibf^1 k](\mathbf{x})+ \varepsilon_ck^2 \nabla \times \acute{\mathcal{A}} _{D}^{2}
[\psibf^0+ k \psibf^1](\mathbf{x})+\varepsilon_ck^2\acute{\acute{\mathcal{A}}}_D^0[\psibf^0+ k \psibf^1](\mathbf{x})\\
\nm
&+\nabla\acute{{S}} _D^{0}[\g_{\dr D}\cdot [\psibf^0+ k \psibf^1]] +\varepsilon_ck^2\nabla\acute{{S}} _D^{ 2}[\g_{\dr D}\cdot [\psibf^0+ k \psibf^1]]+O(k^2)\\
\nm
:=&\ds\bE_0(\mathbf{x})+k\bE_1(\mathbf{x})+O(k^2),
\end{aligned}
\end{equation*}
where the parts $\bE_{0}, \bE_1,$ should be represented as
\begin{equation}
\label{represent_e0}
\ds\bE_0 (\mathbf{x})= \left \{
\begin{array}{ll}
\ds \bE^{i}_0(\mathbf{x}) +  \nabla \times \acute{\mathcal{A}} _{D}^{0}
[\phibf^0](\mathbf{x}) +\nabla\acute{{S}} _D^{0}[\g_{\dr D}\cdot [\psibf^0]] ,\qquad &\mathbf{x} \in \mathbb{R}^3 \setminus \overline{{D}},\\
\nm \ds \nabla \times \acute{\mathcal{A}} _{D}^{0}
[\phibf^0](\mathbf{x}) +\nabla\acute{{S}} _D^{0}[\g_{\dr D}\cdot [\psibf^0]] ,\qquad &\mathbf{x}
\in {D},
\end{array}
\right .
\end{equation}

\begin{equation}
\label{represent_e1}
\ds\bE_1 (\mathbf{x})= \left \{
\begin{aligned}
&\ds \bE^{i}_1(\mathbf{x}) +  \nabla \times \acute{\mathcal{A}} _{D}^{0}
[\phibf^1](\mathbf{x}) +\nabla\acute{{S}} _D^{0}[\g_{\dr D}\cdot [\psibf^1]] ,\quad\qquad \mathbf{x} \in \mathbb{R}^3 \setminus \overline{{D}},\\
\nm
&\ds \nabla \times \acute{\mathcal{A}} _{D}^{0}
[\phibf^1](\mathbf{x}) +\tilde{C} \nabla \times \acute{\mathcal{A}} _{D}^{2}
[\phibf^0](\mathbf{x})+\tilde{C}  \acute{\mathcal{A}} _{D}^{0}[\psibf^0](\mathbf{x})\\
&\ds+\nabla\acute{{S}} _D^{0}[\g_{\dr D}\cdot [\psibf^1]]+ \nabla\acute{{S}} _D^{2}[\g_{\dr D}\cdot [\psibf^0]] ,\qquad\qquad\quad \mathbf{x}
\in {D}.
\end{aligned}
\right .
\end{equation}

\section{Estimate of $\|\bE\|_{L^\infty(\Omega\backslash \overline{B_1\cup B_2})} $}
The objective of this section is to establish the proof of Theorem \ref{th:main01}. We will estimate $\bE_{0}$ and $\bE_1$ individually. Prior to presenting the proof, we will introduce a lemma concerning the boundary condition of the Maxwell equation \eqref{eq:maxwell}.
\begin{lem}\label{lemma1}
	There holds the following transmission condition
	\begin{equation}\label{boundary}
	\left(\nu \cdot\varepsilon \mathbf{E}\right)^{+}=\left(\nu \cdot\varepsilon_c \mathbf{E}\right)^{-} \quad \text { on } \quad \partial B_{j}, j=1,2.
	\end{equation}
	\begin{proof}
		The proof follows a similar spirit to that of Lemma 3.3 in \cite{DengliuG}.
		By taking the divergence of both sides of the second equation in \eqref{eq:maxwell} , we have
		\begin{equation}\label{eq1}
		\nabla\cdot\nabla  \times \bH= - \mathrm{i} \omega \varepsilon \nabla\cdot\bE=0.
		\end{equation}
		By conducting the inner product of both sides of the second equation in \eqref{eq:maxwell} with the gradient of a test function  $\psi \in C_{0}^{\infty}\left(\mathbb{R}^{3}\right)$ , and integrating both sides over  $\mathbb{R}^{3}$, there holds
		\begin{equation}\label{eq2}
		\int_{\mathbb{R}^{3}}(\nabla \times \mathbf{H}) \cdot \nabla \psi=-\mathrm{i} \omega\varepsilon \int_{\mathbb{R}^{3}} \mathbf{E} \cdot \nabla \psi.
		\end{equation}
	Utilizing both the vector calculus identity and Green's formula allows us to rewrite the left-hand side of equation \eqref{eq2} as 
		\begin{equation*}
		\int_{\mathbb{R}^{3}}(\nabla \times \mathbf{H}) \cdot \nabla \psi=\int_{\mathbb{R}^{3}} \nabla \cdot(\mathbf{H} \times \nabla \psi)=0.
		\end{equation*}
		Using \eqref{eq1} and Green's formula, the RHS of \eqref{eq2} implies
		\begin{equation}\label{eq3}
		\begin{aligned}
		& -\mathrm{i}\omega \varepsilon\int_{\mathbb{R}^{3}} \mathbf{E} \cdot \nabla \psi \\
		= & -\mathrm{i}\omega\varepsilon \int_{\mathbb{R}^{3} \backslash  B_1\cup  B_2}\mathbf{E} \cdot \nabla \psi-\mathrm{i}\omega \varepsilon \int_{  B_1\cup  B_2} \mathbf{E} \cdot \nabla \psi \\
		= & -\mathrm{i}\omega\varepsilon \int_{\mathbb{R}^{3} \backslash  B_1\cup  B_2} \nabla\cdot( \mathbf{E} \psi)-\mathrm{i}\omega \varepsilon \int_{  B_1\cup  B_2} \nabla\cdot( \mathbf{E} \psi) \\
		= &\mathrm{i}\omega \varepsilon\int_{ \dr B_1\cup \dr B_2} \left(\nu \cdot \mathbf{E}\right)^{+} \psi-\mathrm{i}\omega\varepsilon \int_{ \dr B_1\cup \dr B_2} \left(\nu \cdot \mathbf{E}\right)^{-} \psi.
		\end{aligned}
		\end{equation}
		Substituting \eqref{eq2} and \eqref{eq3} into \eqref{eq1}, we have
		\begin{equation*}
\int_{\dr B_1\cup \dr B_2} \left(\nu \cdot\varepsilon \mathbf{E}\right)^{+} \psi=\int_{\dr B_1\cup \dr B_2} \left(\nu \cdot\varepsilon \mathbf{E}\right)^{-} \psi .
		\end{equation*}
Finally, as $\psi$ is arbitrary, we obtain equation \eqref{boundary}, which completes the proof.
	\end{proof}
\end{lem}

\subsection{Analysis of $\bE_0$}
First, let's analyze the properties of the leading order term, $\bE_0$. Here, we obtain the following result in Lemma \ref{le0}.
\begin{lem}\label{le0}
Let $(\bE,\bH)$ be the soultion of \eqref{eq:maxwell}, and let $\bE_{0}$ represent the leading-order term of $\bE$.	For sufficiently small $\epsilon > 0$, there exists a positive constant $C^*$, which depends only on $r$ and $\mathbf{p}$, such that
	\begin{equation}\label{E_0_eatimate}
	\begin{aligned}
	\|\bE_0\|_{L^\infty(\mathbb{R}^3\backslash \overline{B_1\cup B_2})} \sim C^*r|p_1| \frac{1}{\epsilon |\ln \epsilon|}.
	\end{aligned}
	\end{equation}
\end{lem}
Before proving Lemma \ref{le0}, let us first present the following preliminary results.
\begin{lem}\label{lem63}
Let $(\bE,\bH)$ be the soultion of \eqref{eq:maxwell}, and let $\bE_{0}$ represent the leading-order term of $\bE$. Then, there exists a function $u$ in $H^{1}_{loc}(\mathbb{R}^3)$ such that
\begin{equation}\label{e0eqgu}
\bE_0=\nabla u , \quad \mbox{in} \quad\mathbb{R}^3\setminus\partial D,
\end{equation}
and
\begin{equation}\label{master_u}
\Delta  u =0, \quad \mbox{in} \quad\mathbb{R}^3\setminus\partial D.
\end{equation}
\begin{proof}
For $\mathbf{x}\in \mathbb{R}^3 \setminus \overline{{D}}$, taking the curl of $\bE_0$ in \eqref{represent_e0} and combine it with \eqref{as_in2} to obtain the following expression:
\begin{equation*}
\begin{aligned}
\nabla \times \bE_0(\mathbf{x})
\ds=\nabla \times\bE^{i}_0(\mathbf{x})+\nabla \acute{{S}} _D^{0}[\g_{\dr D}\cdot [\phibf^0]]
\ds=-\nabla \acute{{S}} _D^{0}[{\nu}\cdot\nabla \times\bE^{i}_0|_{\dr D} ]
=0,
\end{aligned}
\end{equation*}
taking the divergence of $\mathbf{E}_0$, we obtain
\begin{equation*}
\begin{aligned}
\nabla \cdot \bE_0(\mathbf{x})=\nabla \cdot\bE^{i}_0(\mathbf{x})+  \nabla \cdot \nabla \times \acute{\mathcal{A}}_{D}^{0}
[\phibf^0](\mathbf{x}) +\nabla \cdot\nabla\acute{{S}} _D^{0}[\g_{\dr D}\cdot [\psibf^0]]=0.
\end{aligned}
\end{equation*}
Similarly, for $\mathbf{x}\in {D}$, there is
\begin{equation*}
\nabla \times \bE_0(\mathbf{x})=0,\quad \nabla \cdot \bE_0(\mathbf{x})=0.
\end{equation*}
Therefore, since $\mathbf{E}_0$ is a curl-free function, according to the Helmholtz decomposition, there exists a function $u$ in the local Sobolev space $H^1_{\text{loc}}(\mathbb{R}^3)$ such that
\begin{equation*}
\bE_0=\nabla u , \quad \mbox{in} \quad\mathbb{R}^3\setminus\partial D.
\end{equation*}
Furthermore, by using the fact $\nabla \cdot \bE_0=0$, we arrive at 
\begin{equation*}
\Delta  u =0, \quad \mbox{in} \quad\mathbb{R}^3\setminus\partial D.
\end{equation*}
The proof is complete.
\end{proof}
\end{lem}

\begin{lem}\label{E0u}
Let $(\mathbf{E},\mathbf{H})$ be the solution of equation \eqref{eq:maxwell}, and let  $\mathbf{E}_0$ be the leading-order term of $\mathbf{E}$. Then there exists a function $u$ in $H^{1}_{loc}(\mathbb{R}^3)$ such that $\mathbf{E}_0=\nabla u$ in $\mathbb{R}^3\setminus\partial D$ and it satisfies the following equations:
\begin{equation}\label{u_master}
\left \{
\begin{aligned}
&\Delta u = 0 \hspace*{3.90cm} \mbox{in} \quad \RR^3\setminus\partial D,\\
&\ds u^+= u^-+C  \hspace*{2.55cm}\quad \mbox{on} \quad \partial D,\\
&\nabla u=0\hspace*{3.55cm}\quad \mbox{in} \quad  D,\\
&\int_{\partial B_i }^{} \frac{\partial u}{\partial \nu} =0, \hspace*{2.45cm}\quad \mbox{on} \quad \partial B_i \left (i=1,2 \right ) ,\\
&u(\mathrm{x})-H(\mathrm{x})=O\left(|\mathrm{x}|^{-2}\right) \qquad \text { as }|\mathrm{x}| \rightarrow \infty,
\end{aligned}
\right.
\end{equation}
where $H(\mathbf{x})=\mathbf{p}\cdot\mathbf{x}.$
\begin{proof}
By Lemma \ref{lem63}, there exists  $u \in H^{1}_{loc}(\mathbb{R}^3)$ such that 
$\bE_0=\nabla u, \Delta  u =0$ in $\mathbb{R}^3\setminus\partial D$.
 
By the boundary condition 
\[
[\nu \times \mathbf{E}] = 0 \quad \text{on} \quad \partial D,
\]
it is easy to see that $[\nu \times \mathbf{E}_0] = 0$. Furthermore, since 
\begin{equation*}
\nu\times\bE_0=\nu\times\nabla u.
\end{equation*}
we can conclude that 
\begin{equation}\label{partial_T}
\left(\frac{\partial u}{\partial T} \right)^{+} =\left(\frac{\partial u}{\partial T} \right)^{-}\quad \text { on } \partial B_1\cup \partial B_2,
\end{equation}
where  $T$ is tangential direction on $\partial B$.
 \eqref{partial_T} equivalent to
\begin{equation*}
u^{+}=u^{-}+C_i,\quad \text{on}~ \partial B_i,\quad i=1,2,
\end{equation*}
where $C_i,~i=1,2$ are constants.

By substituting the asymptotic expansion of $\mathbf{E}$ into the boundary condition \eqref{boundary}, we can obtain the following expression 
\begin{equation}\label{boundary_E}
\left(\nu \cdot \left ( \bE_0+k\bE_1+O(k^2) \right )\right)^{+}=\left(\nu \cdot \varepsilon_{c}   \left ( \bE_0+k\bE_1+O(k^2) \right )\right)^{-} \quad \text { on } \partial {D},
\end{equation}
together with $\varepsilon _c=O(\frac{1}{\omega })$, we can establish the boundary condition for the leading-order term $\bE_0$:
\begin{equation}\label{boundary_e_0}
\left(\nu \cdot \bE_0\right)^{-}=0 \quad \text { on } \partial {D}.
\end{equation}
By using $\bE_0=\nabla u$, we can write that
\begin{equation*}
\left(\frac{\partial u}{\partial \nu}\right)^{-}=0 \quad \text { on } \partial {D}.
\end{equation*}
Furthermore, using \eqref{master_u} and Green's first theorem, it can be deduced that 
\begin{equation*}
0\le \int_{D }\left | \nabla u \right | ^2\mathrm{d}\mathbf{x} = \int_{\partial D }\left (u\frac{\partial u}{\partial \nu}   \right ) ^-\mathrm{d}s=0,
\end{equation*}
Consequently, we can conclude that  $\nabla u=0$ in $D$.

From \eqref{eq:maxwell}, it follows that
\begin{equation*}
\nabla  \times \bH= - \mathrm{i}\omega \varepsilon_c \bE  \quad\mbox{in} \quad D,
\end{equation*}
 from the asymptotic expansions of $\bE$ and $\bH$, we have
\begin{equation*}
\nabla  \times \bH_0= - \mathrm{i}\omega \varepsilon_c \bE_1=-\mathrm{i}\tilde{C} \bE_1, \quad \mbox{in} \quad D,
\end{equation*}
and
\begin{equation*}
\nabla\cdot \bE_1=\mathrm{i}/\tilde{C} \nabla \cdot \nabla  \times \bH_0=0,
\end{equation*}
therefore, we can obtain that
\begin{equation*}
\int_{\partial B_i }^{} \frac{\partial u}{\partial \nu} =\int_{\partial B_i }^{} (\nu\cdot \bE_0)^+ \\
=\tilde{C} \int_{\partial B_i }^{} (\nu\cdot \bE_1)^-=\tilde{C} \int_{ B_i }^{} \nabla \cdot \bE_1 =0,\quad i=1,2.
\end{equation*}
Finally, note that
$\phibf^0\in \mathcal{L} _0^2\left (  \partial D \right )^3$, $\g_{\dr D}\cdot \psibf^0\in \mathcal{L} _0^2\left (  \partial D \right )$, indeed, there is
\begin{equation*}
\int_{\partial D }\phi ^0(y)d\sigma_y=\int_{\partial D }\nu\times \bE_0^{i} d\sigma_y=\int_{ D }\nabla \times \bE_0^{i} d\sigma_y=0,
\end{equation*}
similarly, by taking surface divergence in \eqref{phi2}, we have
\begin{equation*}
\int_{\partial D} \nabla_{\partial D}\cdot  \psi^0  \mathrm{d}\sigma  =- \int_{\partial D}\psi^0 \nabla_{\partial D}[1] \mathrm{d}\sigma=0.
\end{equation*}
Then, by using the fact (see, e.g. \cite{kangbook})
\begin{equation*}
\left\{
\begin{aligned}
&\acute{{S}} _D^{0}[\g_{\dr D}\cdot [\psibf^0]]=O\left(|\mathrm{x}|^{-2}\right) \quad \text { as }|\mathrm{x}| \rightarrow \infty,\\
& \mathcal{A}_{D}^{0}
[\phibf^0](x)=O\left(|\mathrm{x}|^{-2}\right) \quad \text { as }|\mathrm{x}| \rightarrow \infty,
\end{aligned}
\right.
\end{equation*}
and $\nabla \left ( \mathbf{p}\cdot\mathbf{x} \right ) =\mathbf{p}$,
we obtain
\begin{equation*}
u(\mathrm{x})-\mathbf{p}\cdot\mathbf{x}=O\left(|\mathrm{x}|^{-2}\right) \qquad \text { as }|\mathrm{x}| \rightarrow \infty.
\end{equation*}

\end{proof}
\end{lem}

\begin{proof}[Proof of Lemma \ref{le0}:]
	By Lemma \ref{E0u}, $u$ satisfies system \eqref{u_master}, by using the conclusion in Theorem 2.1 of \cite{lim2009blow}, there exists a positive constant $C^*$ that depends only on $r$ and $\mathbf{p}$, such that
	\begin{equation*}
	\|\nabla u\|_{L^\infty(\mathbb{R}^3\backslash \overline{B_1\cup B_2})} \sim C^*r|p_1| \frac{1}{\epsilon |\ln \epsilon|}.
	\end{equation*}
	It then follows from (\ref{e0eqgu})  that
		\begin{equation*}
	\begin{aligned}
	\|\bE_0\|_{L^\infty(\mathbb{R}^3\backslash \overline{B_1\cup B_2})} \sim C^*r|p_1| \frac{1}{\epsilon |\ln \epsilon|}.
	\end{aligned}
	\end{equation*}
	The proof is complete.
\end{proof}

\subsection{The analysis of $\bE_1$}

It is remained to investigate the properties of the  $\bE_1$. Before our analysis, we introduce some propositions and important lemmas.

For any $\mathbf{x} \in \partial B_1$, let $\tilde{\mathbf{x}}:=-\mathbf{x} \in \partial B_2$, for any $\phi \in L^2_T(\partial D)$, $\varphi \in L^2(\dr D)$, define $\tilde{\phi } \left ( \tilde{\mathbf{x}}  \right ) :=\phi \left ( \mathbf{x} \right ) $, $\tilde{\varphi } \left ( \tilde{\mathbf{x}}  \right ) :=\varphi \left ( \mathbf{x} \right )$. Based on these definitions, we can establish the following results.

\begin{prop}
	For  $\mathbf{x} \in \partial B_1$, we have
	\begin{equation}\label{nueoin}
	\nu_{B_1}(\mathbf{x})\cdot \bE_0^{i}(\mathbf{x})=-\nu_{B_2}(\tilde{\mathbf{x}} )\cdot \bE_0^{i}(\tilde{\mathbf{x}} ).
	\end{equation}
	\begin{equation}\label{kere1in}
	\nu_{B_1}(\mathbf{x})\times \bE_1^{i}(\mathbf{x})=\nu_{B_2}(\tilde{\mathbf{x}} )\times \bE_1^{i}(\tilde{\mathbf{x}} ).
	\end{equation}
	\begin{proof}
	For  $\mathbf{x} \in \partial B_1$, and write $\tilde{\mathbf{x}}=-\mathbf{x} \in \partial B_2$, we have
	\begin{equation*}
	\nu_{B_1}(\mathbf{x})=-\nu_{B_2}(\tilde{\mathbf{x}} ),
	\end{equation*}
	by \eqref{as_in2}, $\bE_0^{i}=\mathbf{p}$, $\bE_1^{i}=\mathrm{i}\mathbf{p}(\mathbf{x}\cdot\mathbf{d})$,
	then it is easy to obtain the results shown by \eqref{nueoin} and \eqref{kere1in}.
	 The proof is complete.
	\end{proof}
\end{prop}

\begin{prop}
For  $\mathbf{x} \in \partial B_1$, $\phi \in L^2_T(\partial D)$, we have
\begin{equation*}\label{nunu}
\nu_{B_1}(\mathbf{x})\cdot\nabla \times  \mathcal{A}_{B_1\cup B_2}^{k} [\phi](\mathbf{x})=\nu_{B_2}(\tilde{\mathbf{x}} )\cdot\nabla \times \mathcal{A}_{B_1\cup B_2}^{k} [\tilde{ \phi}   ] (\tilde{\mathbf{x}} ).
\end{equation*}
\begin{equation*}
\nu_{B_1}(\mathbf{x})\times \mathcal{A}_{B_1\cup B_2}^{k} [\phi](\mathbf{x})=-\nu_{B_2}(\tilde{\mathbf{x}} )\times \mathcal{A}_{B_1\cup B_2}^{k}  [\tilde{ \phi}   ] (\tilde{\mathbf{x}} ).
\end{equation*}
	\begin{equation*}
\nu_{B_1}(\mathbf{x})\cdot {\mathcal{A}}_{B_1}^{k} [\phi](\mathbf{x})=-\nu_{B_2}(\tilde{\mathbf{x}} )\cdot {\mathcal{A}}_{B_2}^{k}  [\tilde{ \phi}   ] (\tilde{\mathbf{x}} ).
\end{equation*}
\begin{proof}
Let $\mathbf{x} \in \partial B_1$,   and write $\tilde{\mathbf{x}}=-\mathbf{x} \in \partial B_2$, then
\begin{equation*}
\nu_{B_1}({\mathbf{x}} )\cdot\nabla_{{\mathbf{x}} } \times  \mathcal{A}_{B_1\cup B_2}^{k} [\phi]({\mathbf{x}} )
=-\nu_{B_2}(\tilde{\mathbf{x}} )\cdot\left ( -\nabla_{\tilde{\mathbf{x}} } \right )  \times\int_{\partial B_1\cup \partial B_2} {\frac{-e^{\mathrm{i}k|-\tilde{\mathbf{x}} -\mathbf{y}|}}{4 \pi |-\tilde{\mathbf{x}} -\mathbf{y}|}  \phi(\mathbf{y}) d \sigma(\mathbf{y})}.
\end{equation*}
\begin{equation*}
\nu_{B_1}({\mathbf{x}} ) \times  \mathcal{A}_{B_1\cup B_2}^{k} [\phi]({\mathbf{x}} )
=-\nu_{B_2}(\tilde{\mathbf{x}} ) \times\int_{ \partial B_1\cup \partial B_2}{\frac{-e^{\mathrm{i}k|-\tilde{\mathbf{x}} -\mathbf{y}|}}{4 \pi |-\tilde{\mathbf{x}} -\mathbf{y}|}  \phi(\mathbf{y}) d \sigma(\mathbf{y})}.
\end{equation*}
\begin{equation*}
\nu_{B_1}({\mathbf{x}} )\cdot \mathcal{A}_{B_2}^{k} [\phi]({\mathbf{x}} )
=-\nu_{B_2}(\tilde{\mathbf{x}} )\cdot\int_{ \partial B_2}{\frac{-e^{\mathrm{i}k|-\tilde{\mathbf{x}} -\mathbf{y}|}}{4 \pi |-\tilde{\mathbf{x}} -\mathbf{y}|}  \phi(\mathbf{y}) d \sigma(\mathbf{y})}.
\end{equation*}
Changing $\mathbf{y}$ by $\tilde{\mathbf{y}}=-\mathbf{y}$ in the integral, we get
\begin{equation*}
\nu_{B_1}({\mathbf{x}} )\cdot\nabla_{{\mathbf{x}} } \times  \mathcal{A}_{B_1\cup B_2}^{k} [\phi]({\mathbf{x}} )
=\nu_{B_2}(\tilde{\mathbf{x}} )\cdot\nabla_{\tilde{\mathbf{x}} }   \times\int_{\partial B_1\cup \partial B_2}{\frac{-e^{\mathrm{i}k|\tilde{\mathbf{x}} -\tilde{\mathbf{y}} |}}{4 \pi |\tilde{\mathbf{x}} -\tilde{\mathbf{y}} |}  \tilde{\phi} (\tilde{\mathbf{y}} ) d \sigma(\tilde{\mathbf{y}} )},
\end{equation*}
\begin{equation*}
\nu_{B_1}({\mathbf{x}} ) \times  \mathcal{A}_{B_1\cup B_2}^{k} [\phi]({\mathbf{x}} )
=-\nu_{B_2}(\tilde{\mathbf{x}} ) \times\int_{\partial B_1\cup \partial B_2}{\frac{-e^{\mathrm{i}k|\tilde{\mathbf{x}} -\tilde{\mathbf{y}} |}}{4 \pi |\tilde{\mathbf{x}} -\tilde{\mathbf{y}} |}  \tilde{\phi} (\tilde{\mathbf{y}} ) d \sigma(\tilde{\mathbf{y}} )},
\end{equation*}
\begin{equation*}
\nu_{B_1}({\mathbf{x}} )\cdot \mathcal{A}_{B_2}^{k} [\phi]({\mathbf{x}} )
=-\nu_{B_2}(\tilde{\mathbf{x}} )\cdot\int_{ \partial B_1}{\frac{-e^{\mathrm{i}k|\tilde{\mathbf{x}} -\mathbf{y}|}}{4 \pi |\tilde{\mathbf{x}} -\mathbf{y}|}  \tilde{\phi} (\tilde{\mathbf{y}} ) d \sigma(\tilde{\mathbf{y}} )},
\end{equation*}
which verify the results.
\end{proof}
\end{prop}
\begin{rem}
	There are
	\begin{equation}\label{nuA2}
	\begin{aligned}
	\nu_{B_j}(\mathbf{x})\cdot\nabla \times  \acute{\mathcal{A}} _{B_1\cup B_2}^{i} [\phi](\mathbf{x})=\nu_{B_{j}}(\tilde{\mathbf{x}} )\cdot\nabla \times \acute{\mathcal{A}} _{B_1\cup B_2}^{i}  [\tilde{ \phi}    ] (\tilde{\mathbf{x}} ),  i=0,1,2,\cdots, j=1,2.
	\end{aligned}
	\end{equation}
	\begin{equation}\label{nukerA}
	\begin{aligned}
	\nu_{B_1}(\mathbf{x})\times  \acute{\mathcal{A}} _{B_1\cup B_2}^{i} [\phi](\mathbf{x})=-\nu_{B_{2}}(\tilde{\mathbf{x}} )\times \acute{\mathcal{A}} _{B_1\cup B_2}^{i}  [\tilde{ \phi}    ] (\tilde{\mathbf{x}} ), \quad i=0,1,2,\cdots.
	\end{aligned}
	\end{equation}
	\begin{equation}\label{nucdotA2}
	\nu_{B_1}(\mathbf{x})\cdot \acute{\mathcal{A}}_{B_2}^{i} [\phi](\mathbf{x})=-\nu_{B_2}(\tilde{\mathbf{x}} )\cdot \acute{\mathcal{A}}_{B_1}^{i}  [\tilde{ \phi}    ] (\tilde{\mathbf{x}} ), \quad i=0,1,2,\cdots.
	\end{equation}
\end{rem}

\begin{prop}
	For  $\mathbf{x} \in \partial B_1$, $\phi \in L^2_T(\partial D)$, $\varphi \in L^2(\partial D)$, we have
	\begin{equation}\label{nuSi}
	\nu_{B_1}(\mathbf{x})\cdot\nabla  \acute{{S}} _{B_2}^{i} [\varphi](\mathbf{x})=\nu_{B_2}(\tilde{\mathbf{x}} )\cdot\nabla \acute{{S}} _{B_1}^{i}  [\tilde{ \varphi}   ] (\tilde{\mathbf{x}} ), \quad i=0,1,2,\cdots.
	\end{equation}
	\begin{equation}\label{nuSi2}
	\nu_{B_1}(\mathbf{x})\cdot\nabla  \acute{{S}} _{B_1}^{i} [\varphi](\mathbf{x})=\nu_{B_2}(\tilde{\mathbf{x}} )\cdot\nabla \acute{{S}} _{B_2}^{i}  [\tilde{ \varphi}   ] (\tilde{\mathbf{x}} ), \quad i=0,1,2,\cdots.
	\end{equation}
	\begin{equation}\label{nuKi}
	\left ( {K}_{B_1}^{i}  \right ) ^*\left [ \varphi  \right ] \left ( \mathbf{x} \right ) =
	\left ( {K}_{B_2}^{i}  \right ) ^* [\tilde{ \varphi}    ] \left ( \tilde{\mathbf{x}}  \right ), \quad i=0,1,2,\cdots.
	\end{equation}
	\begin{equation}\label{nuMi}
	\acute{\mathcal{M}} _{B_1}^{i} \left [ \phi  \right ] \left ( \mathbf{x} \right ) =
	\acute{\mathcal{M}} _{B_2}^{i}  [\tilde{ \phi}    ] \left ( \tilde{\mathbf{x}}  \right ), \quad i=0,1,2,\cdots.
	\end{equation}
	\begin{proof}
		The proof is similar to that in Proposition \ref{nunu}.
	\end{proof}

\end{prop}

\begin{lem}
	For any $\mathbf{x} \in \partial B_1$, let $\tilde{\mathbf{x}}=-\mathbf{x} \in \partial B_2$, we have
	\begin{equation}\label{nusur1}
	\nabla _{\partial B_{1}}\cdot \psi^0(\mathbf{x})=-\nabla _{\partial B_{2}}\cdot \psi^0(\tilde{\mathbf{x}} ).
	\end{equation}
	\begin{proof}
	For $\mathbf{x} \in \partial B_1$,
\eqref{boundary_e_0} follows that
\begin{equation}\label{nue0}
\nu_{B_1}(\mathbf{x})\cdot \bE_0(\mathbf{x})|_{-}=\nu_{B_2}(\tilde{\mathbf{x}} )\cdot \bE_0(\tilde{\mathbf{x}})|_{-}=0.
\end{equation}
According to the expression of $\bE_0$ provided in \eqref{represent_e0}, we obtained the following expression:
\begin{equation}\label{nue01}
\begin{aligned}
\nu_{B_1}(\mathbf{x})\cdot \bE_0(\mathbf{x})|_{-}=&\nu_{B_1}(\mathbf{x})\cdot\nabla \times  \acute{\mathcal{A}}_{B_1}^{0} [\phi^0](\mathbf{x})+\nu_{B_1}(\mathbf{x})\cdot\nabla \times  \acute{\mathcal{A}}_{B_2}^{0} [\phi^0](\mathbf{x})\\
&+\nu_{B_1}(\mathbf{x})\cdot\nabla  \acute{{S}} _{B_2}^{0} [\g_{\dr {B_2}}\cdot [\psibf^0] ](\mathbf{x})+\left (- \frac{{I} }{2}+\left ( {K}_{B_1}^0  \right ) ^*\right ) \left [ \g_{\dr {B_1}}\cdot [\psibf^0] \right ],
\end{aligned}
\end{equation}
\begin{equation}\label{nue02}
\begin{aligned}
\nu_{B_2}(\tilde{\mathbf{x}} )\cdot \bE_0(\tilde{\mathbf{x}} )|_{-}=&
\nu_{B_2}(\tilde{\mathbf{x}} )\cdot\nabla \times  \acute{\mathcal{A}}_{B_1}^{0} [\phi^0](\tilde{\mathbf{x}} )
+\nu_{B_2}(\tilde{\mathbf{x}} )\cdot\nabla \times  \acute{\mathcal{A}}_{B_2}^{0} [\phi^0](\tilde{\mathbf{x}} )\\
&+\nu_{B_2}(\tilde{\mathbf{x}} )\cdot\nabla  {S}_{B_1}^{0} [\g_{\dr {B_1}}\cdot [\psibf^0]](\tilde{\mathbf{x}} )
+\left ( -\frac{{I} }{2}
+\left ( {K}_{B_2}^0  \right ) ^*\right ) \left [ \g_{\dr {B_2}}\cdot [\psibf^0] \right ].
\end{aligned}
\end{equation}
We will analyze each item in \eqref{nue01} and \eqref{nue02} separately. By \eqref{phi1} and \eqref{nueoin}, we have
\begin{equation}\label{nvphi1}
\begin{aligned}
\phi^0(\mathbf{x})=\nu_{B_1}(\mathbf{x})\times  \bE_0^{i}(\mathbf{x})=\nu_{B_1}(\mathbf{x})\times  \mathbf{p}=-\nu_{B_2}(\tilde{\mathbf{x}} )\times  \mathbf{p}=-\phi^0(\tilde{\mathbf{x}} ).
\end{aligned}
\end{equation}
On one hand, substituting \eqref{nvphi1},\eqref{nuA2} into \eqref{nue0}, we get the result
\begin{equation}\label{key1}
\begin{aligned}
&\nu_{B_1}(\mathbf{x})\cdot\nabla  \acute{{S}} _{B_2}^{0} [\g_{\dr {B_2}}\cdot [\psibf^0] ](\mathbf{x})+\left ( -\frac{{I} }{2}+\left ( {K}_{B_1}^0  \right ) ^*\right ) \left [ \g_{\dr {B_1}}\cdot [\psibf^0] \right ](\mathbf{x})\\
&=-\nu_{B_2}(\tilde{\mathbf{x}} )\cdot\nabla  \acute{{S}} _{B_1}^{0} [\g_{\dr {B_1}}\cdot [\psibf^0]](\tilde{\mathbf{x}} )
-\left ( -\frac{{I} }{2}
+\left ( {K}_{B_2}^0  \right ) ^*\right ) \left [ \g_{\dr {B_2}}\cdot [\psibf^0] \right ](\tilde{\mathbf{x}} ).
\end{aligned}
\end{equation}
On the other hand, by stright computing with \eqref{nuSi} and \eqref{nuKi}, there is
\begin{equation}\label{key2}
\begin{aligned}
&\nu_{B_1}(\mathbf{x})\cdot\nabla  \acute{{S}} _{B_2}^{0} [\g_{\dr {B_2}}\cdot [\psibf^0] ](\mathbf{x})+\left (- \frac{{I} }{2}+\left ( {K}_{B_1}^0  \right ) ^*\right ) \left [ \g_{\dr {B_1}}\cdot [\psibf^0] \right ](\mathbf{x})\\
&=\nu_{B_2}(\tilde{\mathbf{x}} )\cdot\nabla  \acute{{S}} _{B_1}^{0} [\widetilde{\g_{\dr {B_1}}\cdot [\psibf^0]} ](\tilde{\mathbf{x}} )
+\left (- \frac{{I} }{2}
+\left ( {K}_{B_2}^0  \right ) ^*\right ) \left [ \widetilde{\g_{\dr {B_2}}\cdot [\psibf^0]}  \right ](\tilde{\mathbf{x}} ).
\end{aligned}
\end{equation}
By comparing \eqref{key1} and \eqref{key2}, we obtain 
\begin{equation*}
{\g_{\dr {B_1}}\cdot [\psibf^0]}(\mathbf{x})=\widetilde{\g_{\dr {B_2}}\cdot [\psibf^0]}(\tilde{\mathbf{x}}) =-{\g_{\dr {B_2}}\cdot [\psibf^0]}(\tilde{\mathbf{x}} ),
\end{equation*}
which concludes the proof.
\end{proof}
	
\end{lem}
\begin{lem}
	For any $\mathbf{x} \in \partial B_1$, let $\tilde{\mathbf{x}}=-\mathbf{x} \in \partial B_2$, we have
	\begin{equation}\label{nue0d}
	(\nu_{B_1}(\mathbf{x})\cdot \bE_0(\mathbf{x}))^+=-(\nu_{B_2}(\tilde{\mathbf{x}} )\cdot \bE_0(\tilde{\mathbf{x}} ))^+.
	\end{equation}
	\begin{proof}
		For $\mathbf{x} \in \partial B_1$, based on the previous analysis, we can derive that
		\begin{equation*}
		\begin{aligned}
		&(\nu_{B_1}(\mathbf{x})\cdot \bE_0(\mathbf{x}))^{+}\\
		\nm
		=&\nu_{B_1}(\mathbf{x})\cdot\nabla \times  \acute{\mathcal{A}}_{B_1}^{0} [\phi^0](\mathbf{x})+\nu_{B_1}(\mathbf{x})\cdot\nabla \times  \acute{\mathcal{A}}_{B_2}^{0} [\phi^0](\mathbf{x})\\
			\nm
		&+\nu_{B_1}(\mathbf{x})\cdot\nabla  \acute{{S}} _{B_2}^{0} [\g_{\dr {B_2}}\cdot [\psibf^0] ](\mathbf{x})+\left ( \frac{{I} }{2}+\left ( {K}_{B_1}^0  \right ) ^*\right ) \left [ \g_{\dr {B_1}}\cdot [\psibf^0] \right ]\\
		=&-	\nu_{B_2}(\tilde{\mathbf{x}} )\cdot\nabla \times  \acute{\mathcal{A}}_{B_1}^{0} [\phi^0](\tilde{\mathbf{x}} )
		-\nu_{B_2}(\tilde{\mathbf{x}} )\cdot\nabla \times  \acute{\mathcal{A}}_{B_2}^{0} [\phi^0](\tilde{\mathbf{x}} )\\
			\nm
		&-\nu_{B_2}(\tilde{\mathbf{x}} )\cdot\nabla  \acute{{S}} _{B_1}^{0} [\g_{\dr {B_1}}\cdot [\psibf^0]](\tilde{\mathbf{x}} )
		-\left ( \frac{{I} }{2}
		+\left ( {K}_{B_2}^0  \right ) ^*\right ) \left [ \g_{\dr {B_2}}\cdot [\psibf^0] \right ]\\
			\nm
		=&-(\nu_{B_2}(\tilde{\mathbf{x}} )\cdot \bE_0(\tilde{\mathbf{x}} ))^+.
		\end{aligned}
		\end{equation*}
		
	\end{proof}
\end{lem}

\begin{lem}
	For any $\mathbf{x} \in \partial B_1$, $\tilde{\mathbf{x}}=-\mathbf{x} \in \partial B_2$, we have
	\begin{equation}\label{e1e1}
	(\nu_{B_1}(\mathbf{x})\cdot \bE_1(\mathbf{x}))^-=-(\nu_{B_2}(\tilde{\mathbf{x}} )\cdot \bE_1(\tilde{\mathbf{x}} ))^-.
	\end{equation}
	\begin{proof}
		By \eqref{boundary_E}, for all $\mathbf{x} \in \partial D$, there is
		\begin{equation}\label{e0e1}
			(\nu(\mathbf{x})\cdot \bE_0(\mathbf{x}))^+=-\tilde{C} (\nu(\mathbf{x})\cdot \bE_1(\tilde{\mathbf{x}} ))^-.
		\end{equation}
 For $\mathbf{x} \in \partial B_1$, combining \eqref{e0e1} and \eqref{nue0d}, we get the desired result.
		
	\end{proof}
\end{lem}

\begin{lem}\label{lempds1}
	For any $\mathbf{x} \in \partial B_1$, let $\tilde{\mathbf{x}}=-\mathbf{x} \in \partial B_2$, we have
	\begin{equation}\label{surphi1}
	\nabla _{\partial B_{1}}\cdot \psi^1(\mathbf{x})=\nabla _{\partial B_{2}}\cdot \psi^1(\tilde{\mathbf{x}} ).
	\end{equation}
	\begin{proof}
		By using the expression of $\bE_1$ given in \eqref{represent_e1}, we have
	\begin{equation}\label{nve1-1}
	\begin{aligned}
	&(\nu_{B_1}(\mathbf{x})\cdot \bE_1(\mathbf{x}))^-\\
	\nm
	=&\ds \nu_{B_1}(\mathbf{x})\cdot\nabla \times  \acute{\mathcal{A}}_{B_1\cup B_2}^{0} [\phi^1](\mathbf{x})+\tilde{C} \nu_{B_1}(\mathbf{x})\cdot\nabla \times  \acute{\mathcal{A}}_{B_1\cup B_2}^{2} [\phi^0](\mathbf{x})\\
	\nm
	&\ds+\tilde{C} \nu_{B_1}(\mathbf{x})\cdot  \acute{\mathcal{A}}_{B_1}^{2} [\psi^0](\mathbf{x})+\tilde{C} \nu_{B_1}(\mathbf{x})\cdot  \acute{\mathcal{A}}_{ B_2}^{2} [\psi^0](\mathbf{x})\\
	&+\nu_{B_1}(\mathbf{x})\cdot\nabla  \acute{{S}} _{B_2}^{0} [\g_{\dr {B_2}}\cdot[\psibf^1] ](\mathbf{x})
	+\left ( -\frac{{I} }{2}+\left ( {K}_{B_1}^0  \right ) ^*\right ) \left [ \g_{\dr {B_1}}\cdot [\psibf^1] \right ](\mathbf{x})\\
	&+\tilde{C} \nu_{B_1}(\mathbf{x})\cdot\nabla  \acute{{S}} _{B_2}^{2} [\g_{\dr {B_2}}\cdot[\psibf^0] ](\mathbf{x})
	+\tilde{C} \left ( -\frac{{I} }{2}+\left ( {K}_{B_1}^0  \right ) ^*\right ) \left [ \g_{\dr {B_1}}\cdot [\psibf^0] \right ](\mathbf{x}).\\
	\end{aligned}
	\end{equation}
	and
		\begin{equation}\label{nve1-2}
	\begin{aligned}
	&(\nu_{B_2}(\tilde{\mathbf{x}})\cdot \bE_1(\tilde{\mathbf{x}}))^-\\
	\nm
	=&\nu_{B_2}(\tilde{\mathbf{x}})\cdot\nabla \times  \acute{\mathcal{A}}_{B_1\cup B_2}^{0} [\phi^1](\tilde{\mathbf{x}})+\tilde{C} \nu_{B_2}(\tilde{\mathbf{x}})\cdot\nabla \times  \acute{\mathcal{A}}_{B_1\cup B_2}^{2} [\phi^0](\tilde{\mathbf{x}})\\
	\nm
	&+\tilde{C} \nu_{B_2}(\tilde{\mathbf{x}})\cdot  \acute{\mathcal{A}}_{B_1}^{2} [\psi^0](\tilde{\mathbf{x}})+\tilde{C} \nu_{B_2}(\tilde{\mathbf{x}})\cdot  \acute{\mathcal{A}}_{ B_2}^{2} [\psi^0](\tilde{\mathbf{x}})\\
	\nm
	&+\nu_{B_2}(\tilde{\mathbf{x}})\cdot\nabla  \acute{{S}} _{B_1}^{0} [\g_{\dr {B_1}}\cdot[\psibf^1] ](\tilde{\mathbf{x}})
	+\left ( -\frac{{I} }{2}+\left ( {K}_{B_2}^0  \right ) ^*\right ) \left [ \g_{\dr {B_2}}\cdot [\psibf^1] \right ](\tilde{\mathbf{x}})\\
	&+\tilde{C} \nu_{B_2}(\tilde{\mathbf{x}})\cdot\nabla  \acute{{S}} _{B_1}^{2} [\g_{\dr {B_1}}\cdot[\psibf^0] ](\tilde{\mathbf{x}})
	+\tilde{C} \left ( -\frac{{I} }{2}+\left ( {K}_{B_2}^0  \right ) ^*\right ) \left [ \g_{\dr {B_2}}\cdot [\psibf^0] \right ](\tilde{\mathbf{x}}).\\
	\end{aligned}
	\end{equation}
 We will analyze each term in formula \eqref{nve1-1} and \eqref{nve1-2} individually.
Recall from \eqref{nusur1} that
\begin{equation}\label{nuphi1}
\nabla _{\partial B_{1}({\mathbf{x}} )}\cdot \psi^0(\mathbf{x}) =-\nabla _{\partial B_{2}(\tilde{\mathbf{x}} )}\cdot \psi^0(\tilde{\mathbf{x}} ).
\end{equation}
From another perspective, by direct calculation, we obtain
\begin{equation}\label{nuphi2}
\nabla _{\partial B_{1}({\mathbf{x}} )}\cdot \psi^0(\mathbf{x}) =-\nabla _{\partial B_{2}(\tilde{\mathbf{x}} )}\cdot\widetilde{ \psi^0} (\tilde{\mathbf{x}} ),
\end{equation}
then by comparing  \eqref{nuphi1} and \eqref{nuphi2}, we obtain that
\begin{equation}\label{nuphi}
\psi^0(\tilde{\mathbf{x}} ) =\widetilde{ \psi^0} (\tilde{\mathbf{x}} ).
\end{equation}
Furthermore,  by using \eqref{nucdotA2}, we have
\begin{equation}\label{e-11}
\begin{aligned}
&\nu_{B_1}(\mathbf{x})\cdot  \acute{\mathcal{A}}_{B_1}^{2} [\psi^0](\mathbf{x})
+\nu_{B_1}(\mathbf{x})\cdot  \acute{\mathcal{A}}_{ B_2}^{2} [\psi^0](\mathbf{x})\\
\nm
&=-\nu_{B_2}(\tilde{\mathbf{x}})\cdot  \acute{\mathcal{A}}_{B_1}^{2} [\psi^0](\tilde{\mathbf{x}})
-\nu_{B_2}(\tilde{\mathbf{x}})\cdot  \acute{\mathcal{A}}_{ B_2}^{2} [\psi^0](\tilde{\mathbf{x}}).
\end{aligned}
\end{equation}

Consider the other parts of \eqref{nve1-1} and \eqref{nve1-2}. By using  \eqref{phi3}, there is
\begin{equation*}
\begin{aligned}
&\nu_{B_1}(\mathbf{x})\cdot\nabla \times  \acute{\mathcal{A}}_{B_1\cup B_2}^{0} [\phi^1](\mathbf{x}) \\
\nm
&=\ds\nu_{B_1}(\mathbf{x})\cdot\nabla \times  \acute{\mathcal{A}}_{B_1\cup B_2}^{0} [{\nu}\times \bE^{i}_1](\mathbf{x})-\tilde{C} \nu_{B_1}(\mathbf{x})\cdot\nabla \times  \acute{\mathcal{A}}_{B_1\cup B_2}^{0} [\acute{\mathcal{M}} _D^{2}\left [ \phibf^0 \right ]](\mathbf{x})\\
\nm
&\quad-\tilde{C} \nu_{B_1}(\mathbf{x})\cdot\nabla \times  \acute{\mathcal{A}}_{B_1\cup B_2}^{0}\left[\nu \times \acute{\mathcal{A}}_D ^0\left[\psi^0 \right]\right]\\
\nm
&\quad-\tilde{C} \nu_{B_1}(\mathbf{x})\cdot\nabla \times  \acute{\mathcal{A}}_{B_1\cup B_2}^{0} [\nu \times \g\acute{{S}} _D^2[\g_{\dr D} \cdot \psi^0]](\mathbf{x}).
\end{aligned}
\end{equation*}
Utilizing \eqref{kere1in} and \eqref{nuA2}, we obtain
\begin{equation}\label{e-12}
\nu_{B_1}(\mathbf{x})\cdot\nabla \times  \acute{\mathcal{A}}_{B_1\cup B_2}^{0} [{\nu}\times \bE^{i}_1](\mathbf{x})=\nu_{B_2}(\tilde{\mathbf{x}} )\cdot\nabla \times \acute{\mathcal{A}}_{B_1\cup B_2}^{0} [{\nu}\times \bE^{i}_1](\tilde{\mathbf{x}} ).
\end{equation}
By \eqref{nuMi}, \eqref{nvphi1} and \eqref{nuA2}, we have
\begin{equation}\label{kerAM}
\nu_{B_1}(\mathbf{x})\cdot\nabla \times  \acute{\mathcal{A}}_{B_1\cup B_2}^{0} [\acute{\mathcal{M}} _D^{2}\left [ \phibf^0 \right ]](\mathbf{x})= -\nu_{B_2}(\tilde{\mathbf{x}} )\cdot\nabla \times  \acute{\mathcal{A}}_{B_1\cup B_2}^{0} \left[\acute{\mathcal{M}} _D^{2}\left [ \phi^0 \right ]\right](\tilde{\mathbf{x}} ).
\end{equation}
By \eqref{nuphi} and \eqref{nukerA}, we have
\begin{equation}\label{key}
\nu_{B_1}(\mathbf{x})\cdot\nabla \times  \acute{\mathcal{A}}_{B_1\cup B_2}^{0}\left[\nu \times \acute{\mathcal{A}}_D ^0\left[\psi^0 \right]\right](\mathbf{x})
=-\nu_{B_2}(\tilde{\mathbf{x}} )\cdot\nabla \times  \acute{\mathcal{A}}_{B_1\cup B_2}^{0}\left[\nu \times \acute{\mathcal{A}}_D ^0\left[\psi^0 \right]\right](\tilde{\mathbf{x}} ),
\end{equation}
and by \eqref{nusur1}, \eqref{nuSi}, \eqref{nuSi2} and \eqref{nuA2}, we have
\begin{equation}\label{kerAGrS}
\nu_{B_1}(\mathbf{x})\cdot\nabla \times  \acute{\mathcal{A}}_{B_1\cup B_2}^{0} [\nu \times \g\acute{{S}} _D^2[\g_{\dr D} \cdot \psi^0]](\mathbf{x})
= -\nu_{B_2}(\tilde{\mathbf{x}} )\cdot\nabla \times  \acute{\mathcal{A}}_{B_1\cup B_2}^{0} [\nu \times \g\acute{{S}} _D^2[\g_{\dr D} \cdot \psi^0]](\tilde{\mathbf{x}} ).
\end{equation}
Similarly, by using \eqref{nuA2} and \eqref{nvphi1}, there is
\begin{equation}\label{e-13}
\nu_{B_1}(\mathbf{x})\cdot\nabla \times  \acute{\mathcal{A}}_{B_1\cup B_2}^{2} [\phi^0](\mathbf{x})=\nu_{B_2}(\tilde{\mathbf{x}} )\cdot\nabla \times \acute{\mathcal{A}}_{B_1\cup B_2}^{2} [{\phi^0}](\tilde{\mathbf{x}} ),
\end{equation}
and by \eqref{nuSi}, \eqref{nuKi} and \eqref{nusur1}, there is
\begin{equation}\label{e-14}
\nu_{B_1}(\mathbf{x})\cdot\nabla  \acute{{S}} _{B_2}^{2} [\g_{\dr {B_2}}\cdot[\psibf^0] ](\mathbf{x})=-\nu_{B_2}(\tilde{\mathbf{x}})\cdot\nabla  \acute{{S}} _{B_1}^{2} [\g_{\dr {B_1}}\cdot[\psibf^0] ](\tilde{\mathbf{x}}),
\end{equation}
\begin{equation}\label{e-15}
\left ( -\frac{{I} }{2}+\left ( {K}_{B_1}^0  \right ) ^*\right ) \left [ \g_{\dr {B_1}}\cdot [\psibf^0] \right ]=-\left ( -\frac{{I} }{2}+\left ( {K}_{B_2}^0  \right ) ^*\right ) \left [ \g_{\dr {B_2}}\cdot [\psibf^0] \right ](\tilde{\mathbf{x}}).
\end{equation}
Therefore, by \eqref{e1e1}, combining with \eqref{e-11}, \eqref{e-12},  \eqref{kerAM}, \eqref{key}, \eqref{kerAGrS}, \eqref{e-13}, \eqref{e-14}, \eqref{e-15}, we can obtain that
    \begin{equation*}
    \begin{aligned}
    &\nu_{B_1}(\mathbf{x})\cdot\nabla \times  \acute{\mathcal{A}}_{B_1\cup B_2}^{0} [{\nu}\times \bE^{i}_1](\mathbf{x})\\
    &+\nu_{B_1}(\mathbf{x})\cdot\nabla  \acute{{S}} _{B_2}^{0} [\g_{\dr {B_2}}\cdot[\psibf^1] ](\mathbf{x})
    +\left ( -\frac{{I} }{2}+\left ( {K}_{B_1}^0  \right ) ^*\right ) \left [ \g_{\dr {B_1}}\cdot [\psibf^1] \right ](\mathbf{x})\\
   =&-\nu_{B_2}(\tilde{\mathbf{x}})\cdot\nabla \times  \acute{\mathcal{A}}_{B_1\cup B_2}^{0} [{\nu}\times \bE^{i}_1](\tilde{\mathbf{x}})\\
   &-\nu_{B_2}(\tilde{\mathbf{x}})\cdot\nabla  \acute{{S}} _{B_1}^{0} [\g_{\dr {B_1}}\cdot[\psibf^1] ](\tilde{\mathbf{x}})
   -\left ( -\frac{{I} }{2}+\left ( {K}_{B_2}^0  \right ) ^*\right ) \left [ \g_{\dr {B_2}}\cdot [\psibf^1] \right ](\tilde{\mathbf{x}}),
    \end{aligned}
    \end{equation*}
    which is equivalent to
    \begin{equation*}
    \nu(\mathbf{x})\cdot\nabla \times  \acute{\mathcal{A}}_{D}^{0} [{\nu}\times \bE^{i}_1](\mathbf{x})=\nu(\mathbf{x})\cdot\nabla  \acute{{S}} _{D}^{0} [\g_{\dr {D}}\cdot[\psibf^1] ](\mathbf{x})|_{-},\quad \mathbf{x}\in \dr D.
    \end{equation*}
    Therefore, for $\mathbf{x}\in \dr B_1$, there is
     \begin{equation}\label{a-s3}
    \begin{aligned}
    &\nu_{B_1}(\mathbf{x})\cdot\nabla  \acute{{S}} _{B_2}^{0} [\g_{\dr {B_2}}\cdot[\psibf^1] ](\mathbf{x})
    +\left ( -\frac{{I} }{2}+\left ( {K}_{B_1}^0  \right ) ^*\right ) \left [ \g_{\dr {B_1}}\cdot [\psibf^1] \right ](\mathbf{x})\\
    \nm
    &\ds=\nu_{B_1}(\mathbf{x})\cdot\nabla \times  \acute{\mathcal{A}}_{B_1\cup B_2}^{0} [{\nu}\times \bE^{i}_1](\mathbf{x})\\
    \nm
    &\ds=\nu_{B_2}(\tilde{\mathbf{x}})\cdot\nabla \times  \acute{\mathcal{A}}_{B_1\cup B_2}^{0} [{\nu}\times \bE^{i}_1](\tilde{\mathbf{x}})\\
    &=\nu_{B_2}(\tilde{\mathbf{x}})\cdot\nabla  \acute{{S}} _{B_1}^{0} [\g_{\dr {B_1}}\cdot[\psibf^1] ](\tilde{\mathbf{x}})
    +\left ( -\frac{{I} }{2}+\left ( {K}_{B_2}^0  \right ) ^*\right ) \left [ \g_{\dr {B_2}}\cdot [\psibf^1] \right ](\tilde{\mathbf{x}}).
    \end{aligned}
    \end{equation}
    On the other hand, by stright computing, there is
    \begin{equation}\label{a-s4}
    \begin{aligned}
    &\nu_{B_1}(\mathbf{x})\cdot\nabla  \acute{{S}} _{B_2}^{0} [\g_{\dr {B_2}}\cdot [\psibf^1] ](\mathbf{x})+\left (- \frac{{I} }{2}+\left ( {K}_{B_1}^0  \right ) ^*\right ) \left [ \g_{\dr {B_1}}\cdot [\psibf^1] \right ](\mathbf{x})\\
    &=\nu_{B_2}(\tilde{\mathbf{x}} )\cdot\nabla  \acute{{S}} _{B_1}^{0} [\widetilde{\g_{\dr {B_1}}\cdot [\psibf^1]} ](\tilde{\mathbf{x}} )
    +\left ( -\frac{{I} }{2}
    +\left ( {K}_{B_2}^0  \right ) ^*\right ) \left [ \widetilde{\g_{\dr {B_2}}\cdot [\psibf^1]}  \right ](\tilde{\mathbf{x}} ).
    \end{aligned}
    \end{equation}
    By comparing  \eqref{a-s3} and \eqref{a-s4}, we can write that
    \begin{equation*}
    {\g_{\dr {B_1}}\cdot [\psibf^1]}(\mathbf{x})=\widetilde{\g_{\dr {B_2}}\cdot [\psibf^1]}(\tilde{\mathbf{x}}) ={\g_{\dr {B_2}}\cdot [\psibf^1]}(\tilde{\mathbf{x}} ),
    \end{equation*}
     which concludes the proof.
\end{proof}
\end{lem}
Based on above analysis, for $ \mathbf{x} \in \mathbb{R}^3 \setminus \overline{{D}}$,  we have the follow results:
\begin{lem}\label{lemgraS}
For any bounded domain $\Omega$ containing $B_1$ and $B_2$, there exists a  positive constant C depending only on $r$ and $\|\partial \Omega \|_{C^2}$, such that
	\begin{equation*}
\left \| \nabla \acute{{S}} _D^{0}[\g_{\dr D}\cdot [\psibf^1]] \right \|_{L^\infty(\Omega\backslash \overline{D})}\le C.
	\end{equation*}
	\begin{proof}
		For any $\mathbf{x} \in \partial B_1$, let $\tilde{\mathbf{x}}=-\mathbf{x} \in \partial B_2$, by using \eqref{surphi1} in Lemma \ref{lempds1}, we have
		\begin{equation*}
		\begin{aligned}
		&  \acute{{S}} _D^{0}[\g_{\dr D}\cdot [\psi^1]]\left ( \mathbf{x} \right )  \\
		\nm
		& =\int_{\partial {B_1}}{\frac{-1}{4 \pi |{\mathbf{x}} -\mathbf{y}|}  \g_{\dr B_1}\cdot [\psi^1](\mathbf{y}) d \sigma(\mathbf{y})}
		+\int_{\partial {B_2}}{\frac{-1}{4 \pi |{\mathbf{x}} -\mathbf{y}|}  \g_{\dr B_2}\cdot [\psi^1](\mathbf{y}) d \sigma(\mathbf{y})}
		\\
		\nm
		& \overset{x=-\tilde{x} }{=}  \int_{\partial {B_1}}{\frac{-1}{4 \pi |-\tilde{\mathbf{x}} -\mathbf{y}|}  \g_{\dr B_1}\cdot [\psi^1](\mathbf{y}) d \sigma(\mathbf{y})}
		+\int_{\partial {B_2}}{\frac{-1}{4 \pi |-\tilde{\mathbf{x}} -\mathbf{y}|}   \g_{\dr B_2}\cdot [\psi^1](\mathbf{y}) d \sigma(\mathbf{y})}
		\\
		\nm
		& \overset{y=-\tilde{y} }{=}  \int_{\partial {B_2}}{\frac{-1}{4 \pi |\tilde{\mathbf{x}} -\tilde{\mathbf{y}} |}   \g_{\dr B_2}\cdot [\psi^1](\tilde{\mathbf{y}}) d \sigma(\tilde{\mathbf{y}})}
		+\int_{\partial {B_1}}{\frac{-1}{4 \pi |\tilde{\mathbf{x}} -\tilde{\mathbf{y}} |}  \g_{\dr B_1}\cdot [\psi^1](\tilde{\mathbf{y}}) d \sigma(\tilde{\mathbf{y}})}\\
		\nm
		&= \acute{{S}} _D^{0}[\g_{\dr D}\cdot [\psi^1]]\left ( \tilde{\mathbf{x}}  \right ).  \\
		\end{aligned}
		\end{equation*}
		Thus, we have
		\begin{equation*}
		\acute{{S}} _D^{0}[\g_{\dr D}\cdot \psi^1]\big| _{\partial B_1}=\acute{{S}} _D^{0}[\g_{\dr D}\cdot \psi^1]\big| _{\partial B_2}.
		\end{equation*}
		That is,  $\acute{{S}} _D^{0}[\g_{\dr D}\cdot [\psi^1]]$  has no potential gap on  $\partial B_{1}$  and  $\partial B_{2}$, from which one can infer that  $\nabla \acute{{S}} _D^{0}[\g_{\dr D}\cdot [\psi^1]]$  is bounded in  $\mathbb{R}^{3} \backslash\left(B_{1} \cup B_{2}\right)$  (this can be proved by following the same lines of the proof of Theorem 2.1 in \cite{KLY13}).
	\end{proof}
\end{lem}

Then, consider the curl item $\nabla \times \acute{\mathcal{A}}_{D}^{0}
[\phibf^1](\mathbf{x})$ in $\bE_1$. We have the following  estimation result.
\begin{lem}\label{lemkerA}
For any bounded domain $\Omega$ containing $B_1$ and $B_2$, there exists a  positive constant C depending only on $r$ and $\|\partial \Omega \|_{C^2}$, such that
\begin{equation*}
\left \| \nabla \times \acute{\mathcal{A}}_{D}^{0}
[\phibf^1] \right \|_{L^\infty(\Omega\backslash \overline{D})}\le C.
\end{equation*}

\begin{proof}
	To begin with, let's prove the boundedness of  $\|{\phi^1}\|_{\mathrm{TH}({\rm div}, \dr D)}$. Recall that
		\begin{equation*}
		\phibf^1={\nu}\times \bE^{i}_1- \tilde{C} \acute{\mathcal{M}} _D^{2}\left [ \phibf^0 \right ]-\tilde{C} \nu \times \acute{\mathcal{A}}_D ^0\left[\psi^0 \right]-\tilde{C} \nu \times \g\acute{{S}} _D^2[\g_{\dr D} \cdot \psi^0].
		\end{equation*}
		We will estimate $\phi_1$ term by term. From the analyticity of $\bE_1^{i}$ and $\nabla {\times \bE_1^{i}}$, it can be easily shown that the the first part
		$$\|{\nu \times \bE_1^{i}}\|_{\mathrm{TH}({\rm div}, \dr D)}=\|{\nu \times \bE_1^{i}}\|_{L^2(\dr D)}+\|\nu \cdot(\nabla {\times \bE_1^{i}})\|_{L^2(\dr D)}$$  is bounded.
	 For the second part, since $\phi^0=\nu\times \bE_0^{i}$ is analytic, and combining this with the fact that $\mathcal{M}_D^k:
		\mathrm{L}_T^2 (\partial D)  \longrightarrow \mathrm{L}_T^2 (\partial D)$ is a bounded operator, we can conclude that $\|{\acute{\mathcal{M}} _D^{2}\left [ \phibf^0 \right ]}\|_{\mathrm{TH}({\rm div}, \dr D)}$ is bounded. Before proving the last two terms, let's analyze the properties of their kernel functions $\psi^0$ and $\g_{\dr D} \cdot \psi^0$. Recall that
		\begin{equation}\label{phi4}
		 \tilde{C} \left ( \frac{{\mathcal{I} }}{2}\mathcal  +\acute{\mathcal{M}} _D^{0} \right ) \left [ \psi^0 \right ]=	\nu(\mathbf{x}) \times \acute{\mathcal{A}}_D^0\left [ \phi^0 \right ] +i\nu\times \bH_0^{i},
		\end{equation}
the operator $\nu(\mathbf{x}) \times \acute{\mathcal{A}}_D^0$ is continuous on $\partial D$, and $\bH_0^{i}$ is analytic. Therefore, the right-hand side of \eqref{phi4} is bounded. We note that $-1/2$ is not an eigenvalue of $\mathcal{M}_D^{k}$ (cf. \cite{RG}). Therefore, using equation \eqref{phi4}, we can deduce that $\psi^0$ is bounded.
 Taking surface divergence in \eqref{phi4}, we get
\begin{equation}\label{phi4sur}
 \tilde{C} \left ( \frac{{I} }{2}-\left ( {K}_{D}^0  \right ) ^*\right ) \left [ \g_{\dr {D}}\cdot [\psi^0] \right ]=\nu\cdot\nabla\times\mathcal{A}_D^0[\phi^0],
\end{equation}
	by using the continuity of $\nu\cdot\nabla\times\mathcal{A}_D^0$, we can conclude that the right-hand side of equation \eqref{phi4sur} is bounded. Furthermore, note that
		\begin{equation*}
		\int_{\partial D} \nabla _{\partial D}\cdot  \psi^0=-\int_{\partial D} \psi^0\cdot\nabla _{\partial D} [1]=0.
		\end{equation*}
		As a result, we can conclude that $\frac{{I}}{2}- {K}_{D}^0$ is an invertible operator. Consequently, from equation \eqref{phi4sur}, we can deduce that $\g_{\dr D} \cdot \psi^0$ is bounded. The operators $\nu(\mathbf{x}) \times \acute{\mathcal{A}}_D^0$ and $\nu \times \g\acute{{S}} _D^2$ are continuous on $\partial D$. Therefore, we can conclude that the third and fourth parts of $\phi^1$ are bounded. Using the analysis above, we can deduce that $\|{\phi^1}\|_{\mathrm{TH}({\rm div}, \dr D)}$ is bounded.
	
	Next, we will analyze the boundedness of $\nabla \times \acute{\mathcal{A}}_{D}^{0}[\phibf^1]$. Let $f(\mathbf{x}):=\nabla \times \acute{\mathcal{A}}_{D}^{0}[\phibf^1](\mathbf{x})$,  it satisfies the following system:
		\begin{equation}\label{fx}
		\left \{
		\begin{aligned}
		&\Delta f(\mathbf{x}) = 0, \hspace*{9.74cm} \mbox{in} \quad \RR^3\backslash \overline{D} \,\\
		&\ds \nu_{B_1}(\mathbf{x})\times f(\mathbf{x})\mid_{\partial B_1}^+
		=\left ( \frac{{-\mathcal{I} }}{2}\mathcal  +\acute{\mathcal{M}} _{B_1}^{0} \right ) \left [ \phi^1 \right ](\mathbf{x})+
		\nu_{B_1}(\mathbf{x}) \times\nabla \times \acute{\mathcal{A}}_{B_2}^0\left [ \phi^1 \right ](\mathbf{x}),  \hspace*{0cm}\quad x\in  \partial B_1,\\
		&\nu_{B_2}(\mathbf{x})\times f(\mathbf{x})\mid_{\partial B_2}^+
		=\left ( \frac{{-\mathcal{I} }}{2}\mathcal  +\acute{\mathcal{M}} _{B_2}^{0} \right ) \left [ \phi^1 \right ](\mathbf{x})+
		\nu_{B_2}(\mathbf{x}) \times\nabla \times \acute{\mathcal{A}}_{B_1}^0\left [ \phi^1 \right ](\mathbf{x}),\hspace*{0cm}\quad x\in  \partial B_2,\\
		&\nu\cdot f(\mathbf{x})=\nu \cdot\nabla \times \acute{\mathcal{A}}_{D}^{0}[\phibf^1](\mathbf{x}),\hspace*{7.4cm}\quad \text{on}\quad   \partial D,\\
		&f(\mathbf{x})=O\left(|\mathbf{x}|^{-2}\right), \hspace*{7.64cm}\qquad \text { as }|\mathbf{x}| \rightarrow \infty.
		\end{aligned}
		\right.
		\end{equation}
We observe from equation \eqref{fx} that both $\nu \times f$ and $\nu \cdot f$ are bounded on the boundary $\partial D$. By applying the maximum principle (\cite{MP}), we can conclude that $\left \| f\right \| _{L^\infty(\Omega\backslash \overline{D})}=\left \| \nabla \times \acute{\mathcal{A}}_{D}^{0}[\phi^1] \right \|_{L^\infty(\Omega\backslash \overline{D})}$ is also bounded. Hence, we have completed the proof.

\end{proof}
\end{lem}

Based on the above analysis operators and functions, we have obtained the following estimation results for $\bE_1$:
	\begin{lem}\label{E1_esti}
		For any bounded domain $\Omega$ containing $B_1$ and $B_2$, there has a  positive constant C depending only on $r$ and $\|\partial \Omega \|_{C^2}$, such that
		\begin{equation}\label{E1_esti1}
		\begin{aligned}
		\|\bE_1\|_{L^\infty(\Omega\backslash \overline{D})} \le C.
		\end{aligned}
		\end{equation}
	\begin{proof}

Using the representation of $\bE_1$ in Equation \eqref{represent_e1}, we obtain 
 \begin{equation*}
 \begin{aligned}
& \left \| \bE_1\right \| _{L^{\infty }\left ( \Omega\backslash \overline{D} \right ) }  \\
&= \left \| \bE^{i}_1 +  \nabla \times \acute{\mathcal{A}} _{D}^{0}
 [\phibf^1] +\nabla\acute{{S}} _D^{0}[\g_{\dr D}\cdot [\psibf^1]] \right \| _{L^{\infty }\left ( \Omega\backslash \overline{D} \right ) }
 \\
 &\le \left \| \bE_1^{i}\right \| _{L^{\infty }\left ( \Omega\backslash \overline{D} \right ) }+\left \| \nabla \times \acute{\mathcal{A}}_{D}^{0}
 [\phibf^1] \right \|_{L^\infty(\Omega\backslash \overline{D})}+\left \| \nabla \acute{{S}} _D^{0}[\g_{\dr D}\cdot [\psibf^1]] \right \|_{L^\infty(\Omega\backslash \overline{D})}.
 \end{aligned}
 \end{equation*}
In conjunction with the analyticity of $\bE_1^{i}$ and the conclusions in Lemma \ref{lemgraS} and Lemma \ref{lemkerA}, it can be deduced that $\left \| \bE_1\right \| _{L^{\infty }\left ( \Omega\backslash \overline{D} \right ) }$ is bounded for any bounded domain $\Omega$ that contains $B_1$ and $B_2$.
	\end{proof}
	\end{lem}
	
\subsection{Proof of Theorem \ref{th:main01}}
\begin{proof}
 We are now in the position to validate our main theorem, Theorem \ref{th:main01}. Taking note of the expansions $\bE = \bE_0(\mathbf{x}) + k\bE_1(\mathbf{x}) + k^2\bE_2(\mathbf{x}) + O(k^2)$, we have already obtained the estimations of $\bE_0$ and $\bE_1$ in Lemma \ref{le0} and Lemma \ref{E1_esti} respectively. By substituting \eqref{E_0_eatimate} and \eqref{E1_esti1} into the expansions of $\bE$, we consequently derive the estimation of $\left \| \bE\right \| _{L^{\infty }\left ( \Omega\backslash \overline{D} \right ) }$.
\end{proof}
	
\section{Numerical results and discussions}
In this section, we conduct some numerical experiments to corroborate the characterisations of the eletric fields blowup rate, that is to verify \eqref{th:main01} in Theorem \eqref{th:main01}. The main numerical calculations are performed by COMSOL Multiphysics.
\subsection{Spherical domains}
In the first example, we consider spherical domains that is in consistent with theoretical analysis. Give $r=1$ m, incident wave field is the plane wave and is given by \eqref{in}, incident direction $\mathbf{d}=(0,1,0)$, polarisation vector $ \mathbf{p}=(1,0,1)$, $\omega=0.005$. According to the field estimate \eqref{main}, the optimal blowup rate of the   electric field is of order $\epsilon^{-1} |\ln \epsilon|^{-1}$.

\begin{figure}[!htpb]
	\centering
	\subfigure[ $\epsilon=0.05,3D$]{\includegraphics[width=0.4\textwidth]{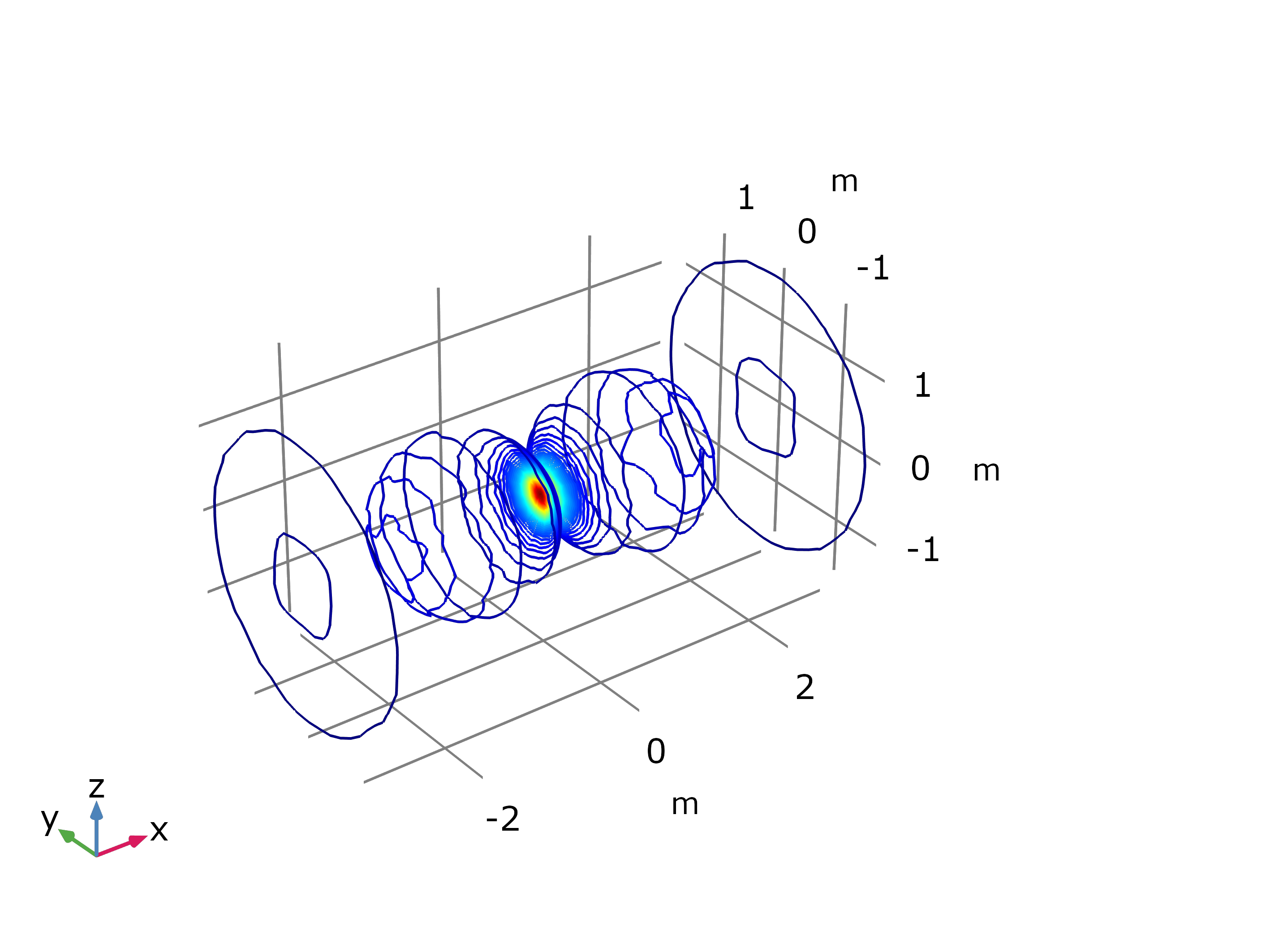}}
	\subfigure[ $\epsilon=0.05$, the XY-plane]{\includegraphics[width=0.5\textwidth]{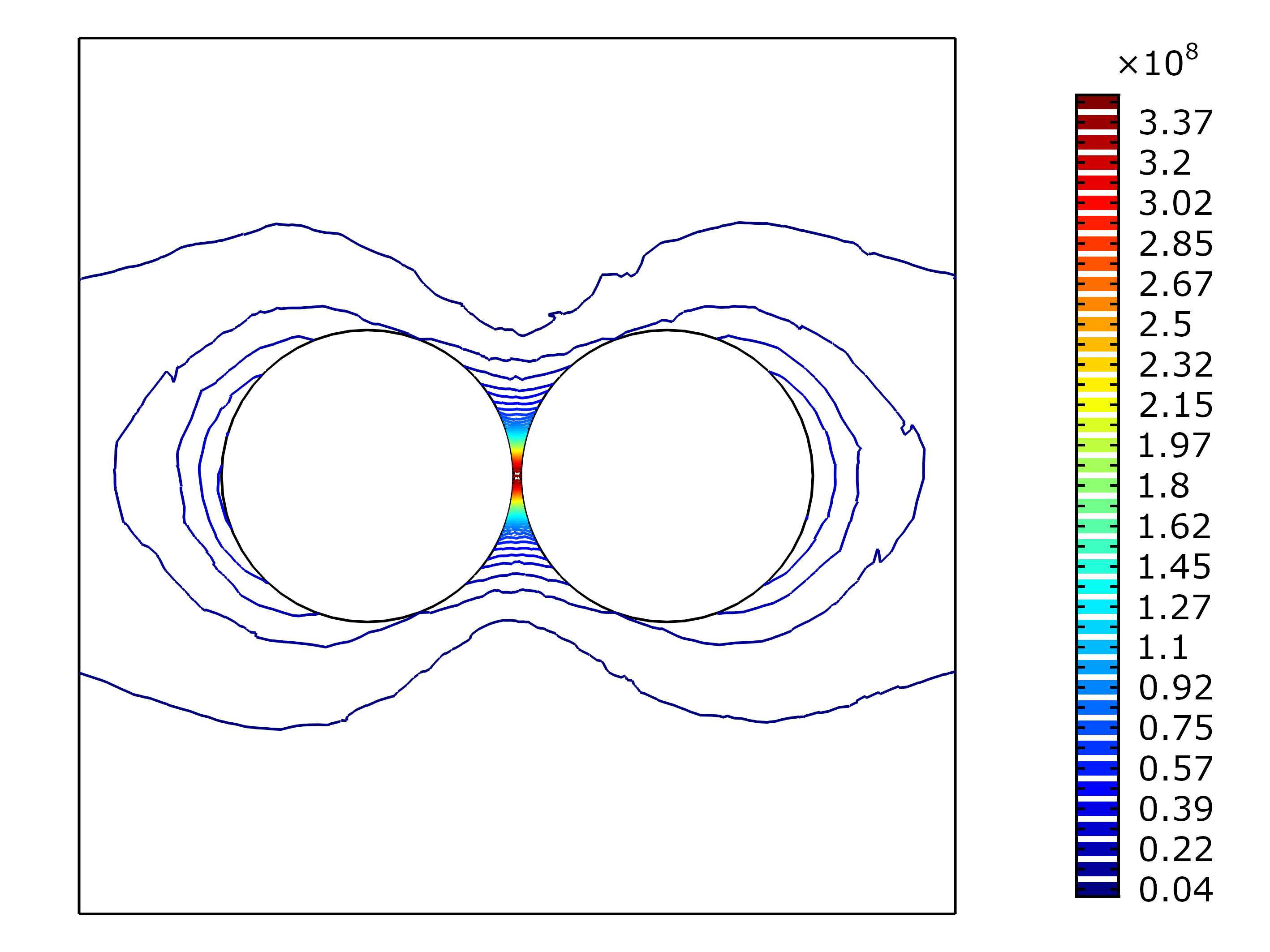}}\\
	\subfigure[ $\epsilon=0.1,3D$]{\includegraphics[width=0.4\textwidth]{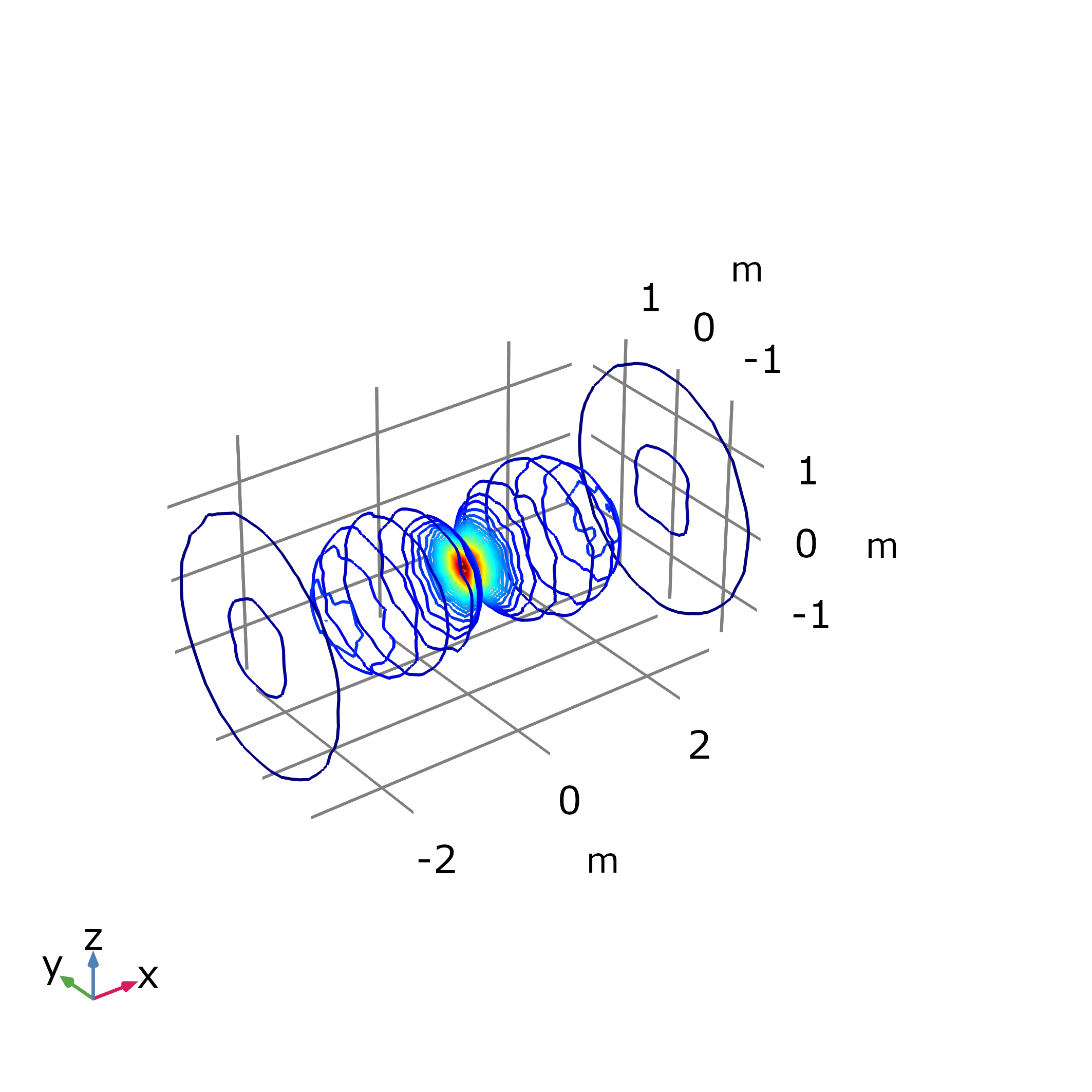}}
	\subfigure[ $\epsilon=0.1$, the XY-plane]{\includegraphics[width=0.5\textwidth]{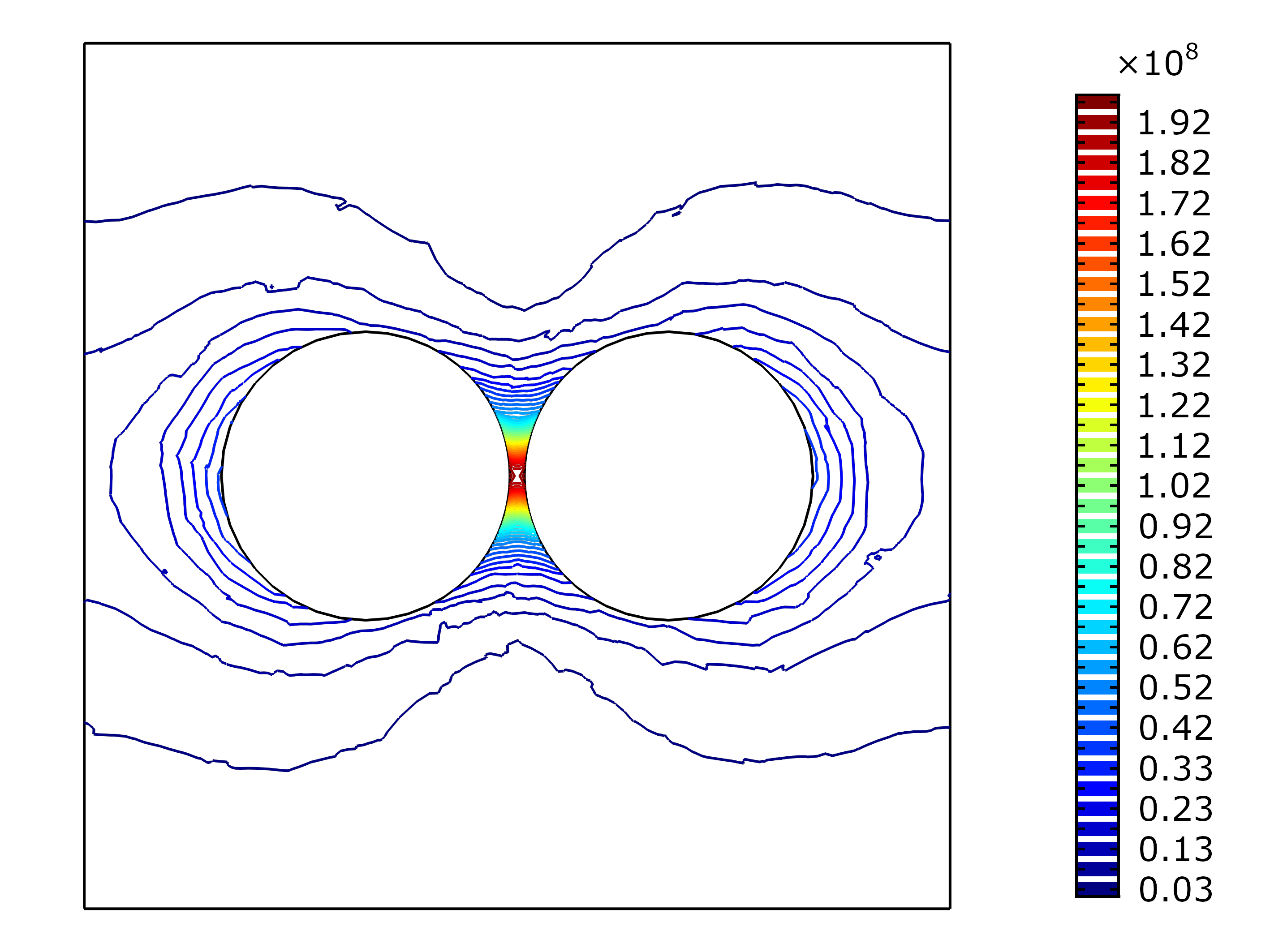}}
	\caption{\label{fig1}Spherical domains. Contour: Norm of electric field (V/m). }
\end{figure}
In Fig. \ref{fig1}, we present the electric field images when the asymptotic distance parameter is $\epsilon=0.05$ and $\epsilon=0.1$. It can be clearly seen that a significant field blow up phenomenon occurs in the area between the two inclusions. Additionally, we will verify the optimal blow up rate in \eqref{main}.
For the sake of simplicity, we replace $\|\bE\|_{L^{\infty}(\Omega\setminus\overline{B_1\cup B_2})}$ by the value of the gradient field at the midpoint $O=(0,0,0)$ of the two inclusions.

\begin{figure}[!htpb]
	\centering
	\subfigure[ Numerical results]{\includegraphics[width=0.5\textwidth]{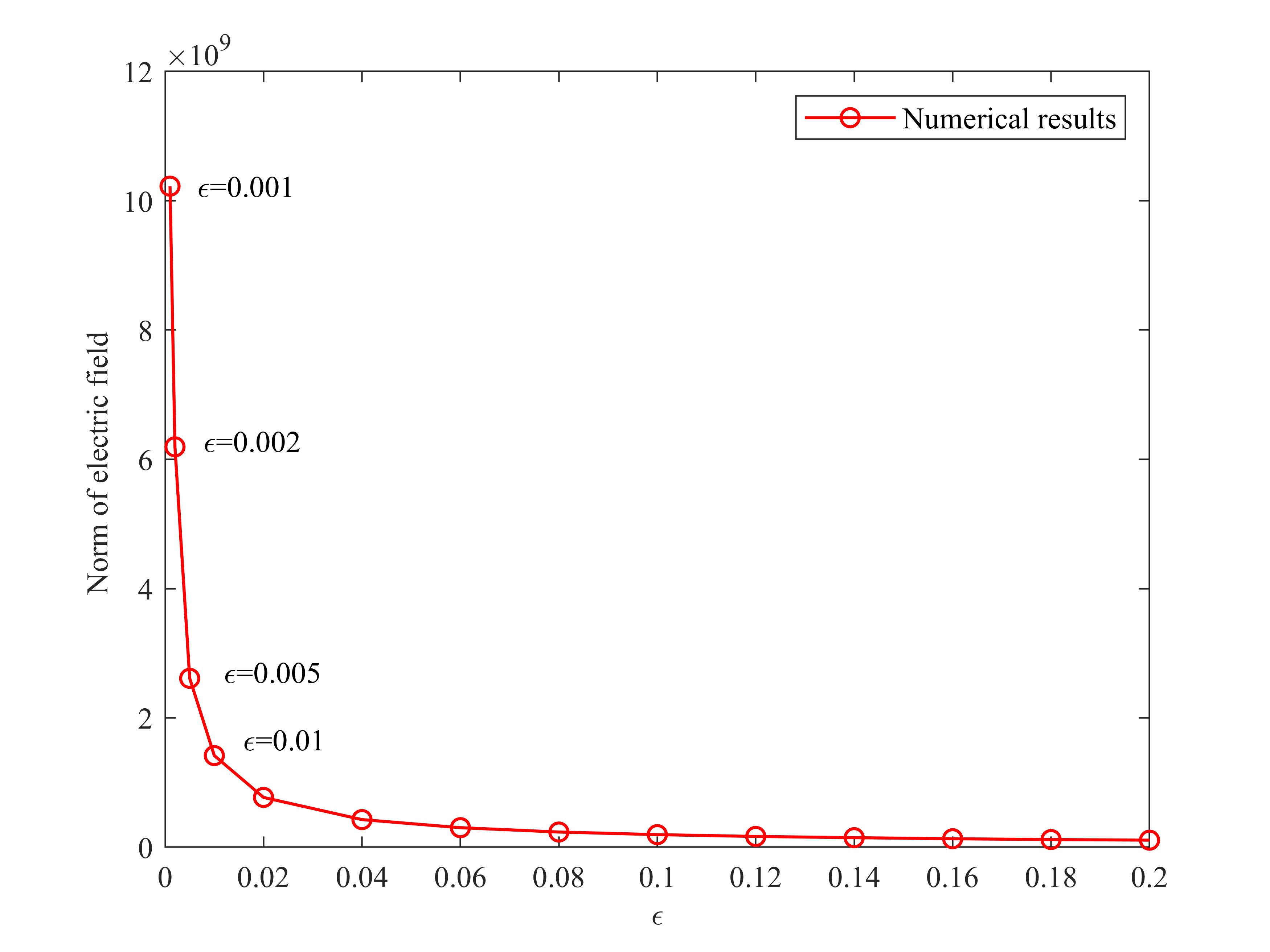}}
	\subfigure[ Comparison of estimate results with numerical results.] {\includegraphics[width=0.5\textwidth]{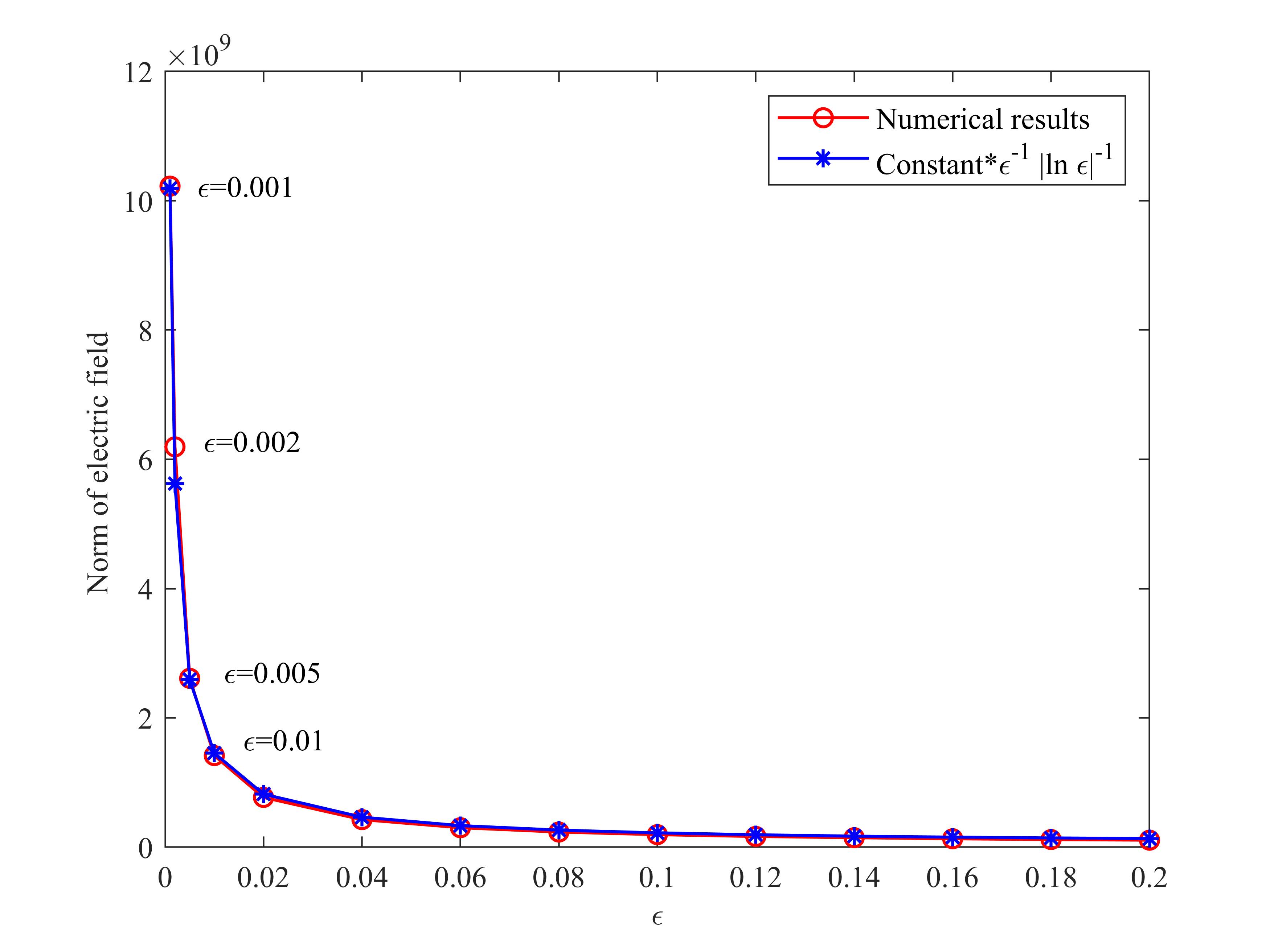}}\\
	\caption{\label{fig2}Spherical domains. The variation of electric field norm with $\epsilon$. }
\end{figure}

 We calculate the trend of the electric field with respect to $epsilon$ within the range of $[0.001, 0.2]$ and represent the results graphically in (a). Simultaneously, by utilizing the electric field estimation formula given in \eqref{main}, we compute the theoretical estimates of the electric field and compare them with the numerical results. This comparison is illustrated graphically in (b). The findings demonstrate a significant coherence between the theoretical predictions and the numerical results. This instance substantiates the efficacy of the blow-up order estimation result presented in this paper.

\subsection{Hemisphere domains}
In this example, we consider the hemishpere domains which have the same curvature in the nearly-touching surfaces with sphere domains. The physical settings are the same as example 1.
\begin{figure}[!htpb]
	\centering
	\subfigure
	[$\epsilon=0.05$,3D]{\includegraphics[width=0.4\textwidth]{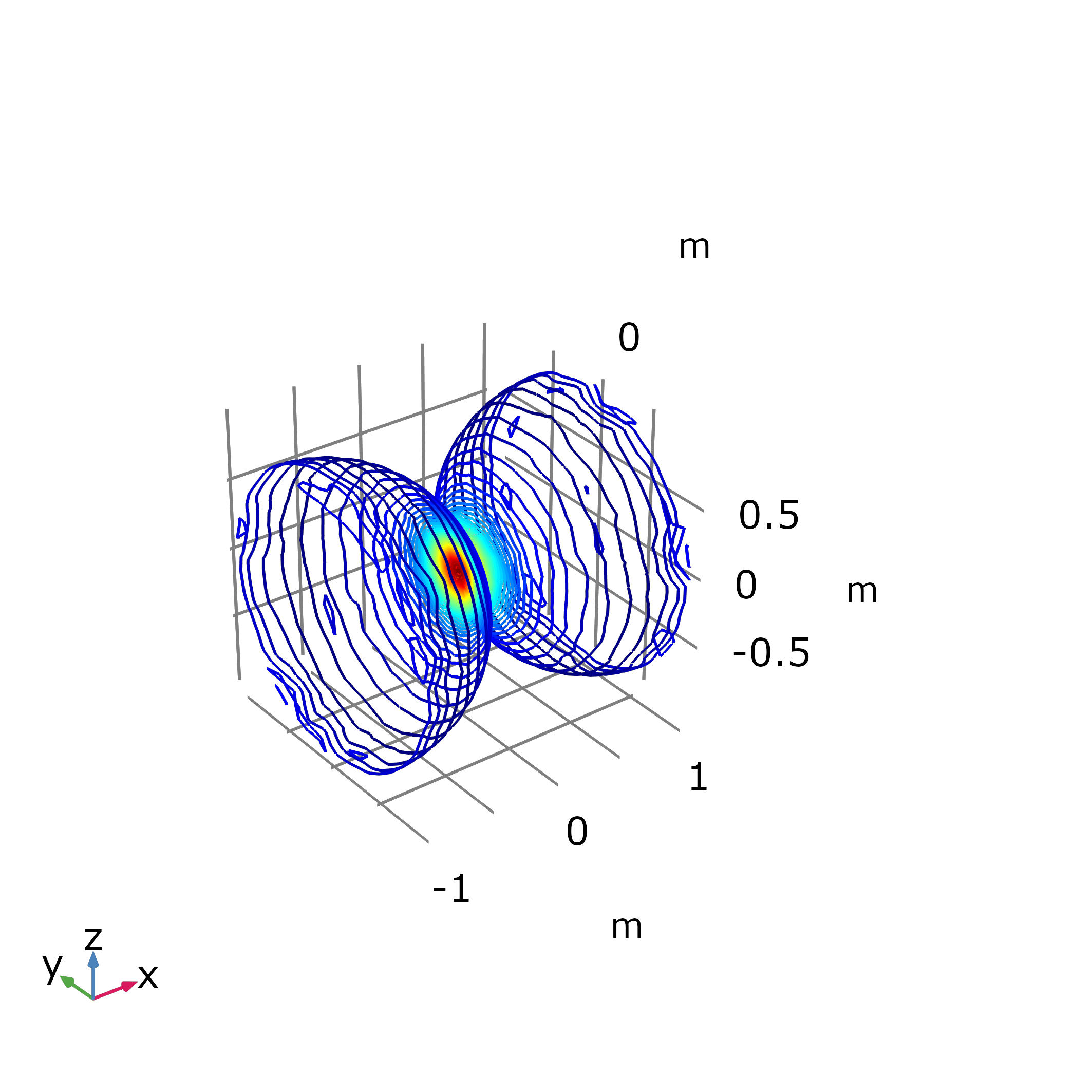}}
	\subfigure
	[$\epsilon=0.05$,the XY-plane]{\includegraphics[width=0.5\textwidth]{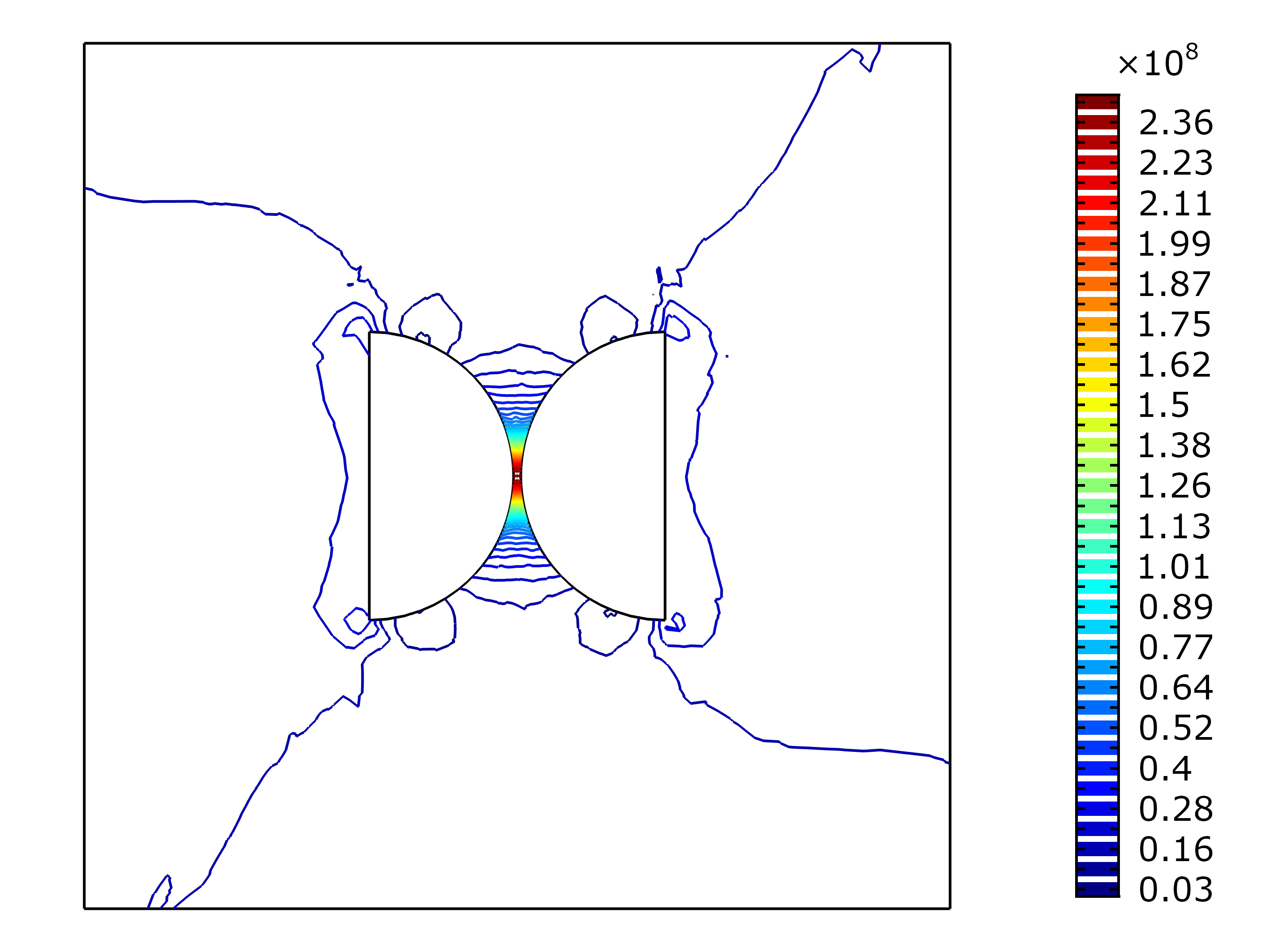}}\\
	\subfigure[ $\epsilon=0.1,3D$]{\includegraphics[width=0.4\textwidth]{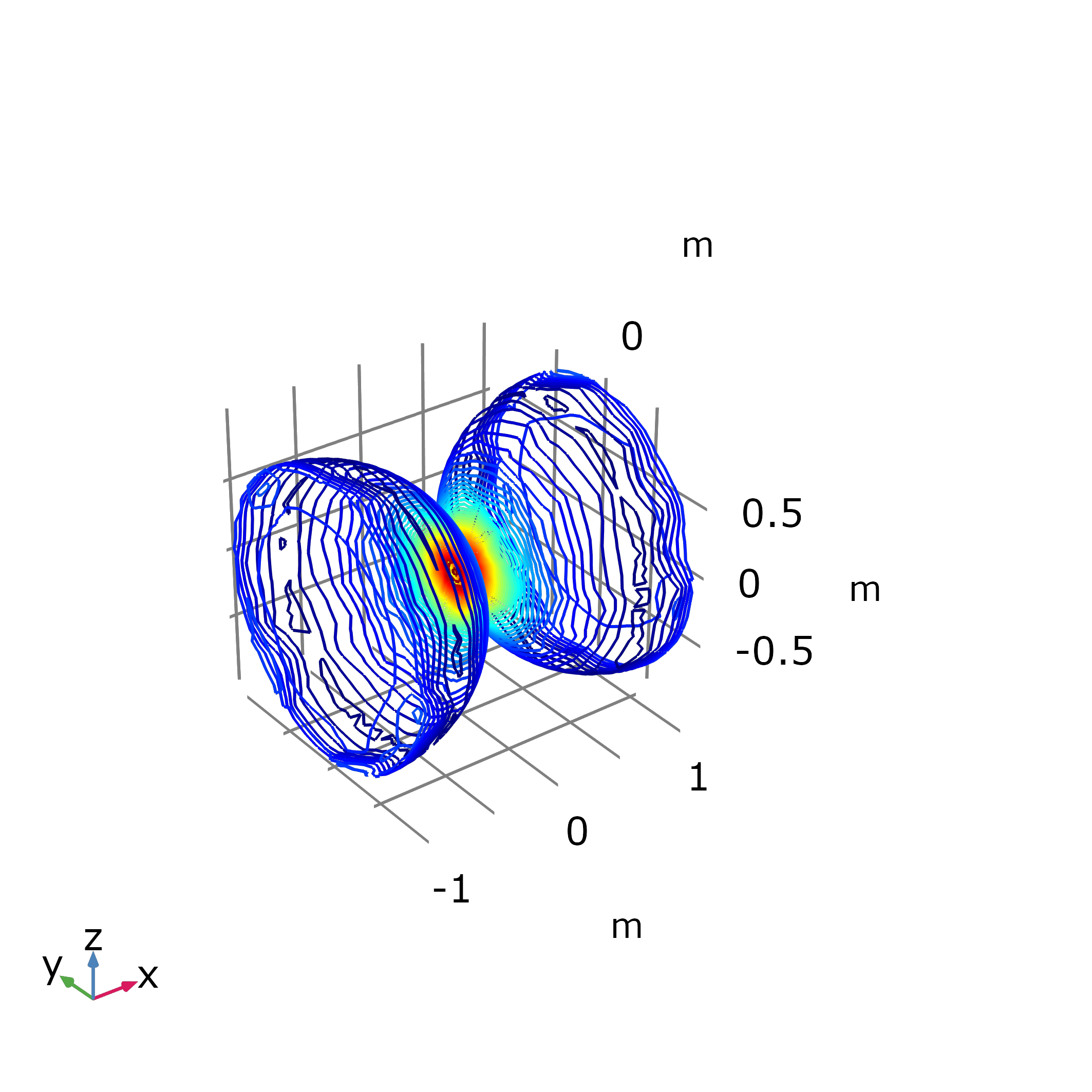}}
	\subfigure[ $\epsilon=0.1$, the XY-plane]{\includegraphics[width=0.5\textwidth]{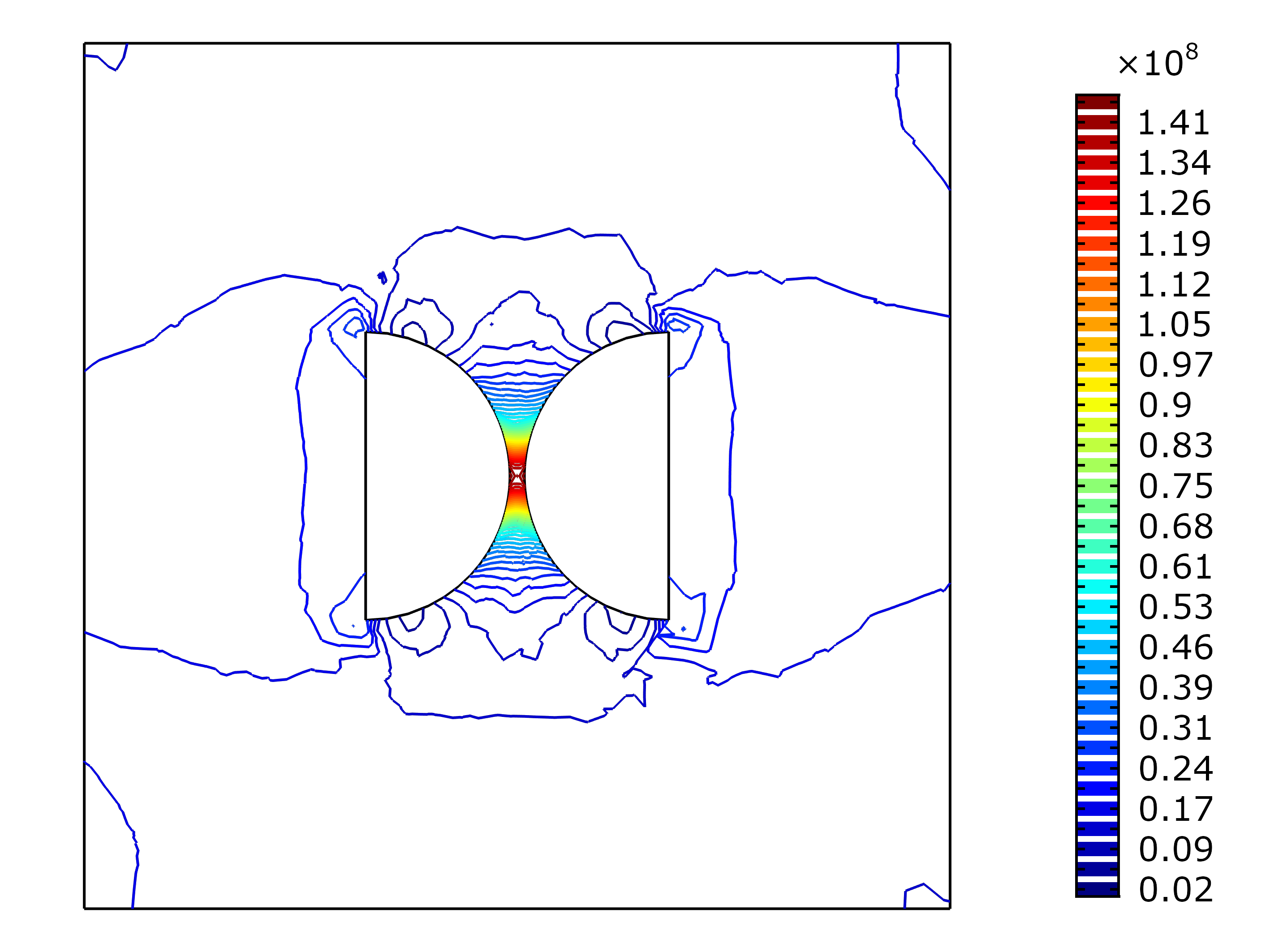}}
	\caption{\label{fig3}Hemisphere domains. Contour: Norm of electric field (V/m). }
\end{figure}

In Fig. \ref{fig3}, we plot the electric fields images when the asymptotic distance parameters are $\epsilon=0.05$ and  $\epsilon=0.01$, respectively. It can be clearly observed that when the two inclusions are hemispherical, a significant blow-up phenomenon still occurs in the area between the two inclusions.
\begin{figure}[!htpb]
	\centering
	\subfigure[ Numerical results]{\includegraphics[width=0.5\textwidth]{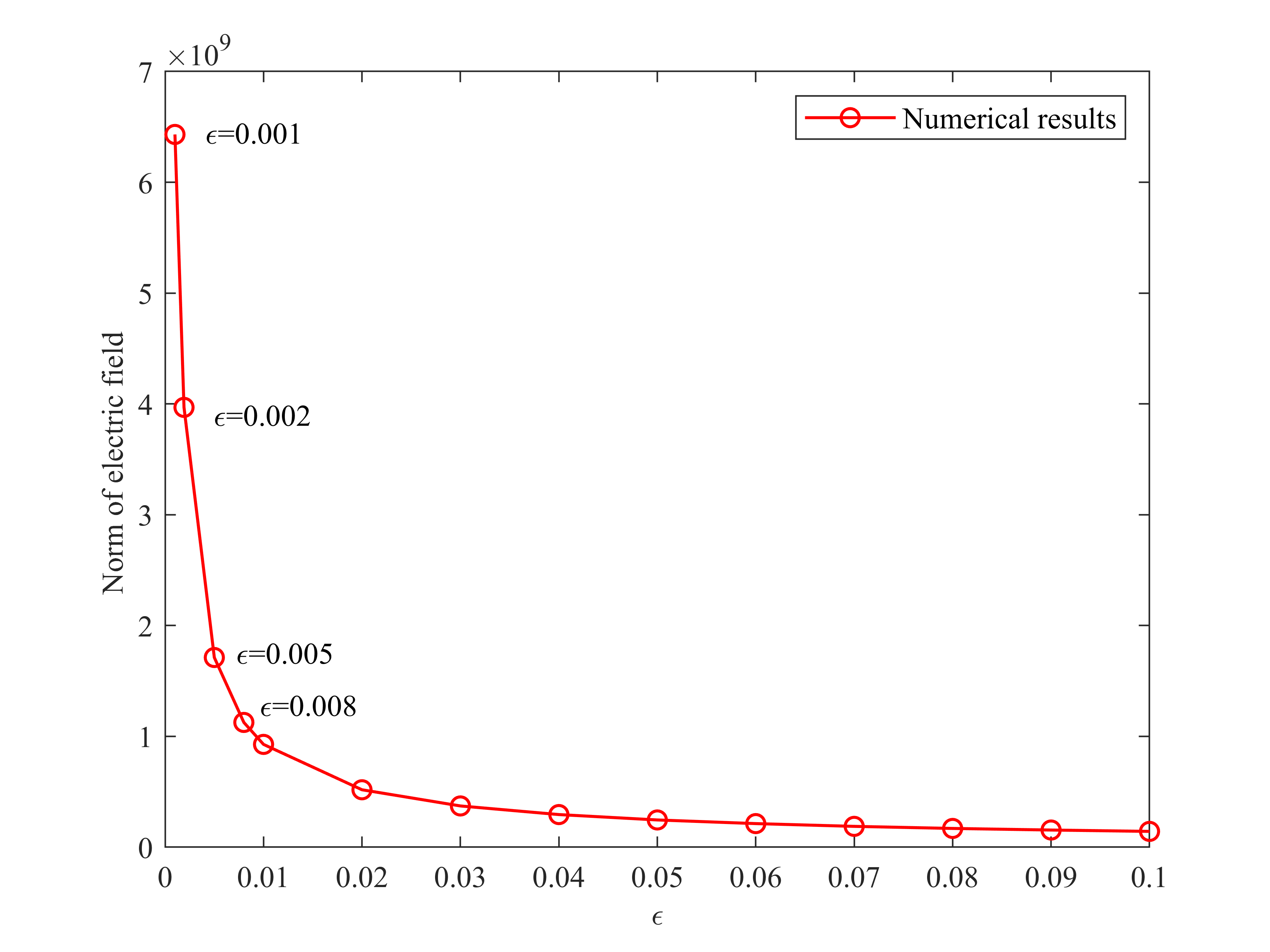}}
	\subfigure[ Comparison of estimate results with numericalby results.]{\includegraphics[width=0.5\textwidth]{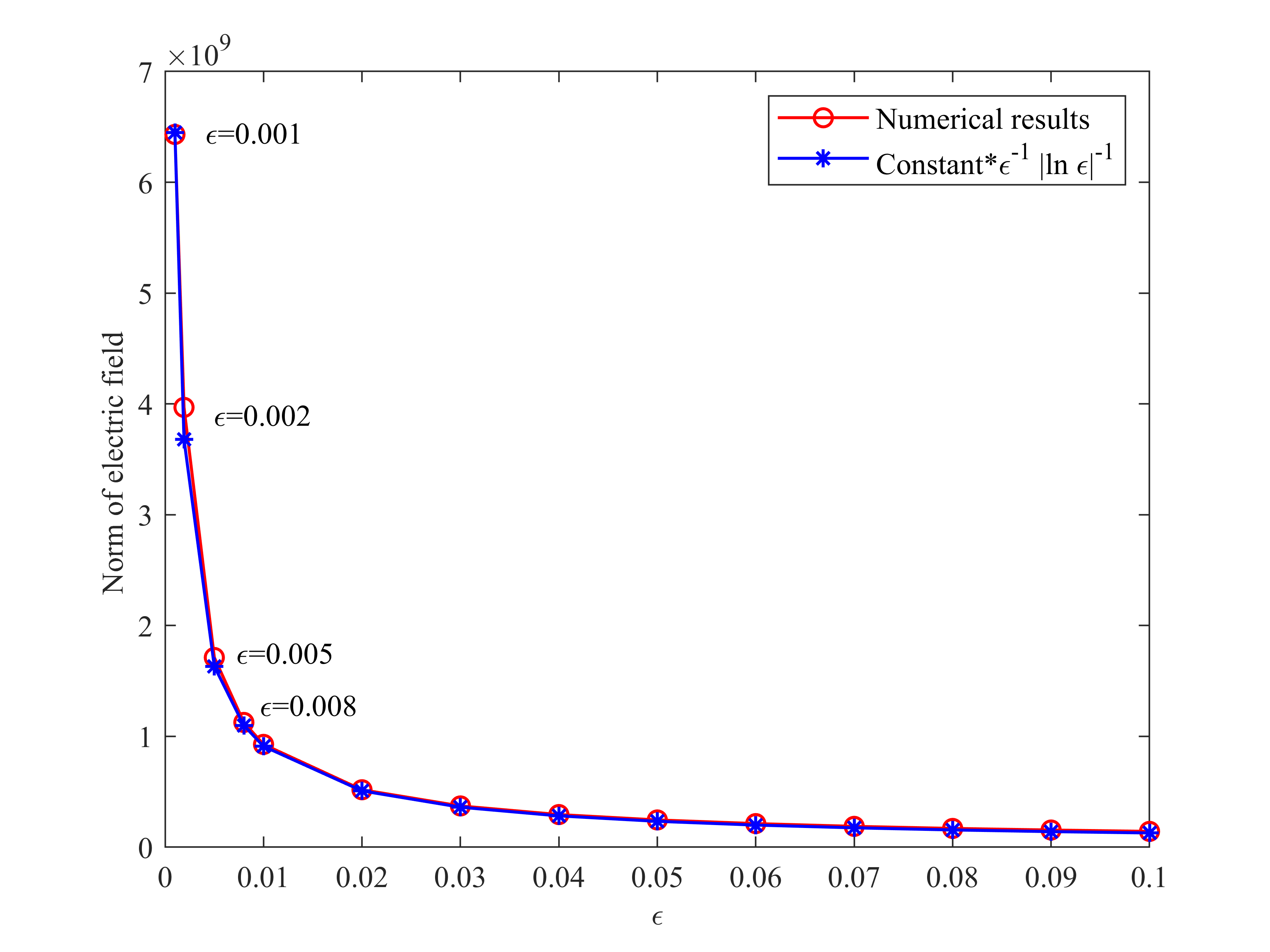}}\\
	\caption{\label{fig4}Hemisphere domains. The variation of electric field norm with $\epsilon$. }
\end{figure}

Similar to Example 1, we calculate the electric fields for $\epsilon \in [0.001,0.1]$, Fig. \ref{fig4} shows that when the two inclusions in almost contact are hemispherical, it has the same blow up order as in the case of spherical domains.

\subsection{Gourd-shaped domains}
In this example, we consider the more general geometry, gourd-shaped domains(see Fig. \ref{model3}), which also have almost the same curvature in the nearly-touching surfaces with sphere domains. The physical settings are the same as example 1.
\begin{figure}[!htpb]
	\centering
	\subfigure[Gourd-shaped] {\includegraphics[width=0.3\textwidth]{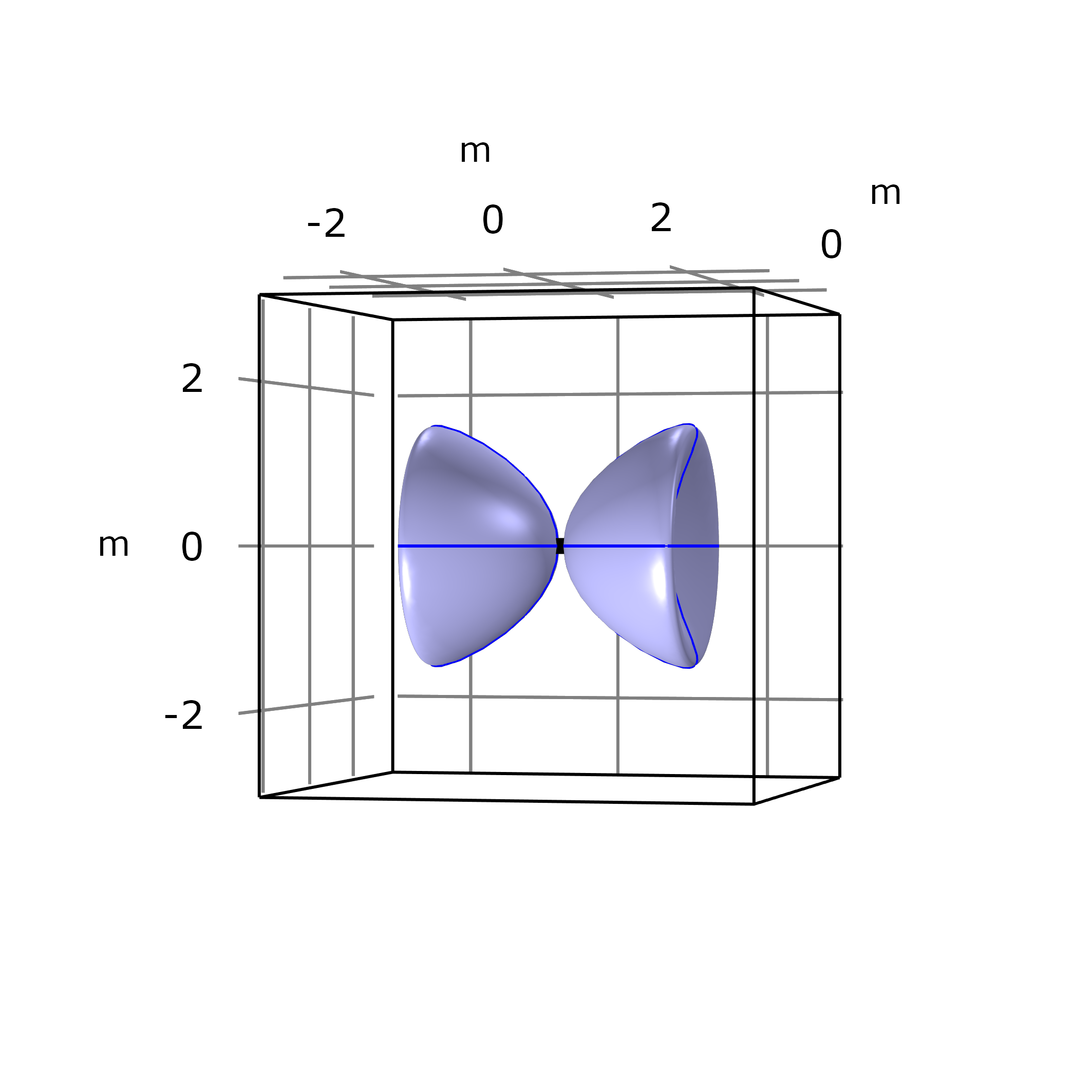}}
	\subfigure[Kite-shaped domain in 2D(the XY-plane)]{\includegraphics[width=0.3\textwidth]{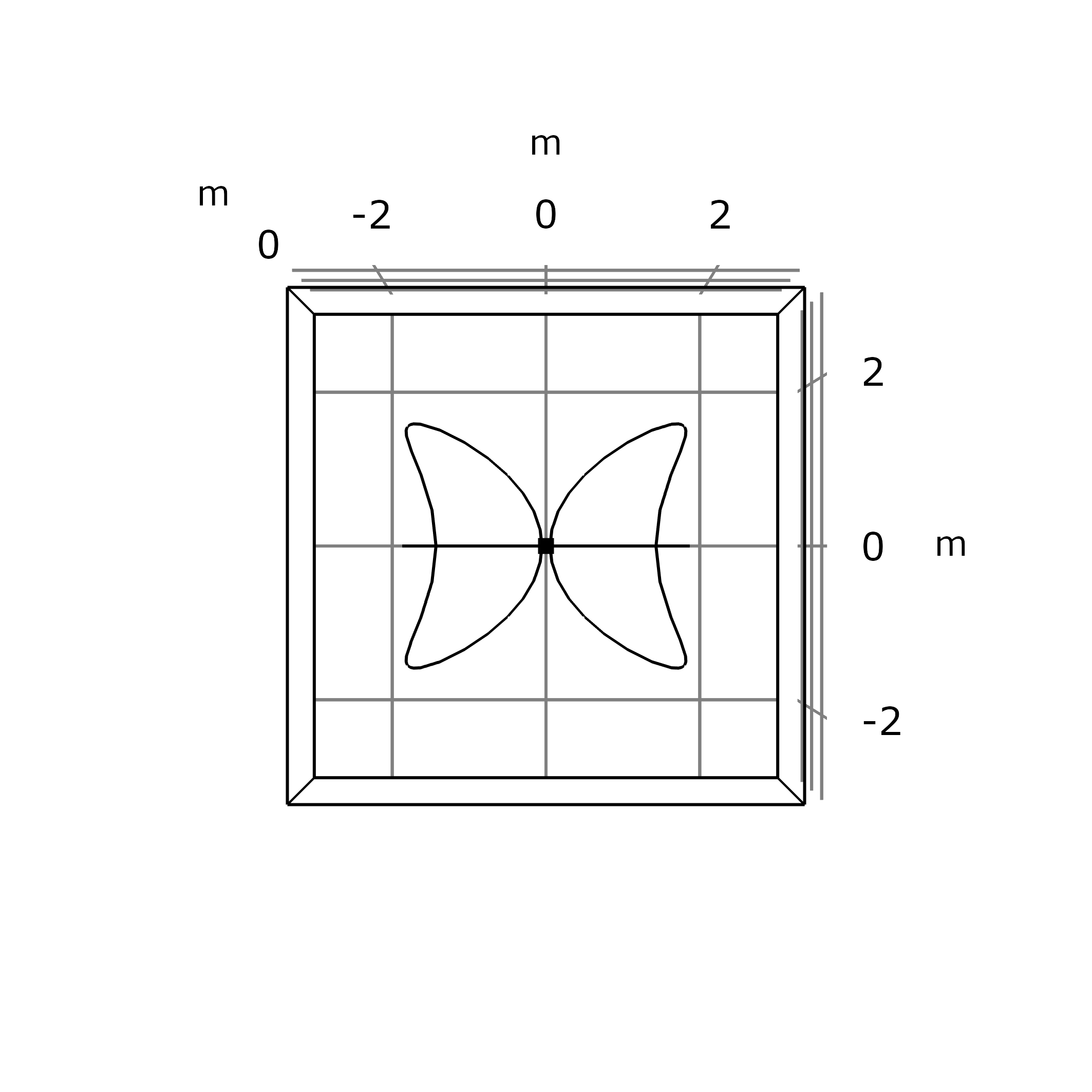}}
	\subfigure[Comparison with spherical shapes(the XY-plane)] {\includegraphics[width=0.3\textwidth]{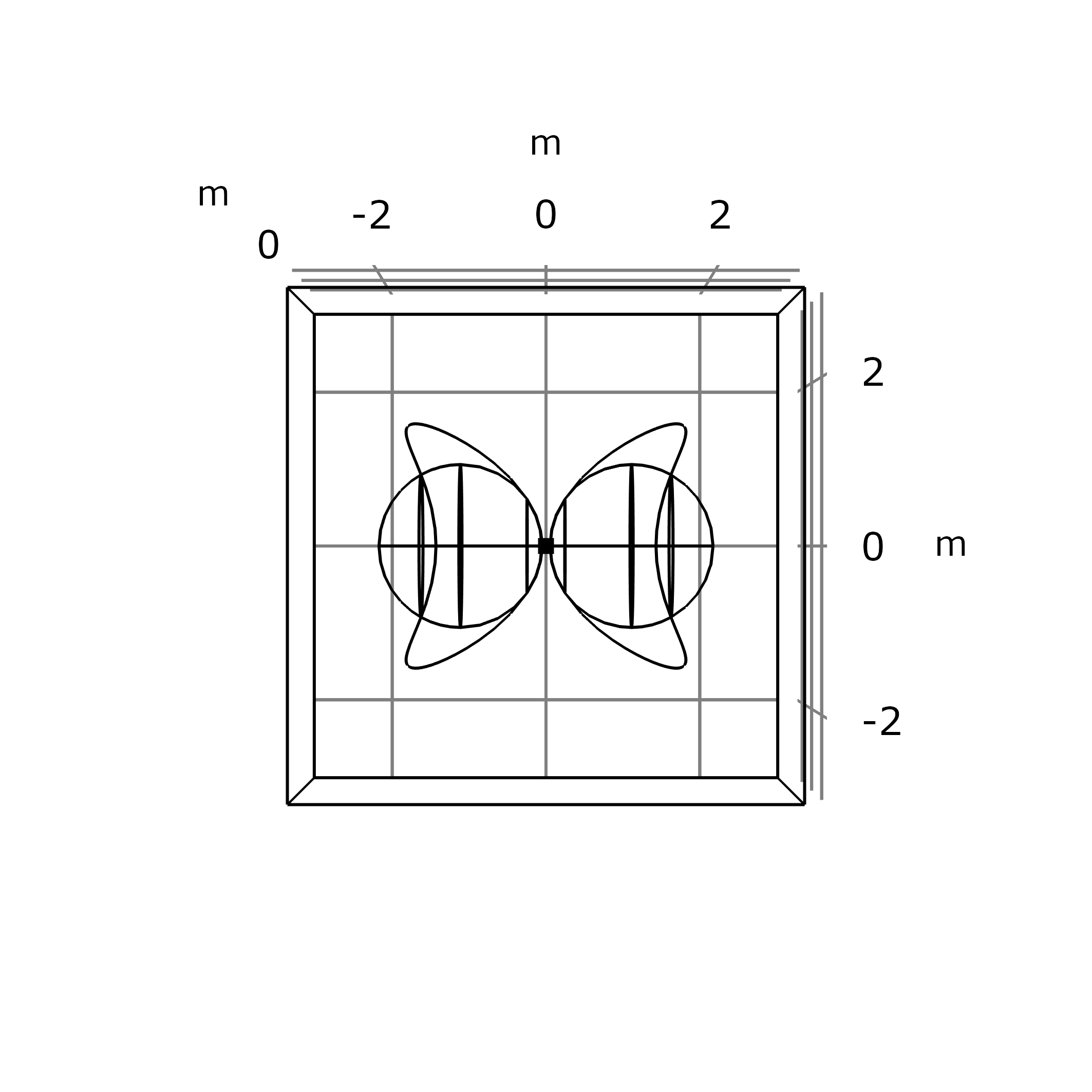}}\\
	\caption{\label{model3}Shapes considered in the numerical example. }
\end{figure}

\begin{figure}[!htpb]
	\centering
	\subfigure[ $\epsilon=0.05,3D$]{\includegraphics[width=0.4\textwidth]{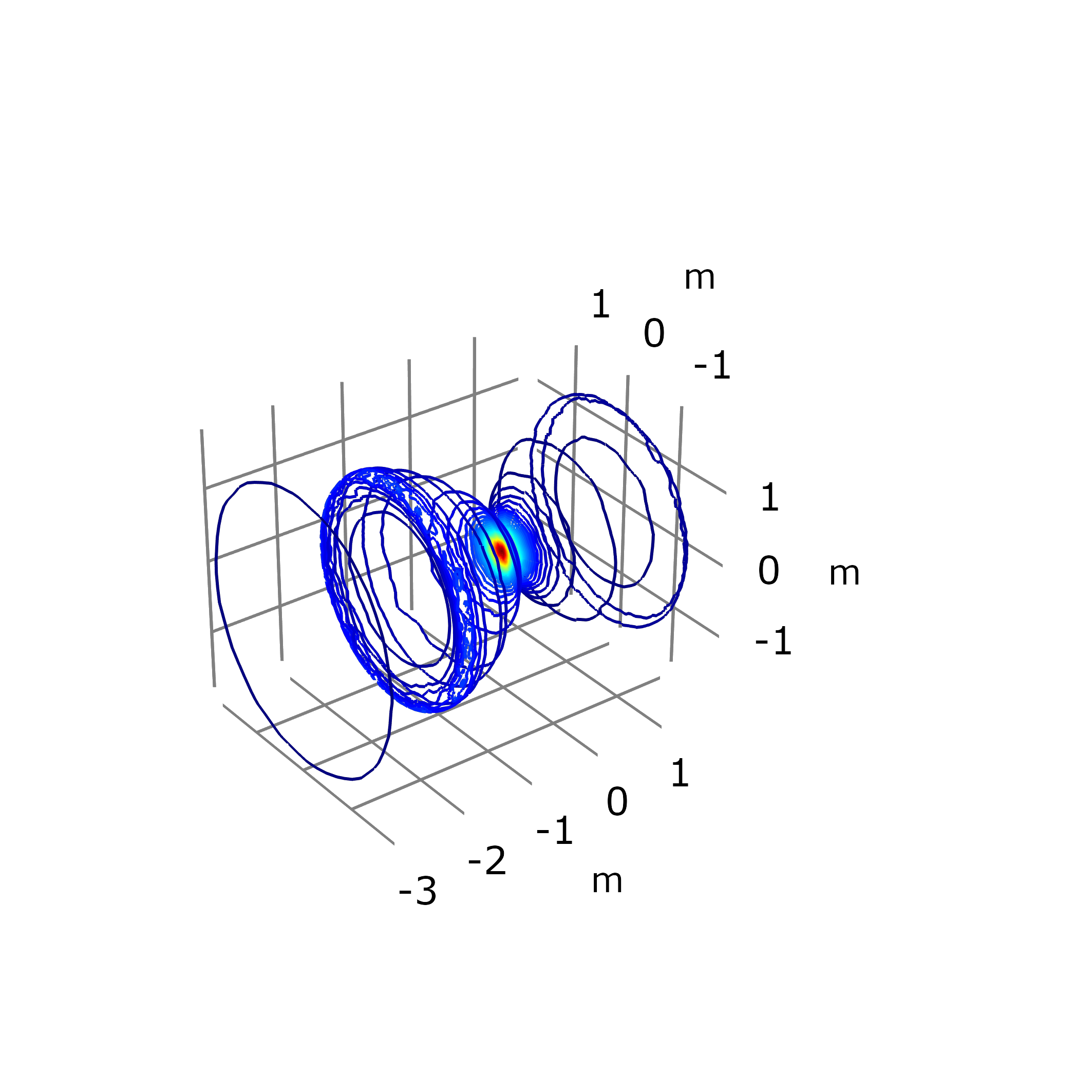}}
	\subfigure[ $\epsilon=0.05$, the XY-plane]{\includegraphics[width=0.5\textwidth]{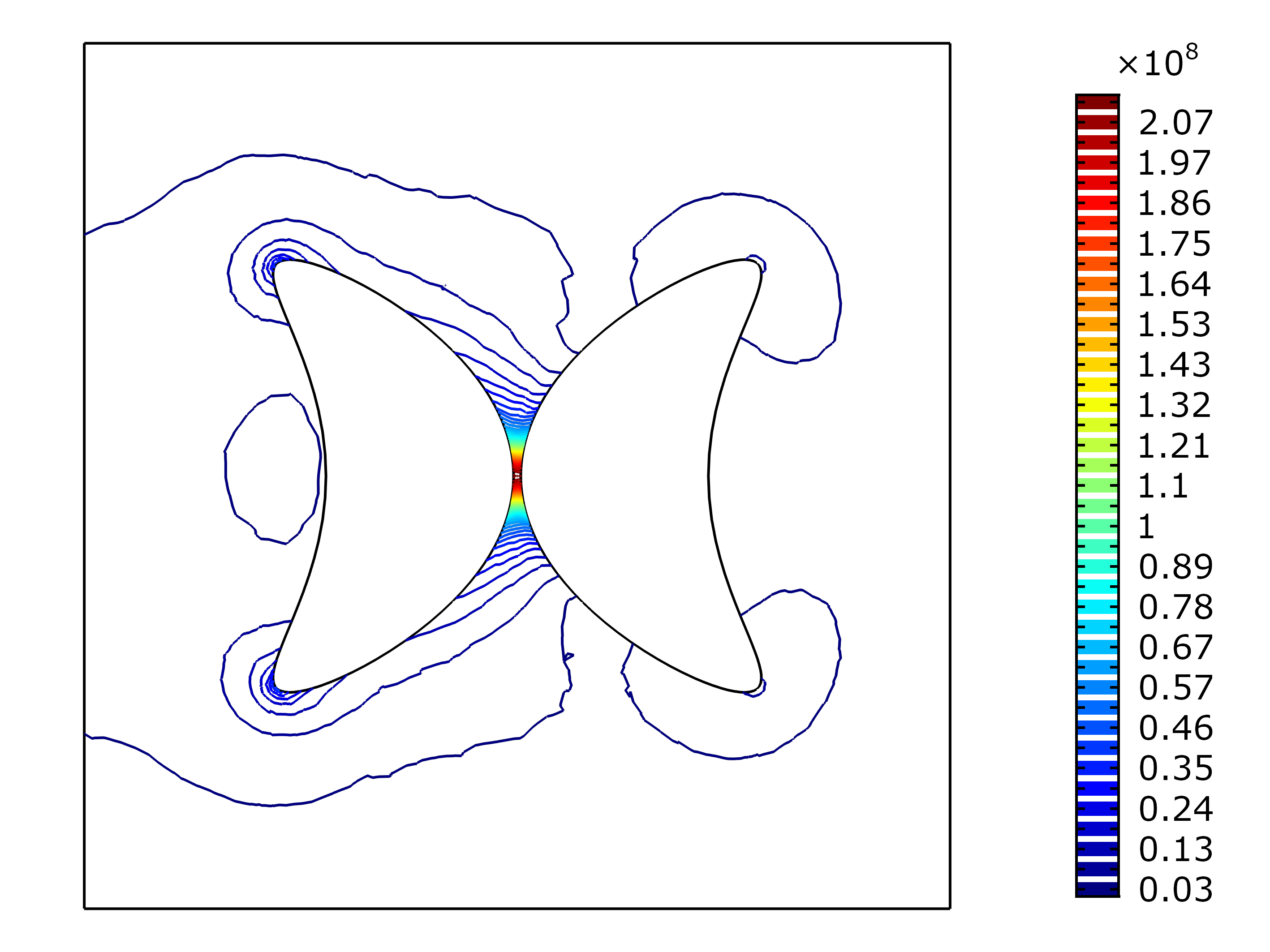}}\\
	\subfigure[ $\epsilon=0.1,3D$]{\includegraphics[width=0.4\textwidth]{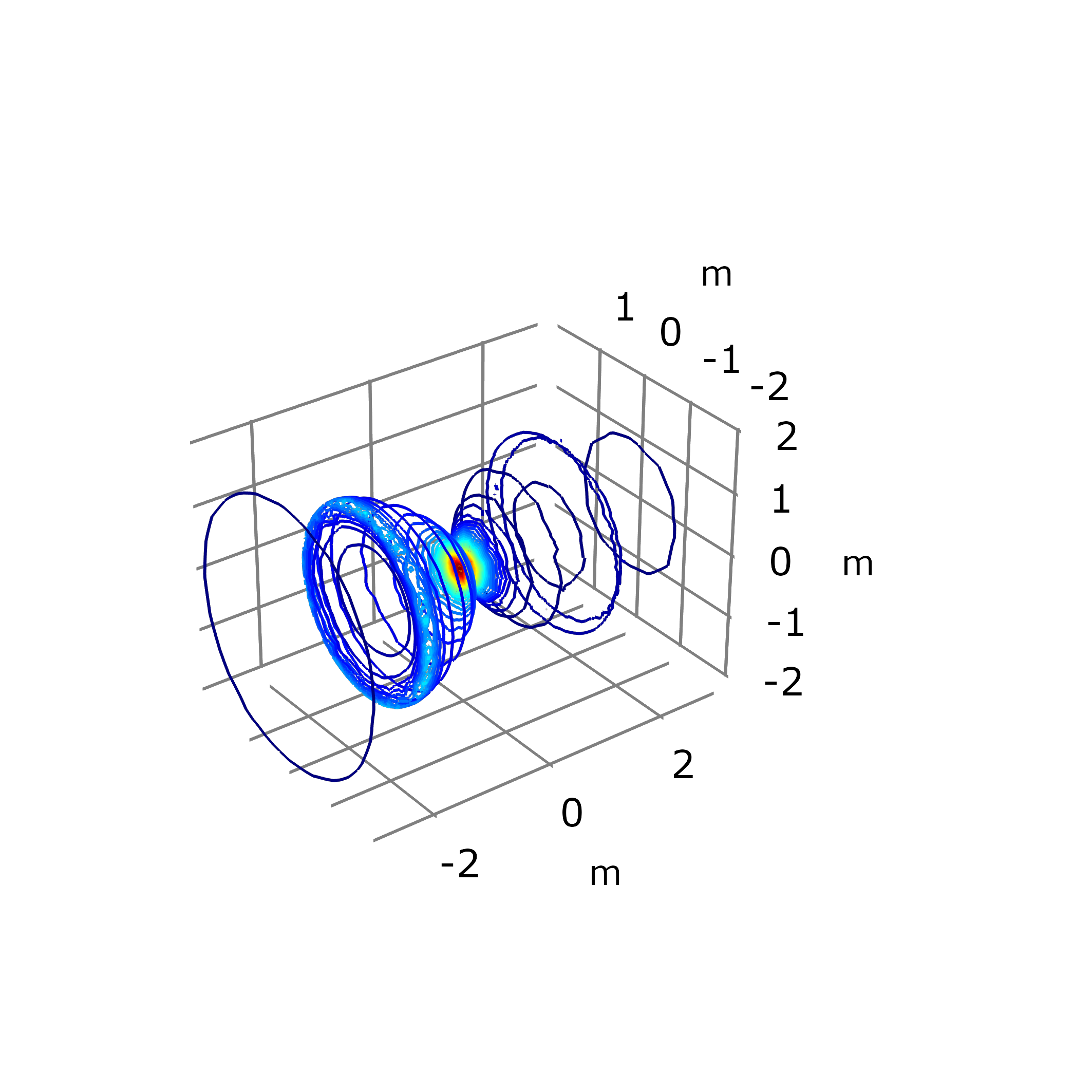}}
	\subfigure[ $\epsilon=0.1$, the XY-plane]{\includegraphics[width=0.5\textwidth]{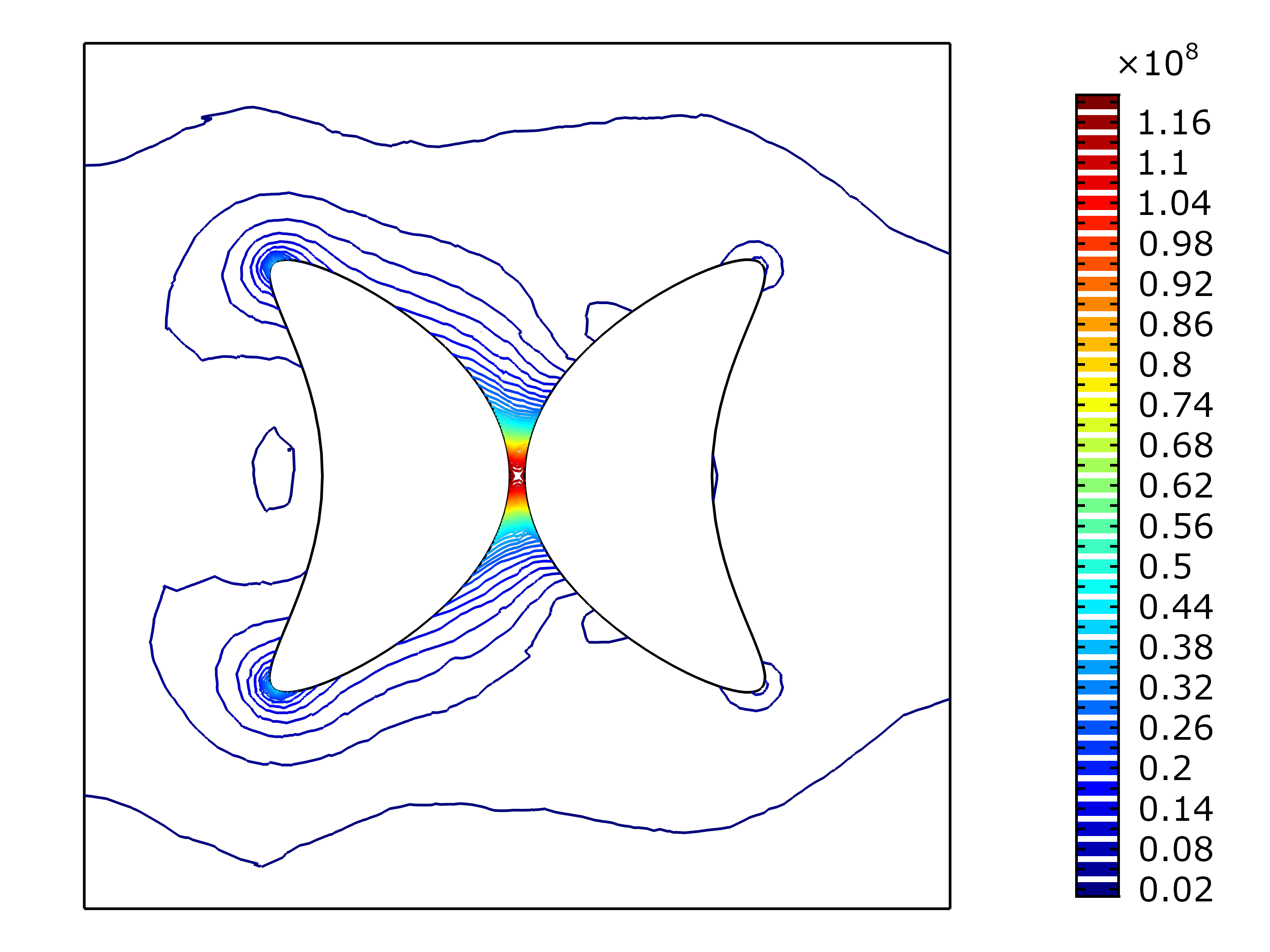}}
	\caption{\label{fig5} Gourd-shaped domains. Contour: Norm of electric field (V/m). }
\end{figure}
From Fig. \ref{fig5}, one can see that the electric field still exhibits blowup phenomena when the two inclusions are nearly touching.
\begin{figure}[!htpb]
	\centering
	\subfigure[ Numerical results]{\includegraphics[width=0.5\textwidth]{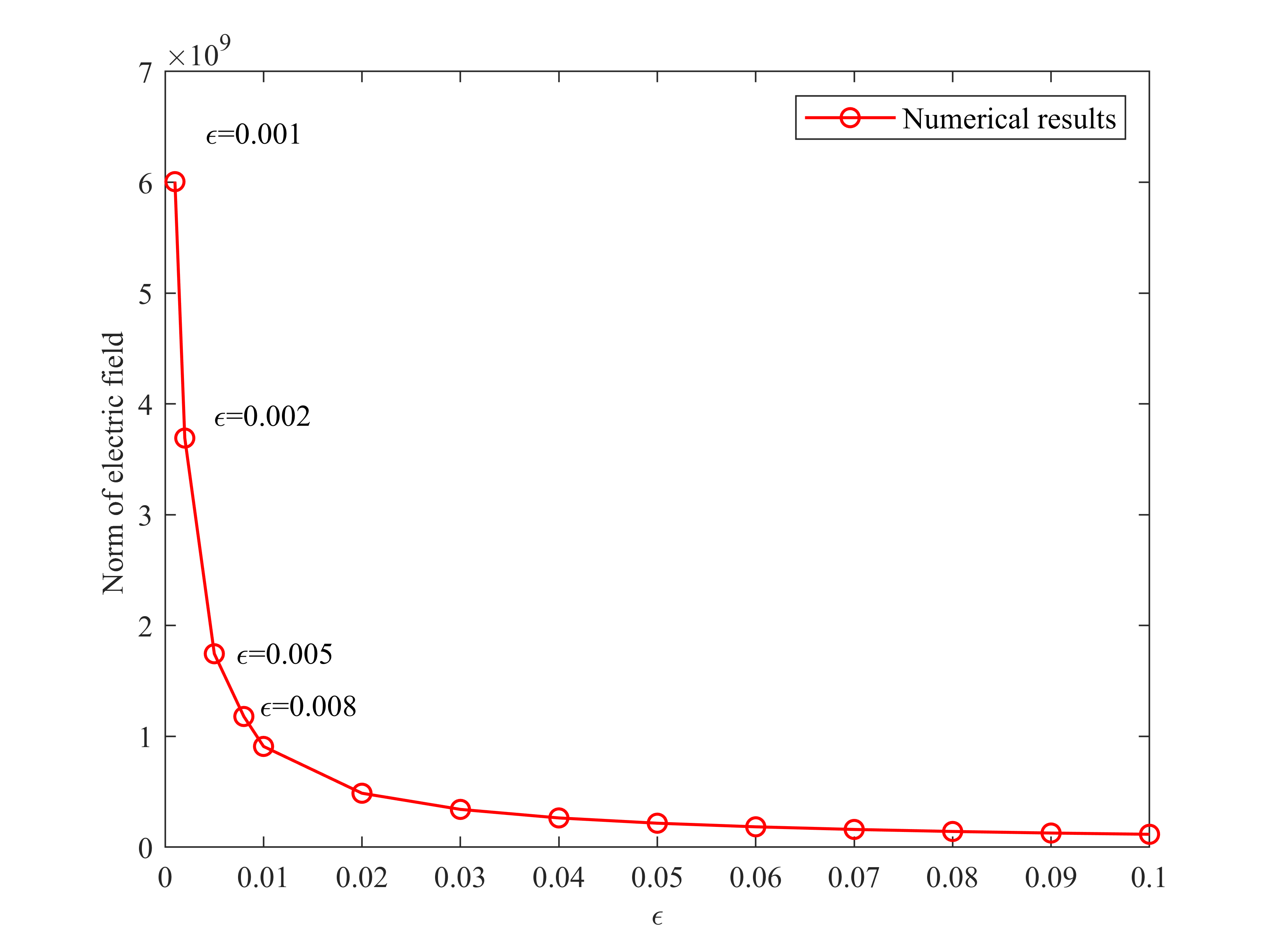}}
	\subfigure[ Comparison of estimate results with numericalby results.]{\includegraphics[width=0.5\textwidth]{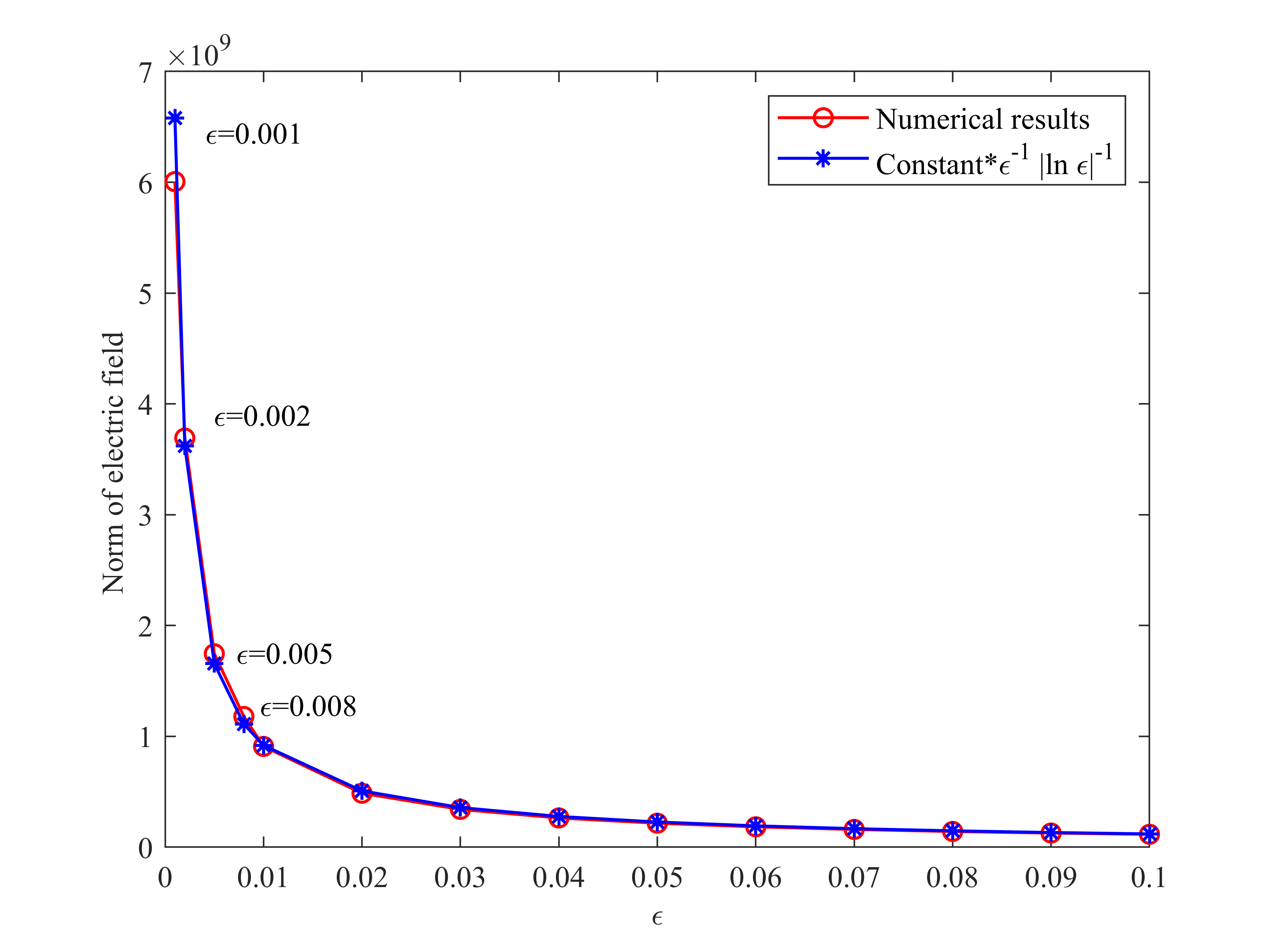}}\\
	\caption{\label{fig6}Gourd-shaped domains. The variation of electric field norm with $\epsilon$. }
\end{figure}
Finally, in Fig. \ref{fig6}, similar numerical results and comparison with theoretical results are given,  they also matched well.

Through the last two examples, we find that the blowup rate is related to the curvature of the approximate contact surface of the inclusions. Even if the inclusions are not spherical, when the two almost touching inclusions have the same curvature near their almost touching area, the burst order still holds. This is our new discovery in the general geometric area, which may guide the further development of theoretical analysis in the future.
\section{Conclusions}

This paper studies the concentration of the electric field caused by two nearly touching inclusions in Maxwell equations. By using vector layer potential technology, we derive the representation of the solution to Maxwell's equations. Through low frequency asymptotic analysis, we derive the asymptotic expansion of the solution.
For the leading term $\bE_0$ of the asymptotic expansion, we proved that $\bE_0$ has a blowup order of $\epsilon^{-1} |\ln \epsilon|^{-1}$ by proving that $\bE_0=\nabla u$ and $u$ satisfies the perfect conductor system. Through some operator analysis, we proved the boundedness of the first order term $\bE_1$.

Through numerical experiments, we have confirmed the blasting phenomenon in the narrow region between two nearly touching high-contrast spherical conductors, and also confirmed the estimates of the blasting order given in this paper. Furthermore, through our numerical experiments, we have made new discoveries in more general geometric settings, namely, even if the nearly-touching inclusions are not spheres, the blasting order still holds when their curvature in the nearly-touching region is close to that of a sphere.

\section*{Acknowledgment}
The work of Y. Deng was supported by NSFC-RGC Joint Research Grant No. 12161160314.
The work of H. Liu was supported by the Hong Kong RGC General Research
Funds (projects 11311122, 11300821 and 12301420), the NSFC/RGC Joint Research Fund
(project N\_CityU101/21), and the ANR/RGC Joint Research Grant, A\_CityU203/19.


\begin{thebibliography}{90}
\bibitem{ADM16} H. Ammari, Y. Deng, and P. Millien.
\newblock Surface plasmon resonance of nanoparticles and applications in imaging.
\newblock {\em Arch. Ration. Mech. Anal.}, 220(1):109--153, 2016. 
	
\bibitem{ammari2007optimal}
	H.~Ammari, H.~Kang, H.~Lee, J.~Lee, and M.~Lim.
	\newblock Optimal estimates for the electric field in two dimensions.
	\newblock {\em J. Math. Pures Appl.},
	88(4):307--324, 2007.
	
	\bibitem{ammari2005gradient}
	H.~Ammari, H.~Kang, and M.~Lim.
	\newblock Gradient estimates for solutions to the conductivity problem.
	\newblock {\em Math. Ann.}, 332(2):277--286, 2005.
	
	

	
	
	
	\bibitem{ammari2013} H. Ammari, G. Ciraolo, H. Kang, H. Lee, and K. Yun. \newblock Spectral analysis of the Neumann-Poincaré
	operator and characterization of the stress concentration in anti-plane elasticity.
	\newblock{\em Arch. Ration.
	Mech. Anal.}, 208(1):275--304, 2013.
	
	
	
	
\bibitem{kangbook}	H. Ammari, H. Kang. Polarization and Moment Tensors with Applications to Inverse
Problems and Effective Medium Theory. Applied Mathematical Sciences, Vol. 162,
Springer-Verlag, New York, 2007.
	
	\bibitem{AKKY21}
	H. Ammari, H. Kang, D.W. Kim, and S. Yu. {Quantitative estimates for stress concentration of the Stokes flow between adjacent circular cylinders}.
	{\em SIAM J. Math. Anal.}, 55(4):3755--3806, 2023.
	
	\bibitem{AZ}
	H.~Ammari, P.~Millien, M.~Ruiz, and H.~Zhang.
	\newblock Mathematical analysis of plasmonic nanoparticles: the scalar case.
	\newblock {\em Arch. Ration. Mech. Anal.}, 224(2):597--658,
	  2017.
	
	\bibitem{BA} I. Bab$\breve{\text{u}}$ska, B. Andersson, P. Smith and K. Levin.
	 {Damage analysis of fiber composites. I. Statistical analysis on fiber scale}. {\em Comput. Meth. Appl. Mech. Engrg.},  172(1-4):27--77, 1999.
	
	
	\bibitem{bao2009gradient}
	E.~Bao, Y.~Y. Li, and B.~Yin.
	\newblock Gradient estimates for the perfect conductivity problem.
	\newblock {\em Arch. Ration. Mech. Anal.}, 193(1):195--226,
	2009.
	
	\bibitem{bao2010gradient}
	E.~Bao, Y.~Li, and B.~Yin.
	\newblock Gradient estimates for the perfect and insulated conductivity
	problems with multiple inclusions.
	\newblock {\em Comm. Partial Differential Equations},
	35(11):1982--2006, 2010.
	
	
	\bibitem{BLL17} J. Bao, H. Li, and Y. Li. 
	 \newblock Gradient estimates for solutions of the Lam\'e system with partially infinite coefficients in dimensions greater than two.
 {\em Adv. Math.}, {305}:298--338, 2017.

 \bibitem{baolili}J.~Bao, H.~Li, and Y.~Li. Gradient estimates for solutions of the Lam\'e system with partially infinite coefficients. {\em Arch. Ration. Mech. Anal.}, 215(1):307--351, 2015.
	
	
	\bibitem{bonnetier2000elliptic}
	E.~Bonnetier, M.~Vogelius.
	\newblock An elliptic regularity result for a composite medium with ``touching"
	fibers of circular cross-section.
	\newblock {\em SIAM J. Math. Anal.}, 31(3):651--677, 2000.
	
\bibitem{Buca1984}B. Budiansky, G. F. Carrier. High shear stresses in Stiff-Fiber composites.
{\em J. Appl. Mech.}, 51(4):733--735, 1984.
	
	\bibitem{ChenLiXu2021} Y.~Chen, H.~Li, and L.~Xu. 
	\newblock {Optimal gradient estimates for the perfect conductivity
			problem with {$C^{1, \alpha}$} inclusions}.
	 {\em Ann. Inst. H. Poincar\'{e} C Anal. Non Lin\'{e}aire.}, 38(4):953--979, 2021.
	
	
	
\bibitem{DongXuCNS} J.~Choi, H.~Dong, and L.~Xu. 
\newblock Gradient estimates for Stokes and Navier-Stokes systems with piecewise DMO coefficients. {\em SIAM J. Math. Anal.} 54(3):3609--3635, 2022. 
	
	
	\bibitem{deng2022gradient}
	Y.~Deng, X.~Fang, and H.~Liu.
	\newblock Gradient estimates for electric fields with multiscale inclusions in
	the quasi-static regime.
	\newblock {\em SIAM Multiscale Model. Simul.}, 20(2):641--656, 2022.
	
	\bibitem{deng2023gradient}
	Y.~Deng, Y.~Hu, and W.~Tang. Optimal estimate of field concentration between multiscale nearly-touching inclusions for 3-D Helmholtz system.
	{\em Comm. Anal. Compu.}, 1(1): 56--71, 2023.
	
    \bibitem{DKLZ23} Y. Deng, L. Kong, H. Liu and L. Zhu, On field concentration between nearly-touching multiscale inclusions in the quasi-static regime, preprint.

	\bibitem{DengliuG} Y.~Deng, H.~Liu, and G.~Uhlmann.
	\newblock On an inverse boundary problem arising in brain imaging.
	\newblock {\em J. Differ. Equ.}, 267(4):2471--2502,2019.
	
	
	
	\bibitem{dong2022optimal}
	H.~Dong, Y.~Li, and Z.~Yang.
	\newblock Optimal gradient estimates of solutions to the insulated conductivity problem in dimension greater than two.
	\newblock {\em  arXiv:2110.11313}, 2021.
	
	\bibitem{DongxuEpllipticSIMA}
	H.~Dong, L.~Xu.
	\newblock Gradient estimates for divergence form elliptic systems arising from composite material. 
	\newblock {\em  SIAM J. Math. Anal.}, 51(3):2444--2478, 2019.
	
		\bibitem{Dongyang}
	H.~Dong, Z.~Yang.
	\newblock Optimal estimates for transmission problems including relative conductivities with different signs.
	\newblock {\em  Adv. Math.}, 428(28):109160, 2023.
	
	\bibitem{DongliARMA2019}
	H.~Dong, H.~Li. 
	\newblock Optimal estimates for the conductivity problem by Green's function method. 
	\newblock {\em Arch. Ration. Mech. Anal.}, 231(3):1427--1453, 2019.
	

	
	\bibitem{RG}
R.~Griesmaier. An asymptotic factorization method for inverse electromagnetic scattering in layered media. {\em SIAM J. Appl. Math.}, 68(5):1378-–1403,2008.
	
	\bibitem{MP}D. Gilbarg, N.S. Trudinger. Elliptic partial differential equations of second order,
	Reprint of the 1998 edition, Classics in Mathematics. Springer, Berlin, 2001.
	
	\bibitem{kang2014characterization}
	H.~Kang, M.~Lim, and K.~Yun.
	\newblock Characterization of the electric field concentration between two
	adjacent spherical perfect conductors.
	\newblock {\em SIAM J. Appl. Math.}, 74(1):125--146, 2014.
	
	\bibitem{KLY15}
	H. Kang, H. Lee, and K. Yun. {Optimal estimates and asymptotics for the stress concentration between closely located stiff inclusions}. {\em Math. Annalen}, 363(3-4):1281--1306, 2015.
	
	
	\bibitem{KLY13} H. Kang, M. Lim, and K. Yun. Asymptotics and computation of the solution to the conductivity equation in the presence of adjacent inclusions with extreme conductivities.
	{\em J. Math Pures Appl}. 99(2): 234--249, 2013.
	
	
	\bibitem{KY19}
	H. Kang, S. Yu. {Quantitative characterization of stress concentration in the presence of closely spaced hard inclusions in two-dimensional linear elasticity}, {\em Arch. Rati. Mech. Anal.}, 232(1):121--196, 2019.
	
	\bibitem{Kell1993}J. B. Keller. Stresses in narrow regions. {\em J. Appl. Mech.}, 60(4):1054--1056, 1993.
	
	\bibitem{Lek10}
	J. Lekner. {Analytical expression for the electric field enhancement between two closely-spaced conducting spheres}. {\em J. Electrostatics}, 68:299--304, 2010.
	
	
	\bibitem{LLZ}
	H.~Li, H.~Liu, and J.~Zou.
	\newblock Minnaert resonances for bubbles in soft elastic materials.
	\newblock {\em SIAM J. Appl. Math.}, 82(1):119--141, 2022.
	
\bibitem{LiXuSIAM} H.~Li, L.~Xu, and P.~Zhang.	
\newblock Stress blowup analysis when a suspending rigid particle approaches the boundary in Stokes flow: 2-dimensional case.
\newblock {\em	SIAM J. Math. Anal.}, 55(5):4493–-4536, 2023.
	
	\bibitem{LHZZSIAM2020} H.~Li, Z.~Zhao.
	 \newblock {Boundary blow-up analysis of gradient estimates for {L}am\'{e}
			systems in the presence of {$m$}-convex hard inclusions}.
	 {\em SIAM J. Math. Anal.}, 52(4):3777--3817, 2020.
		
	
	\bibitem{LN03}
	Y. Li, L. Nirenberg. { Estimates for elliptic systems from composite material}, {\em Commun. Pure Appl. Math.}, 56(7):892--925, 2003.
	
	\bibitem{li2022gradient}
	Y.~Li, Z.~Yang.
	\newblock Gradient estimates of solutions to the insulated conductivity problem in dimension greater than two.
	\newblock {\em Math. Ann.}, 385:1775--1796, 2023.
	
	\bibitem{li2000gradient}
	Y.~Li, M.~Vogelius.
	\newblock Gradient estimates for solutions to divergence form elliptic
	equations with discontinuous coefficients.
	\newblock {\em Arch. Ration. Mech. Anal.}, 153(2):91--151,
	2000.
	
	
	\bibitem{lim2009blow}
	M.~Lim, K.~Yun.
	\newblock Blow-up of electric fields between closely spaced spherical perfect conductors.
	\newblock {\em Commun. Partial Differ. Equ.},
	34(10):1287--1315, 2009.
	
	\bibitem{Mar1996} X. Markenscoff. Stress amplification in vanishingly small geometries.
	 {\em Comput. Mech.}, 19(1):77--83, 1996.
	
	 \bibitem{MBDL} Q. Meng, Z. Bai, H. Diao, Huaian, and H. Liu. \newblock{Effective medium theory for embedded obstacles in elasticity with applications to inverse problems}.
	\newblock {\em SIAM J. Appl. Math.}, 82 (2):720--749, 2022.
	
	
	
		\bibitem{SYHA2018}
	S.~Yu, H.~Ammari.
	\newblock Plasmonic interaction between nanosphere.
	\newblock {\em SIAM Rev.}, 60(2):356--385, 2018.
	
	
	\bibitem{yun2007estimates}
	K.~Yun.
	\newblock Estimates for electric fields blown up between closely adjacent
	conductors with arbitrary shape.
	\newblock {\em SIAM J. Appl. Math.}, 67(3):714--730, 2007.
	
	\bibitem{yun2009optimal}
	K.~Yun.
	\newblock Optimal bound on high stresses occurring between stiff fibers with arbitrary shaped cross-sections.
	\newblock {\em J. Math. Anal. Appl.},
	350(1):306--312, 2009.
	
	\bibitem{Yun2016} K.~Yun.
	\newblock An optimal estimate for electric fields on the shortest line segment between two spherical
	insulators in three dimensions.
	\newblock {\em J. Differ. Equ.}, 261(1):148--188, 2016.
	

	
	
	
	
	
\end{thebibliography}

\end{document}